\documentclass[11pt,reqno,a4paper]{amsart}
\usepackage{amsmath,amsthm,verbatim,amscd,amssymb,setspace,enumitem,hyperref}
\usepackage{exscale,color}
\usepackage[colorinlistoftodos,prependcaption,textsize=tiny]{todonotes}
\usepackage{enumitem}
\usepackage{tikz-cd}
\usepackage{mathrsfs}

\renewcommand{\epsilon}{\varepsilon}

\newcommand{\Q}{\mathbb{Q}}

\newcommand{\C}{\mathbb{C}}
\newcommand{\cO}{\mathcal O}
\newcommand{\Spec}{\operatorname{Spec}}
\newcommand{\Vol}{\operatorname{vol}}

\newcommand{\ii}{\ensuremath{\sqrt{-1}}}

\newcommand{\reg}{\mathrm{reg}}
\newcommand{\loc}{\mathrm{loc}}
\newcommand{\pp}{\overline{\partial}}
\renewcommand{\Re}{\operatorname{Re}}
\newcommand{\rf}{\left(M, (g_t)_{t\in [-1,0)}\right)}
\newtheorem{mainthm}{Theorem}

\newcommand{\supp}{\operatorname{supp}}
\newcommand{\lie}{\operatorname{Lie}}

\newcommand{{\vol}}{\rm vol}
\newcommand{\p}{\partial}

\newcommand{\Ric}{\operatorname{Ric}}

\newcommand{\Rm}{\operatorname{Rm}}

\def\tr{\operatorname{tr}}

\def\Div{\operatorname{div}}
\def\Id{\operatorname{Id}}

\newtheoremstyle{fancy}{}{}{\itshape}{}{\textbf\bgroup}{.\egroup}{ }{}
\newtheoremstyle{fancy2}{}{}{\rm}{}{\textbf\bgroup}{.\egroup}{ }{}

\theoremstyle{fancy}
\newtheorem{theorem}{Theorem}[section]
\newtheorem{lemma}[theorem]{Lemma}
\newtheorem{corollary}[theorem]{Corollary}

\newtheorem{prop}[theorem]{Proposition}

\theoremstyle{fancy2}
\newtheorem{definition}[theorem]{Definition}

\newtheorem{remark}[theorem]{Remark}

\usepackage{cleveref}

\crefname{maintheorem}{Theorem}{Theorems}
\Crefname{maintheorem}{Theorem}{Theorems}

\usepackage[a4paper,margin=2.5cm]{geometry}
\setlist{leftmargin=*}

\numberwithin{equation}{section}

\begin{document}
\title{Singular K\"ahler--Ricci shrinkers and polarized Fano fibrations}
\date{\today}

\author{Max Hallgren}
\address{Department of Mathematics, Rutgers University, Piscataway, NJ 08854, USA}
\email{mh1564@scarletmail.rutgers.edu}

\author{Junsheng Zhang}
\address{Courant Institute of Mathematical Sciences\\
  New York University, 251 Mercer St\\
  New York, NY 10012\\}
\email{jz7561@nyu.edu}
\date{\today}

\begin{abstract}
We prove that any singular K\"ahler--Ricci shrinker $X$ arising as a noncollapsed limit of K\"ahler--Ricci flows admits a natural structure of a polarized Fano fibration.
We also show that it is simply connected, has unique tangent cones at every point, and is a smooth orbifold outside a subset of complex codimension at least three. As an application, we prove a new long-time pseudolocality theorem for almost-selfsimilar K\"ahler--Ricci flows.
\end{abstract}
\maketitle
\section{Introduction}

Many compactness results in K\"ahler geometry produce metric-space limits whose
regular sets carry natural complex-analytic structures. A fundamental problem is
to determine whether these structures extend globally across the singular set
and, moreover, whether the resulting spaces are algebraic. A positive answer
makes available the powerful tools of complex and algebraic geometry in the
study of such limits.

In this paper, we establish the global algebraicity of singular
K\"ahler--Ricci shrinkers, namely, shrinking gradient K\"ahler--Ricci solitons. More
precisely, we show that every singular K\"ahler--Ricci shrinker admits a
polarized Fano fibration structure, extending to the singular setting the
structure introduced in \cite{SunZhang}.

The problem of promoting metric limits to complex-analytic spaces has been most thoroughly developed for Gromov--Hausdorff limits of polarized noncollapsed K\"ahler manifolds. Beyond the case of K\"ahler-Einstein limits in complex dimension two \cite{anderson1989r,BKN,Tian90}, where global analyticity was shown using a removable singularities theorem, a major breakthrough was achieved by Donaldson–Sun \cite{DS1,DS2}. They showed that any pointed Gromov--Hausdorff limit of noncollapsed polarized K\"ahler manifolds with bounded Ricci curvature is a complex-analytic variety with log terminal singularities and proved the uniqueness of metric tangent cones at every point by characterizing such cones algebraically; see also \cite{HS,SWZ2025}. Some of these results were extended in \cite{CDSII,CDS3,LiuSz1,LiuSz2}, where the assumption of an upper bound on the Ricci curvature was removed, and various geometric consequences were derived. 

In the setting of normalized K\"ahler--Ricci flow on a Fano manifold, analogous results for Gromov-Hausdorff limits were established in \cite{TZ16,CW1,CW2}. 
Bamler’s compactness results for Ricci flows \cite{bamler1,bamler2,bamler3} make it possible to consider more general limits of noncollapsed K\"ahler–Ricci flows. Making use of this, certain Gromov-Hausdorff limits of polarized K\"ahler–Ricci flows based at Ricci vertices were shown to be normal analytic varieties \cite{JST23,HJST}.

 Two compactness frameworks are relevant here. Bamler's compactness theory
for noncollapsed Ricci flows is formulated in terms of $\mathbb F$-convergence
of flows of metric-measure spaces, with reference measures typically given by
conjugate heat-kernel measures \cite{bamler2,bamler3}. On the other hand,
Li--Li--Wang developed a compactness theory adapted to sequences of
noncollapsed Ricci shrinkers: after passing to a subsequence, one obtains
pointed Gromov--Hausdorff convergence together with smooth Cheeger--Gromov
convergence on the regular locus
\cite[Theorem~1.1 and Definition~7.1]{LLW21}.
The latter framework can in fact be incorporated into the former. By
\cite{LW24}, the pointed Ricci flows induced by the shrinkers, based at
minimal points of their soliton potentials, also converge in the
$\mathbb F$-sense, and the time slices of the resulting limit flow are
isometric to the corresponding pointed Gromov--Hausdorff limits of
Li--Li--Wang. See Section~\ref{sec:prelim} for a detailed discussion.

In the
smooth setting, it is proved that every complete K\"ahler--Ricci shrinker
naturally defines a polarized Fano fibration
\cite[Theorem~1.1]{SunZhang}.  Our main result extends this theorem to
singular  K\"ahler--Ricci shrinkers arising in both settings.

\begin{mainthm}\label{thm:pffibration}
Let $X$ be a singular K\"ahler--Ricci shrinker arising in either
setting \ref{assume:KRF} or setting \ref{assume:KRS} discussed in Section \ref{sec;convergence setting}.  Then $X$ admits a
polarized Fano fibration.
\end{mainthm}
 The proof of Theorem~\ref{thm:pffibration} consists of two main steps. The first
is to establish the complex-analyticity of the space, and the second is to
upgrade this complex-analytic structure to an algebraic one. The first step is
accomplished by the following result.
\begin{mainthm} \label{thm:complexstructure}
Let $\mathcal{O}_X$ denote the sheaf of continuous functions on $X$ whose restrictions to $\mathcal{R}_X$ are holomorphic. Then:  
\begin{enumerate}
    \item \label{thm:complexstructure1} $(X,\mathcal{O}_X)$ is a normal complex-analytic variety with log terminal singularities;
    \item \label{thm:complexstructure2} $\omega_X$ extends to $X$ as a positive $(1,1)$-current with bounded potential. Moreover, it is a K\"ahler current in the sense that for any $p \in X$, there exist a neighborhood $U_p$ of $p$ and a holomorphic embedding $U_p \hookrightarrow \mathbb{C}^N$ such that for some constant $c_p>0$,
\begin{equation}
\omega_X \ge c_p\omega_{\mathbb{C}^N}\big|_{U_p};
\end{equation}
    \item \label{thm:complexstructure3} the metric singular set of $X$ coincides with its complex analytic singular set;    
    \item \label{thm:complexstructure4} at every point of $X$, the tangent cone is unique, and is a K\"ahler Ricci-flat cone on an affine variety with log terminal singularities;
    \item \label{thm:complexstructure5} $\omega_X$ is a smooth orbifold metric on the orbifold locus $X^{\mathrm{orb}}$ of $X$.
\end{enumerate}
\end{mainthm}

% \begin{remark}  In particular, if $n=3$, then $X$ is a smooth orbifold away from a discrete set, and the link of any tangent cone is a smooth Riemannian orbifold \cite[Theorem 5.1]{DS1}. 
% \end{remark}

\begin{remark}
    Since the objects we consider arise as limits of smooth objects, we can obtain estimates from those in the smooth setting, combined with a limiting argument. In particular, estimates for the conjugate heat kernel, which play an essential role in the proof of the above theorem, allow us to derive the gradient estimate for holomorphic pluri-anticanonical sections established in Lemma~\ref{lem:C1estimates}.  There are also related questions concerning singular canonical metrics constructed in the sense of pluripotential theory. We refer to \cite{EGZ,hein-sun,LiTian,GGZ,zhang2025,AIM,FGS2025,GPP} and the references therein for results along this direction.
\end{remark}

Once this analytic structure is established, one can derive further consequences from the torus symmetry of the shrinker. Prior work shows that $J_X\nabla f_X$ generates a complete flow on
$\mathcal{R}_X$ \cite{halljian}, and that this flow induces an isometric action
of a compact torus $\mathbb T$ on $X$ whose elements restrict to biholomorphisms
of $\mathcal{R}_X$ \cite{hallgrenKahlerRicciTangentFlows2023}.
We show that this extends to an action of the complexified torus $\mathbb{T}^{\mathbb{C}}$ on $X$, and that there exists a locally Lipschitz moment map $\mu:X\to \operatorname{Lie}(\mathbb{T})^*$ for the $\mathbb{T}$-action. Then we prove the following.

\begin{mainthm} \label{thm:additionalstructure}
\begin{enumerate}
    \item \label{thm:additional2} $\mu(X)$ is a closed convex  subset of $\operatorname{Lie}(\mathbb{T})^*$.
        \item \label{thm:additional1} $X$ is simply connected.
\end{enumerate}
\end{mainthm}

When $X$ is smooth, \ref{thm:additional2} was proved in \cite[Theorem 4.1]{HNP} and \ref{thm:additional1} was proved in  \cite{esparza2,SunZhang}. The convexity of the moment map image is used to show that one can perturb the soliton vector field to obtain a holomorphic Killing vector field generating an $\mathbb{S}^1$-action with proper Hamiltonian potential; see \cite[Lemma 3.3]{SunZhang}. The simply-connectedness of a singular K\"ahler--Ricci shrinker is a non-compact analogue of the result that a projective variety of Fano type is simply connected \cite{KMM,campana1994,Tak,QZhang,HM07}. 

After we prove Theorem \ref{thm:complexstructure} and \ref{thm:additionalstructure}, we can follow the strategy in \cite{SunZhang} to show the singular K\"ahler--Ricci shrinker admits a polarized Fano fibration. As in
\cite{SunZhang}, the crucial point is to show that the vanishing locus of the
soliton vector field, or equivalently, the fixed-point locus of the compact
torus $\mathbb T$ generated by $\xi$, is compact. To this end, we perform
K\"ahler reduction at regular values and study the birational transformations
that occur when crossing critical values, a phenomenon commonly referred to
as wall crossing.

In the smooth setting considered in \cite{SunZhang}, the coefficients of the difference divisors arising from wall crossing can be computed explicitly using local models of the quotients. The new difficulty here is that the
fixed center of the wall is singular, in which case it appears difficult
to determine these coefficients explicitly. The key observation is that
their precise values are not needed: as long as the difference divisors
have the correct sign, both the propagation of complements and the
preservation of a lower bound for discrepancies across the wall can still
be established.

\

Using the completeness of $\nabla f_X$ and a contradiction-compactness argument, we also obtain the following long-time pseudolocality result near almost-selfsimilar points of the non-collapsed K\"ahler--Ricci flow. For the definitions of $P^*$-parabolic balls and the curvature radius $r_{\Rm}$, we refer to \cite[Sections 9 and 10]{bamler1}; for the definition of self-similar points, see \cite[Section 5]{bamler3}.
\begin{mainthm} \label{thm:pseudolocality} For any $W,\sigma_0>0$ and $D \in (1,\infty)$, there exist $\delta=\delta(D,W,\sigma_0)>0$ and $\sigma =\sigma(D,W,\sigma_0)>0$ such that the following holds. Suppose $(M^n,(g_t)_{t\in I})$ is a compact K\"ahler--Ricci flow and $(x_0,t_0)\in M\times I$, $r>0$ are such that $\mathcal{N}_{x_0,t_0}(r^2)\geq -W$ and $(x_0,t_0)$ is $(\delta,r)$-selfsimilar. Then for any $(x,t)\in P^{\ast}(x_0,t_0;Dr)$ satisfying $t\leq t_0-D^{-2}r^2$ and $r_{\Rm}(x,t) \geq \sigma_0 r$, we have 
$$\inf\{r_{\operatorname{Rm}}(x,t')\mid t'\in (t_0-D^2 r^2, t_0-D^{-2}r^2)\} \geq  \sigma r.$$
\end{mainthm}

\subsection*{Outline of Paper} 

In Section \ref{sec:prelim}, we first review some facts about singular solitons obtained from Bamler's compactness theory. We then recall or establish some estimates and identities necessary for the proof of Theorem \ref{thm:complexstructure}.

In Section \ref{sec:Hormander}, we establish the analytic machinery necessary for proving Theorem \ref{thm:complexstructure} via the general strategy of \cite{DS2}. We establish H\"ormander's \( L^2 \)-estimate directly on the singular space \( X \) as in \cite{LiuSz2,hallgrenKahlerRicciTangentFlows2023}. After constructing appropriate cut-off functions using the heat kernel of the shrinker (Section \ref{sec:cutoff}), the appropriate H\"ormander-type $L^2$ estimates follow from Bochner formulas once we prove a local boundedness result for bundle-valued $(0,1)$-forms which are harmonic near the singular set. This is done in Section \ref{sec:localboundedness}. Some of the main ingredients used to prove this in our setting are as follows:
\begin{enumerate}
    \item We regard a singular K\"ahler--Ricci shrinker as polarized by its anticanonical bundle with Hermitian metric corresponding to the weighted volume form. 
    \item Since we do not have the expected optimal local finiteness of the singular set with respect to the Minkowski codimension-four content, we instead rely on refined Kato inequalities (see \S \ref{sec:Katoinequalities}) to bypass this issue.
    \item We use the shrinker equation to estimate the Ricci curvature in terms of the complex Hessian of \( f \). Various integrations by parts are then combined with Perelman's differential Harnack inequality to handle the terms involving second-order derivatives of \( f \) (see \S \ref{sec:IBPtrick}).
\end{enumerate}
Similar ideas are used to prove $C^0$ and $C^1$ estimates for holomorphic pluri-anticanonical sections, which is done in Section \ref{sec:C0C1}. 

In Section \ref{sec:Thm1Pf}, we carry out the proof of Theorem \ref{thm:complexstructure}. Given the results of Section \ref{sec:Hormander}, our proof that the soliton is a normal complex variety closely follows that of \cite{LiuSz2}, and we record the necessary modifications in Section \ref{sec:normal}. In Section \ref{sec:orbifold}, we prove the klt property using the methods of \cite{HS}, and then leverage the K\"ahler current property with elliptic estimates to obtain orbifold regularity on the singular soliton metric on the orbifold locus of the variety.

In Section \ref{sec:equivariant}, we show that the action of the compact torus associated to the soliton vector field extends to a holomorphic action of its complexification. As a consequence, we deduce that the standard time-like vector field $\partial_{\mathfrak{t}}$ on the shrinker is complete on times $(-\infty,0)$. A straightforward contradiction-compactness argument then yields the long-time pseudolocality result, Theorem \ref{thm:pseudolocality}.

In Section \ref{sec:moment}, we establish Theorem \ref{thm:additionalstructure}. The proof relies on the quotient theory developed in \cite{H91,HL94,HH96,HHL94} for a holomorphic action of a complex reductive group on a possibly noncompact K\"ahler space. A new difficulty in our setting is that the K\"ahler–Ricci shrinker metric $\omega_X$ is not a smooth K\"ahler form, in the sense that it need not arise as the restriction of a smooth K\"ahler form under a local embedding. Nevertheless, by combining the local boundedness of holomorphic vector fields with respect to $\omega_X$ (cf. Lemma~\ref{lem-vector field bounded}) and the local lower bound for $\omega_X$ established in Theorem \ref{thm:complexstructure}-\ref{thm:complexstructure2}, we recover sufficient analytic control to carry out the arguments in \cite{HL94,HHL94}. We then apply the methods of \cite{SunZhang} and the results of \cite{greb-projec,BGLM2024}, together with additional arguments needed to handle the singularities of the quotient, thereby completing the proof of Theorem \ref{thm:additionalstructure}.

In Section~\ref{sec:fano-fibration}, we prove
Theorem~\ref{thm:pffibration}. We first record the analytic inputs and
construct the reductions at regular values. We then analyze wall crossing
on the singular total space and prove the effectivity of character jump and then establish uniform bounds for quotient
boundaries, stabilizers, Cartier indices, and critical values. Finally, we use an equivariant anticanonical
embedding to construct the polarized Fano fibration.

\subsection*{Conventions and Notation}
\begin{itemize}
    \item Throughout, we let $\nabla$ denote both the Levi--Civita connection and the Chern connection on a Hermitian holomorphic vector bundle $E$, as well as the induced connections on $\mathcal A^{p,q}(E)$. We write $\nabla^{1,0}$ and $\nabla^{0,1}$ for the $(1,0)$- and $(0,1)$-components of $\nabla$, respectively, and denote by $\partial$ and $\bar\partial$ the anti-symmetrizations of $\nabla^{1,0}$ and $\nabla^{0,1}$.
    \item On a K\"ahler manifold, we let $\langle \cdot ,\cdot \rangle$ denote the sesquilinear extension of $g$ to $T_{\mathbb{C}}M$. We let $\Delta = \Delta_{\pp }=g^{\overline{\jmath}i}\partial_i \partial_{\overline{\jmath}}$ (in local normal coordinates) denote half of the Riemannian Laplacian and $$R=\tr_{\omega}(\Ric(\omega))$$ denote half the Riemannian scalar curvature.
    \item   For $f\in C^{\infty}(M,\mathbb R)$, we let $\nabla f = 2\mathrm{Re}(g^{\overline{\jmath}i}\partial_if\frac{\partial}{\partial \overline{z}_j})$ denote the Riemannian gradient. It decomposes as
\begin{equation}
    \nabla f = \nabla^{1,0} f + \nabla^{0,1} f,
\end{equation}
where $\nabla^{1,0} f$ and $\nabla^{0,1} f$ are sections of  $T^{1,0}M$ and $T^{0,1}M$ respectively.
    \item We let $\Delta_{f}=\Delta-\nabla_{\frac{1}{2}\nabla f}$ denote the drifted Laplacian. When applied to a smooth real valued function \( u \), it is given by 
\begin{equation}
\begin{aligned}
        \Delta_f u=&\Delta u-\frac{1}{2}\left<\nabla f,\nabla u\right>\\
    &=g^{\overline{\jmath}i}\partial_i \partial_{\overline{\jmath}}u-\mathrm{Re}(g^{\overline{\jmath}i}\partial_i f \partial_{\overline{\jmath}}u).
\end{aligned}
\end{equation}
\end{itemize}  

\subsection*{Acknowledgements}
The authors would like to thank Ronan Conlon, Yu Li, G\'abor Sz\'ekelyhidi, Jian Song, Song Sun, and Xiaowei Wang for useful conversations and suggestions. The first author was supported in part by the
National Science Foundation under Grant No. DMS-2202980.

\subsection*{Declaration on the use of AI}
AI tools were used in developing  the argument in Section \ref{sec:fano-fibration}, including suggestions
leading to Lemmas~\ref{lem:divisorial-character-jump}
and~\ref{lem:complements}. The authors take
full responsibility for the content and correctness of the paper.

\section{Preliminaries} \label{sec:prelim}

\subsection{Finite-time singularities of K\"ahler--Ricci flow}\label{sec--finite time sing of KRF}

Let $\rf$ be a compact Ricci flow developing a finite-time singularity at $t=0$. Note that the metric space at the singular time 0 is not defined a priori. Following \cite{bamler2,bamler3}, we may view a point at the singular time $0$ as a conjugate heat flow whose variance tends to zero as $t \to 0$.

\begin{definition}[\cite{bamler3}]\label{conjugate kernel based at sing time}
    A conjugate heat flow based at the singular time is a conjugate heat flow $\nu=\left(\nu_t\right)_{t \in[-1, 0)}$ on $\rf$ with the property that
\begin{equation}
\lim _{t \nearrow 0} \operatorname{Var}\left(\nu_t\right)=0.    
\end{equation}

\end{definition}

\begin{remark}
    The set of conjugate heat flows based at the singular time is always non-empty. In fact, given any $x_0 \in M$ and any sequence of times $T_i \nearrow 0$, we can pass to a subsequence to obtain convergence in $C_{\operatorname{loc}}^{\infty}(M\times [-1,0))$ to a positive solution $$K(x_0,0;\cdot,\,\cdot):= \lim_{i\to \infty} K(x_0,\,T_i;\cdot,\,\cdot):M \times [-1,0)\to\mathbb{R}$$ of the conjugate heat equation \cite[Lemma 2.2]{MM2015}. Let $\nu=(\nu_t)_{t\in [-1,0)}$ denote the measure obtained via this limit. Then by the $C^{\infty}_{loc}$ convergence and the estimates in \cite[Section 3]{bamler1}, we know that for any  $t\in [-1,0)$, we have
$$\operatorname{Var}(\nu_t)=\int_{M} \int_{M} d_t^2(x,y)\,d\nu_{t}(x)\, d\nu_{t}(y) \leq H_{n}|t|,$$ where $H_n:= \frac{(n-1)\pi^2}{2}+4$ and $n$ is the real dimension of $M$.  In particular, we have 
   $ \lim_{t\nearrow 0}\mathrm{Var}(\nu_{t})=0.$
\end{remark}

By \cite[Theorem 2.37]{bamler3}, for any sequence $\tau_i \searrow 0$, if we set $g_{i,t}:=\tau_i^{-1}g_{\tau_i t}$ and $\nu_{i;t} := \nu_{\tau_i t}$, then we can pass to a subsequence to obtain a metric soliton $(\mathcal{X},(\mu_t)_{t\in (-\infty,0)})$ such that 
\begin{equation} \label{eq:Fconverge} (M, (g_{i,t})_{t\in [-\tau_i^{-1},0)},(\nu_{i;t})_{t\in [-\tau_i^{-1},0)})\xrightarrow[i\to \infty]{\mathbb{F}} (\mathcal{X},(\mu_t)_{t\in (-\infty,0)}) \end{equation}
uniformly on compact time intervals. Any $\mathbb{F}$-sequential limit of the rescaled flows $(M,\tau_i^{-1}g_{\tau_i t},\nu_{\tau_i t})$ along a sequence $\tau_i \to 0$ is called a tangent flow at $\nu$.

\subsection{Two convergence frameworks}\label{sec;convergence setting} The proof of Theorem~\ref{thm:pffibration} works as long as the
analytic and equivariant conclusions of
Theorems~\ref{thm:complexstructure} and~\ref{thm:additionalstructure},
together with the weighted and cutoff estimates used in
Section~\ref{sec:fano-fibration}, are available. Thus, the same proof
applies to any singular K\"ahler--Ricci shrinker satisfying this package
of assumptions. In particular, we consider the following two convergence
settings.

\begin{enumerate}[label=\textup{(\Roman*)}]
\item \label{assume:KRF} \emph{Noncollapsed $\mathbb F$-limits of compact K\"ahler--Ricci flows.}
Let $T_i\to\infty$, and let
\[
 \bigl(M_i,J_i,(g_{i,t})_{t\in[-T_i,0]},
       (\nu_{x_i,0;t})_{t\in[-T_i,0]}\bigr)
\]
be compact K\"ahler--Ricci flows equipped with the conjugate heat-kernel
measures based at $(x_i,0)$.  Assume that, for some $W<\infty$,
\[
       \liminf_{i\to\infty}\mathcal N^{g_i}_{x_i,0}(1)\ge -W
\]
and that, after passing to a subsequence,
\[
 \bigl(M_i,(g_{i,t}),(\nu_{x_i,0;t})\bigr)
 \xrightarrow[i\to\infty]{\mathbb F}
 \bigl(\mathcal X,(\mu_t)_{t\in(-\infty,0)}\bigr)
\]
uniformly on compact time intervals, where the limit is a metric soliton.
Then $\mathcal X$ is modeled on a singular K\"ahler--Ricci shrinker
$(X,d_X)$ with singular set of Minkowski codimension at least four.  Its
regular Ricci-flow spacetime
$(\mathcal R,g,\mathfrak t,\partial_{\mathfrak t},\mathcal J)$ is the
spacetime induced by the smooth, generally incomplete, shrinker
$(\mathcal R_X,g_X,\omega_X,J_X,f_X)$, and
$d_X|_{\mathcal R_X}$ is the length metric of $g_X$
\cite{bamler2,bamler3}. Note that this case includes the tangent flows at finite-time singularities
discussed in the previous section.

\item \label{assume:KRS} \emph{Pointed Gromov-Hausdorff limits of
noncollapsed K\"ahler--Ricci shrinkers.}
Let $n\ge2$ and let
\begin{equation*} (M_i,p_i,g_i,J_i,f_i) \end{equation*}
be complete smooth K\"ahler--Ricci shrinkers, normalized by
\begin{equation*}
 \Ric(g_i)+(\nabla^{g_i})^2f_i=g_i,
 \qquad R_{g_i}+|\nabla f_i|_{g_i}^2=f_i,
\end{equation*}
where $p_i$ is a minimum point of $f_i$.  Suppose that the 
shrinker entropy satisfies
\[
       \boldsymbol{\mu}(M_i,g_i)\ge-A
\]
for some $A<\infty$.  By
\cite[Theorem~1.1]{LLW21}, after passing to a subsequence,
\[
 (M_i,p_i,d_{g_i},f_i)
 \xrightarrow[i\to\infty]{\mathrm{pGH}}
 (X_\infty,p_\infty,d_\infty,f_\infty),
\]
where $X_\infty=\mathcal R_\infty\sqcup\mathcal S_\infty$,
$\mathcal S_\infty$ is closed with
$\dim_{\mathcal M}\mathcal S_\infty\le2n-4$, and
$(\mathcal R_\infty,g_\infty,f_\infty)$ is a smooth K\"ahler--Ricci shrinker.  The convergence is pointed-$\widehat{C}^{\infty}$-Cheeger--Gromov
in the following precise sense \cite[Definition~7.1]{LLW21}.  For every $L>0$ and every compact set
$K\Subset\mathcal R_\infty\cap B_{d_\infty}(p_\infty,L)$, after adjusting
$\epsilon_{L,i}\downarrow0$, the $\epsilon_{L,i}$-isometry 
\[
 \psi_{L,i}\colon B_{d_\infty}(p_\infty,L)
 \longrightarrow B_{d_{g_i}}(p_i,L+\epsilon_{L,i})
\]
may be chosen so that
\[
       \psi_{L,i}|_K=\varphi_{K,L,i},
\]
where
\(
\varphi_{K,L,i}\colon K\to\varphi_{K,L,i}(K)\subset M_i
\)
is a diffeomorphism and
\[
 \varphi_{K,L,i}^*g_i\longrightarrow g_\infty,
 \quad
 \varphi_{K,L,i}^*f_i\longrightarrow f_\infty, \quad \varphi_{K,L,i}^*J_i\longrightarrow J_\infty
\]
in $C^\infty(K)$.  By the work in \cite{LW24}, one can also consider the \(\mathbb F\)-convergence of pointed Ricci flows induced by Ricci shrinkers and based at minimal points of their soliton potentials. The time slices of the resulting limit flow $\mathcal X$ are isometric to the corresponding pointed Gromov--Hausdorff limits defined above.
\end{enumerate}

\subsection{Singular K\"ahler--Ricci shrinkers}\label{Singular Kahler-Ricci shrinkers}

Let $(X,d_X)$ be a singular K\"ahler--Ricci shrinker  arising as a
limit in one of the setting considered in Section \ref{sec;convergence setting}. It is known by    \cite{bamler2,bamler3} and \cite{LLW21,LW24}, that its regular part admit a (incomplete) K\"ahler--Ricci shrinker structure $$(\mathcal{R}_X,g_X,\omega_X,J_X,f_X)$$ and singular part has Minkowski codimension at least 4.  Moreover, the restriction of $d_X$ to $\mathcal R_X$ is equal to the length metric induced by the smooth Riemannian metric $g_X$. In the following, we will denote $(\mathcal{X},\mu_t)$ denote the metric soliton induced by $(X,d_X)$ and then we know that the K\"ahler--Ricci flow spacetime structure $(\mathcal{R},g,\mathfrak{t},\partial_{\mathfrak{t}},\mathcal{J})$ on the regular part of $\mathcal{X}$ can be identified with the Ricci flow spacetime induced by the regular part $(\mathcal{R}_X,g_X,\omega_X,J_X,f_X)$ of $(X,d_X)$.

Note that $(\mathcal{R}_X,g_X)$ is isometric to $(\mathcal{R}_{-1}=\mathcal{R}\cap \mathfrak{t}^{-1}(-1),g_{-1})$ by a diffeomorphism identifying $d\mu_{-1}$ with $(2\pi)^{-n}e^{-f_X}\frac{\omega_X^n}{n!}$, where we normalize $f_X$ such that 
\begin{equation}
    \frac{1}{(2\pi)^n}\int_{\mathcal{R}_X}e^{-f_X}\frac{\omega_X^n}{n!}=1.
\end{equation}
In the following, we use $\mathcal{X}_{-1}$ and $X$, as well as $\mathcal{R}_{-1}$ and $\mathcal{R}_X$, interchangeably, and omit the dependence on $X$ whenever there is no risk of confusion.
By the following identities, $\nabla f$ is locally bounded, and hence $f$ is locally Lipschitz. 
Moreover, $f$ is proper \cite{cao-zhou,CMZ2024}. 
We therefore can fix a point $p \in X$ such that
\begin{equation}
    f(p) = \min_X f.
\end{equation}

We now recall some important identities that hold on the regular part $\mathcal{R}$ of $\mathcal{X}$:
\begin{equation}\label{eq--shrinker equation}
    \operatorname{Ric}+\nabla^2 f= \frac{1}{\tau}g, \quad \partial_t f=|\nabla^{1,0} f|^2=\frac{1}{2}|\nabla f|^2,
\end{equation}
\begin{equation}\label{eq-scalar curvature and laplacian}
     R+\Delta f=\frac{n}{ \tau}, \quad R>0 \text{ or } \Ric\equiv 0,
\end{equation}
\begin{equation}\label{eq--relation between nabla f and f}
     -\tau\left(R+|\nabla^{1,0} f|^2\right)+f \equiv  W, 
\end{equation}
where the constant $W$ is called the shrinker entropy, which coincides with Perelman's $\nu$-functional \cite{LW20,CMZ2024}. Moreover, the Nash entropy of $\mu_t$ at every scale equals $W$, i.e., for all $\tau>0$,
\begin{equation}
    \mathcal{N}_{(\mu_t)}(\tau)=W,
\end{equation}
and for any point $x\in \mathcal X_{-1}$ and $\tau>0$, we have \cite{bamler3,fangli2025}
\begin{equation}
    \mathcal{N}_x(\tau)\geq W.
\end{equation}

Let $h$ be the Hermitian metric on $K_{\mathcal{R}_{-1}}^{-1}$ induced by the volume form $(2\pi)^{-n}e^{-f} \frac{\omega^{n}}{n!}$. For $\ell\in \mathbb N$, it induces a Hermitian metric on $K_{\mathcal{R}_{-1}}^{-\ell}$. Let $\overline{\partial}^{\ast}_f$ denote the formal adjoint of $\overline{\partial}$ with respect to the inner product
$$\left<\alpha,\beta\right>_f:=\frac{1}{(2\pi)^n} \int_{\mathcal{R}_{-1}} \langle \alpha ,\beta\rangle_{h^\ell} e^{-f}\frac{1}{n!}\omega^n.$$  Note that $\overline{\partial}^{\ast}_f$ coincides with the usual formal adjoint of $\overline{\partial}$ when using the Hermitian metric $h^\ell e^{-f}$ on $K_{\mathcal{R}_{-1}}^{-\ell}$. Similarly for any $t\in (-\infty,0)$, let  $h_{t}$ be the Hermitian metric on $K_{\mathcal{R}_{t}}^{-1}$ corresponding
to the volume form $(2\pi|t|)^{-n}e^{-f_{t}}\frac{\omega_{t}^{n}}{n!}$.

Suppose $\eta_{-1}\in \mathcal{A}^{0,1}(\mathcal{R}_{-1},K_{\mathcal{R}_{-1}}^{-\ell})$ with $\pp \eta_{-1}=0$ and $\pp^{\ast}_f\eta_{-1}=0$. We extend $\eta_{-1}$ to a tensor on $\mathcal{R}$ by demanding that 
\begin{equation}\label{extension of eta}
\mathcal{L}_{\partial_{\mathfrak{t}}-\frac{1}{2}\nabla f}\eta=0.
\end{equation} Equivalently, we obtain a tensor $\eta$ on $\mathcal{R}_{-1}\times (-\infty,0)$ by 
$\eta_t=\phi_t^*\eta,$ where
\begin{equation}\label{eq-self-diff}
    \dot \phi_t=\frac{1}{-2t}\nabla^{g_{-1}} f \quad \text{ and } \quad \phi_{-1}=\Id.
\end{equation}Here, $\phi_t^*\eta$ is defined by pulling back the form part via $\phi_t$ and pushing forward the anticanonical section part via $(\phi_t)^{-1}$; since $\phi_t^*$ and $(\phi_t^{-1})_*$ coincide on functions, this definition makes $\phi_t^*$ well-defined on $\mathcal{A}^{0,1}(\mathcal{R}_{-1},K_{\mathcal{R}_{-1}}^{-\ell})$. We have the following:
\begin{enumerate}
    \item If $\overline{\partial}\eta_{-1}=0$, then because $\mathcal{L}_{\partial_t-\frac{1}{2}\nabla f}J=0$, we have $\overline{\partial} \eta_t=0$ for all $t\in (-\infty,0)$;
    \item If $\pp^{\ast}_f \eta_{-1}=0$, then since $\mathcal{L}_{\partial_t-\frac{1}{2}\nabla f}J=0$ and $\mathcal{L}_{\partial_t-\frac{1}{2}\nabla f} f=0$, we have $\pp^{\ast}_{f_t}\eta_t=0$ for all $t\in (-\infty,0)$.
\end{enumerate}

\subsection{Bochner formulas}

We recall some Bochner--Weitzenb\"ock type formulas that will be used
later.
\begin{lemma}\label{lem-bochner formula}
Let $(L,h)$ be a holomorphic Hermitian line bundle on a K\"ahler manifold $M$. Then for 
 $f\in C^{\infty}(M)$, and $\eta\in\mathcal{A}^{0,1}(M,L)$ 
we have 
\begin{equation} 
    \begin{aligned}
\Delta_{\pp}\eta:=\left(\overline{\partial}^{\ast}\overline{\partial}+\overline{\partial}\overline{\partial}^{\ast}\right)\eta &=\left(\partial^{\ast}\partial+\partial\partial^{\ast}\right)\eta+\Theta(\eta)-\tr_{\omega}(\Theta)\eta\\
&=(\nabla^{0,1})^{\ast}\nabla^{0,1}\eta+(\Ric+\Theta)(\eta),
    \end{aligned}
\end{equation} and
\begin{equation}\label{eq--f-twisted laplacian}
\Delta_{\pp,f}\eta:=\left(\overline{\partial}_{f}^{\ast}\overline{\partial}+\overline{\partial}\overline{\partial}_{f}^{\ast}\right)\eta= \Delta_{\pp}\eta+(\mathcal{L}_{\nabla^{0,1}f}\eta)^{0,1}.
\end{equation}
\end{lemma}

\begin{proof}The first two identities are standard. We give just the proof of the last identity.
Letting $(\nabla^{0,1})_{f}^{\ast}$ denote the formal adjoint of
$\nabla^{0,1}$ with respect to the $f$-weighted measure, we have
\begin{equation} \label{eq:differenceonlyinbochner}
    \begin{aligned}
        (\Delta_{\pp,f}\eta)_{\overline{\jmath}}= & ((\nabla^{0,1})_{f}^{\ast}\nabla^{0,1}\eta)_{\overline{\jmath}}+(R_{i\overline{\jmath}}+\Theta_{i\overline{\jmath}}+\nabla_{i}\nabla_{\overline{\jmath}}f)g^{\overline{l}i}\eta_{\overline{l}}\\
&= ((\nabla^{0,1})^{\ast}\nabla^{0,1}\eta)_{\overline{\jmath}}+(\nabla^{0,1}f\lrcorner \nabla^{0,1}\eta)_{\bar j}+(R_{i\overline{\jmath}}+\Theta_{i\overline{\jmath}}+\nabla_{i}\nabla_{\overline{\jmath}}f)g^{\overline{l}i}\eta_{\overline{l}}\\
&=(\Delta_{\pp}\eta)_{\bar j}+g^{\bar l i}\nabla_{i}f\nabla_{\bar l}\eta_{\bar j}+\nabla_{i}\nabla_{\overline{\jmath}}f g^{\overline{l}i}\eta_{\overline{l}}.
    \end{aligned}
\end{equation}
Then we compute (where now $\nabla^{0,1}f$ denotes the $(0,1)$-part
of the gradient of $f$)
\begin{equation}\label{eq--difference between lie-derivative and connection}
    \begin{aligned}
        (\nabla^{0,1}f\lrcorner \nabla^{0,1}\eta)_{\bar j}= & \nabla^{0,1}f\left(\eta\left(\frac{\partial}{\partial\overline{z}_{j}}\right)\right)-\eta\left(\nabla_{\nabla^{0,1}f}\frac{\partial}{\partial\overline{z}_{j}}\right)\\
= & (\mathcal{L}_{\nabla^{0,1}f}\eta)\left(\frac{\partial}{\partial\overline{z}_{j}}\right)-\eta\left(\nabla_{\frac{\partial}{\partial\overline{z}_{j}}}\nabla^{0,1}f\right)\\
= & (\mathcal{L}_{\nabla^{0,1}f}\eta)_{\overline{\jmath}}-g^{\overline{l}i}\eta_{\overline{l}}\nabla_{i}\nabla_{\overline{\jmath}}f,
    \end{aligned}
\end{equation}
where $\mathcal{L}_{\nabla^{0,1}f}\eta$ is defined in Remark \ref{rem:defofliederivative}. Then the desired identity \eqref{eq--f-twisted laplacian} follows by combining \eqref{eq:differenceonlyinbochner} and \eqref{eq--difference between lie-derivative and connection}.
\end{proof}

\begin{remark} \label{rem:defofliederivative}
In the lemma above, we use the following convention about the Lie derivatives. If $\eta\in\mathcal{A}^{0,1}(M,L)$
is written on some open subset $U\subseteq M$ as $\eta'\otimes\sigma$
for some $\eta'\in\mathcal{A}^{0,1}(U)$ and $\sigma\in H^{0}(U,L)$,
then for
any $V\in C^{\infty}(U,T^{0,1}M)$,
\begin{equation}\label{eq--temporary for lie derivative}
\mathcal{L}_{V}\eta:=(\mathcal{L}_{V}\eta')\otimes\sigma.
\end{equation}  Note that since $\sigma$ is holomorphic and $V$ is of type $(0,1)$, $\mathcal{L}_V\eta$ is well-defined, i.e., independent of the choice of decomposition $\eta=\eta'\otimes \sigma.$

For the main purposes of this paper, we take $L$ to be the anticanonical or pluri-anticanonical bundle. Let $V$ be a real-holomorphic vector field on $M$, and denote by $(\phi_s^V)_{s\in(-\epsilon,\epsilon)}$ its flow.
Let $U \subseteq M$ be an open subset and suppose $\sigma$ is a local section of $K_M^{-\ell}$ for some $\ell\geq 1$. Then  we define
\begin{equation}
\mathcal{L}_V \sigma := \left.\frac{d}{ds}\right|_{s=0} (\phi_{-s}^V)_* \sigma .
\end{equation}
This definition extends naturally, via the Leibniz rule and complex linearity, to the derivative $K_M^{-\ell}$-valued $(p,q)$-forms with respect to $V,V^{1,0},V^{0,1}$. We note the following.
\begin{enumerate}
    \item If $\sigma$ is a holomorphic section of $K_M^{-\ell}$, then for any real-holomorphic vector field $V$, $\mathcal{L}_{V^{0,1}}\sigma=0$. Consequently, for a $K_M^{-\ell}$-valued $(0,1)$-form $\eta$, the definition of $\mathcal{L}_{V^{0,1}}\eta$ agrees with \eqref{eq--temporary for lie derivative}.
    
    \item If $V$ is a real holomorphic vector field, then for any $K_M^{-\ell}$-valued $(0,1)$-form $\eta$, the Lie derivative $\mathcal{L}_{V^{0,1}}\eta$ remains a $K_M^{-\ell}$-valued $(0,1)$-form.
\end{enumerate}
\end{remark}

\begin{lemma}\label{lem--bochner for gradient}
    	For a holomorphic section $s$ of a Hermitian holomorphic line bundle $(L,h)$ on $(X^n,\omega)$ with $\Theta_h=\lambda\omega$ for some $\lambda>0$, we have
\begin{equation}
	\label{eq:bochnerlemma1} \Delta |s|^2=|\nabla s|^2-n\lambda|s|^2,
\end{equation}and 
\begin{equation} \label{eq:bochnerlemma2}
\begin{aligned}
	\Delta|\nabla s|^2=&|\nabla^{1,0} \nabla s|^2+\Ric(\nabla s,\nabla s)
	-\lambda (n+2)|\nabla s|^2+n\lambda^2|s|^2.
\end{aligned}
\end{equation}
\end{lemma}

\begin{proof}
	Locally, choose a holomorphic trivializing section $e$ such that at the given point $\nabla e=0$. Suppose $|e|_h=e^{-\varphi}$ and writing $s=fe$, where $f$ is a local holomorphic function, then in local normal coordinates, we have 
	\begin{equation*}
	\begin{aligned}
		\Delta|s|^2=&g^{i\bar j}\partial_i\partial_{\bar j}(|f|^2e^{-\varphi})=g^{i\bar j}\partial_i s\overline{\partial_j s}e^{-\varphi}-g^{i\bar j}|f|^2e^{-\varphi}\varphi_{i\bar j}\\
		&=|\nabla s|^2-\tr_{\omega}(\Theta_h)|s|^2.
	\end{aligned}
	\end{equation*}
For a smooth section $s$ of $L$, we have
\begin{equation}\label{commute 2derivative for sections}
	\nabla_j\nabla_{\bar k}s-\nabla_{\bar k}\nabla_j s=\Theta_{j\bar k}s.
\end{equation} 
If $s$ is holomorphic, then we know $\nabla s=\partial s$, hence
\begin{equation*}
\begin{aligned}
    \nabla_{\bar j}\nabla_{j}\nabla_k s&=\nabla_j\nabla_{\bar j}\nabla_k s-R_{kj\bar j}^p\nabla_p s-\Theta_{j\bar j}\nabla_k s\\
    &=-\nabla_j(\Theta_{ k\bar j}s)-R_{kj\bar j}^p\nabla_p s-\Theta_{j\bar j}\nabla_k s.
\end{aligned}
\end{equation*}
Then we can compute 
\begin{equation}
\begin{aligned}
    	\Delta|\nabla s|^2=&|\nabla^{1,0} \nabla s|^2+|\nabla^{0,1}
        \nabla s|^2\\
        &-\nabla_j(\Theta_{k\bar j}s)\overline{\nabla_ks}-\nabla_k s\overline{\nabla_j(\Theta_{k \bar j}s)}+\Ric_{j\bar k}\nabla_j s\overline{\nabla_{k}s}-\Theta_{j\bar j}|\nabla s|^2\\
        =&|\nabla^{1,0} \nabla s|^2+\Ric(\nabla s,\nabla s)
	-\lambda (n+2)|\nabla s|^2+n\lambda^2|s|^2,
\end{aligned}
\end{equation}
where in the last equality we have used that $\Theta_h=\lambda \omega$  and $|\nabla^{0,1}
        \nabla s|^2=|\Theta_h|^2|s|^2=n\lambda^2|s|^2$.
\end{proof}

\subsection{Estimates for conjugate heat kernels} We recall the local gradient estimate for solutions of the conjugate heat equation
\begin{equation}\label{eq--conjugate heat equation}
    \square^*v:=(-\partial_t-\triangle_{g_t}+R)v=0
\end{equation}
proved in \cite{EKNT}. Note that in the following, the assumption on the curvature radius implicitly ensures that $B(x_0,t_0,\frac{r}{2})$ is relatively compact in $M$. The result then follows from \cite[Proof of Theorem 10]{EKNT}, where in the notation of that reference we take $k_1 = k_2 = C(n) r^{-2}$, $k_3 = C(n) r^{-2}$, and $\rho = \frac{r}{2}$.

\begin{lemma}[\cite{EKNT}] \label{lem:conjugategradient} Suppose $(M^n,(g_t)_{t\in I})$ is a Ricci flow and $(x_0,t_0) \in M \times I$, $r \leq r_{\operatorname{Rm}}(x_0,t_0)$ satisfy $[t_0,t_0+r^2] \subseteq I$. Set $P:= B(x_0,t_0,\frac{1}{2}r) \times [t_0,t_0+\frac{1}{4}r^2]$, and suppose $v\in C^{\infty}(P)$ is a positive solution of \eqref{eq--conjugate heat equation} satisfying $v\leq A$. Then there exists a dimensional constant $C=C(n)$ such that 
\begin{equation*}
    |\nabla \log v|(x_0,t_0) \leq \frac{C}{r}\left( 1+ \log \left( \frac{A}{v(x_0,t_0)} \right)\right). 
\end{equation*}
\end{lemma}

We will also need Perelman's differential Harnack inequality \cite{Perelman}, which holds for limits of compact Ricci flows.

\begin{theorem}\label{thm-harnack inequality} Suppose $(M_i^n,(g_{i,t})_{t\in [-\tau_i^{-1},0)})$ is a sequence of closed Ricci flows, $x_i \in M_i$, $\nu_{i;t}:= \nu_{x_i,0;t}$, and assume $\liminf_{i \to \infty}\mathcal{N}_{x_i,0}(1) > -\infty$, and that $(\mathcal{X},(\mu_t)_{t\in (-\infty,0)})$ is a metric soliton such that \eqref{eq:Fconverge} holds uniformly on compact time intervals. Suppose $y\in \mathcal{X}$, and write $\tau_y(t):= \mathfrak{t}(y)-t$, $K(y;\cdot)=(4\pi \tau_y)^{-\frac{n}{2}}e^{-f}$ on $\mathcal{R}_{<\mathfrak{t}(y)}$. Then 
\begin{equation*}
    \tau_y(R+2\Delta f-|\nabla f|^2)+f-n\leq 0.
\end{equation*}
\end{theorem}
\begin{proof} Let $K_i$ denote the conjugate heat kernel of $(M_i,(g_{i,t})_{t\in [-\tau_i^{-1},0)})$. By \cite[Theorem 6.45]{bamler2}, there exist $y_i \in M_i$ such that
\begin{equation*}
    (y_i,\mathfrak{t}(y))\xrightarrow[i\to \infty]{\mathfrak{C}} y.
\end{equation*}
For any $t<\mathfrak{t}(y)$, \cite[Theorem 9.31]{bamler2} gives an open exhaustion $(U_i)_{i\in \mathbb{N}}$ of $\mathcal{R}_t$ and open embeddings $\psi_i : U_i \to M_i$ such that 
\begin{equation*}
    \psi_i^{\ast}g_{i,t} \to g_t, \qquad \psi_i^{\ast}K_i(y_i,\mathfrak{t}(y);\cdot,t) \to K(y;\cdot)
\end{equation*}
in $C_{\operatorname{loc}}^{\infty}(\mathcal{R}_t)$. Writing $K(y_i,\mathfrak{t}(y);\cdot,t)=(4\pi\tau_y)^{-\frac{n}{2}}e^{-f_i}$, \cite[Theorem 16.44]{ChowII} gives
\begin{equation*}
\tau_y (R_{g_i}+2\Delta_{g_i} f_i - |\nabla f_i|_{g_i}^2)+f_i -n\leq 0,
\end{equation*}
so the claim follows by pulling this inequality back by $\psi_i$ and taking $i\to \infty$.
\end{proof}

\subsection{An integration by parts trick} \label{sec:IBPtrick} We will use the following argument several times in the paper, so we record it here as an independent lemma. 

Let $(M,g)$ be a smooth Riemannian manifold and suppose we have the following data:
\begin{enumerate}
    \item $u$ is a smooth positive function satisfying, for some constant $A>0$,
    \begin{equation}\label{eq--harnack type}
        |\nabla \log u|^2\, u \le 2\Delta u + A u.
    \end{equation}
    \item $v$ is a smooth function on $M$, and $\psi : M \to [0,1]$ is a smooth function with compact support in $M$.
\end{enumerate}
Then we have the estimate below.

\begin{lemma}\label{lem--integration by part trick}
 There exists a constant $C_0$ such that for any $\epsilon\in (0,1]$,
\begin{equation}
\begin{aligned}
    \int_M v^2 \psi^2 |\nabla \log u|\, u\, dg
    &\le \epsilon \int_M |\nabla v|^2 \psi^2 \, u\, dg + C_0(\epsilon^{-1} + A\epsilon)\int_M v^2 \psi^2 u\, dg \\
    &\quad + C_0\epsilon\int_{\operatorname{supp}(\nabla \psi)} v^2 \big(|\nabla \psi|^2 + \psi^2|\nabla \log u|^2\big) u\, dg .
\end{aligned}
\end{equation}
\end{lemma}

In later applications, the function $u$ will be the conjugate heat kernel. We remark that on the right-hand side the derivative of $v$ appears with a small coefficient, while the derivatives of $\psi$ and $u$ only need to be integrated over $\supp(\nabla \psi)$. In our later applications, we will use the Kato inequality to absorb the term involving $\nabla v$, and then use the standard estimates for the cut-off function to bound the terms involving the derivatives of $\psi$ and $u$.

\begin{proof}
    Using \eqref{eq--harnack type}, integration by parts, and the Cauchy--Schwarz inequality, we first obtain that
    \begin{equation}
    \begin{aligned}
          \int_{M}v^2\psi^2|\nabla\log u|^2u\,dg\leq&4 \int_M|\psi^2v\nabla v+v^2\psi \nabla \psi||\nabla \log u|u\, dg+A\int_Mv^2\psi^2u\,dg\\
    &\leq \int_M \left( \frac12|\nabla \log u|^2 v^2 +8|\nabla v|^2\right)\psi^2 u\,dg \\
    &+2\int_{\supp(\nabla \psi)}v^2(|\nabla \psi|^2+\psi^2|\nabla \log u|^2)u\, dg\\
          &+A\int_Mv^2\psi^2u\, dg.
    \end{aligned}
    \end{equation}
Rearranging terms gives
\begin{equation}\label{eq--cauchy1}
\begin{aligned}
      \int_M v^2 \psi^2 |\nabla \log u|^2 u dg \leq &16\int_M |\nabla v|^2 \psi^2 u dg+4\int_{\supp(\nabla \psi)}v^2(|\nabla \psi|^2+\psi^2|\nabla \log u|^2)u\, dg\\
          &+2A\int_Mv^2\psi^2u\, dg.
\end{aligned}
\end{equation}

    For any $\epsilon>0$, we have that 
    \begin{equation}\label{eq--cauchy2}
         \int_M v^2 \psi^2 |\nabla \log u| u\, dg\leq \epsilon^{-1}\int_Mv^2\psi^2u\, dg+\frac{\epsilon}{2}\int_M v^2\psi^2|\nabla \log u|^2u\, dg.
    \end{equation}
    We then combine \eqref{eq--cauchy1} and \eqref{eq--cauchy2} to conclude the proof.
\end{proof}

\section{H\"ormander $L^2$ estimates on singular solitons}
\label{sec:Hormander}

In this section, we prove the H\"ormander $L^2$ estimates (Proposition \ref{prop:hormander}) for pluri-anticanonical sections on a singular K\"ahler--Ricci shrinker, as well as the $C^0$ and $C^1$ estimates for holomorphic pluri-anticanonical sections (Lemma \ref{lem:C1estimates}). We assume throughout this section that $(X,d)$ is a singular K\"ahler-Ricci shrinker obtained as an $\mathbb{F}$-limit of noncollapsed K\"ahler-Ricci flows as in assumption \ref{assume:KRF} of Section \ref{sec;convergence setting}, with corresponding metric soliton $\mathcal{X}$. 

The main ingredient for H\"ormander $L^2$ estimates is Proposition~\ref{prop:localbound}, which establishes an $L^\infty$ bound for $\bar{\partial}$-closed and $\bar{\partial}^*$-closed $(0,1)$-forms with values in the pluri-anticanonical bundles.

\subsection{Improved Kato inequalities}

\label{sec:Katoinequalities}

In this subsection, we prove two improved Kato inequalities: one for $\bar{\partial}$-closed and co-closed $(0,1)$-forms, and the other for the gradient of a holomorphic section. Both results are known in the literature; for completeness, we include the proofs here. Throughout this paper, we assume that the complex dimension satisfies $n \geq 2$.

\begin{lemma}\label{lem-improved kato} Suppose $(M,\omega)$ is an $n$-dimensional K\"ahler manifold and $(L,h)$ is a holomorphic Hermitian line bundle on $M$. For any $\eta \in \mathcal{A}^{0,1}(M,L)$ satisfying $\overline{\partial}\eta=0$ and $\overline{\partial}^{\ast}\eta=0$, we have
\begin{equation}
    |\nabla|\eta||^2\leq (1-\frac{1}{2n})|\nabla \eta|^2.
\end{equation}
\end{lemma}
\begin{proof}
This follows from the general result in \cite[Theorem 1.1]{DM2024}, but for completeness we provide a direct proof here. Choose holomorphic coordinates centered at a given point $p\in M$, such that $g_{i\overline{\jmath}}(p)=\delta_{ij}$ and $\eta_{\overline{\jmath}}=0$ for all $j\geq 2$ and choose a local holomorphic trivializing section $e$ of $L$ such that $\nabla e|_p=0$.

From $\overline{\partial}^{\ast}\eta=0$, we have $\nabla_1 \eta_{\overline{1}}=-\sum_{j=2}^n \nabla_j \eta_{\overline{\jmath}}$ at $p$, hence by Cauchy--Schwarz, we know that  
\begin{equation*}
     |\nabla_1 \eta_{\overline{1}}|^2\leq (n-1) \sum_{j=2}^n |\nabla_j \eta_{\overline{\jmath}}|^2.
\end{equation*}
Equivalently,
\begin{equation}\label{eq--dbar star 0}
    |\nabla_1 \eta_{\overline{1}}|^2
\leq  \frac{n-1}{n}\sum_{j=1}^n|\nabla_j \eta_{\overline{\jmath}}|^2.
\end{equation}
From $\overline{\partial}\eta=0$, we have $\nabla_{\overline{\imath}} \eta_{\overline{\jmath}}=\nabla_{\overline{\jmath}}\eta_{\overline{\imath}}$, so that 
\begin{equation}\label{eq-dbar 0}
    \sum_{i=2}^n |\nabla_{\overline{\imath}}\eta_{\overline{1}}|^2 = \frac{1}{2}\sum_{i=2}^n |\nabla_{\overline{1}}\eta_{\overline{\imath}}|^2 + \frac{1}{2} \sum_{i=2}^n |\nabla_{\overline{\imath}}\eta_{\overline{1}}|^2 \leq \frac{1}{2} \sum_{1\leq i \neq j \leq n} |\nabla_{\overline{\imath}}\eta_{\overline{\jmath}}|^2.
\end{equation}

At the point $p$, using the fact that $\eta_{\overline{\jmath}}=0$ for all $j\ge 2$ and the Cauchy--Schwarz inequality, we obtain that for any $\lambda>0$, we have
\begin{equation}\label{eq--nabla square}
    \begin{aligned} \frac{1}{2}|\nabla |\eta|_{g \otimes h}^2|_{g_t}^2 =& g^{\overline{\jmath}i}g^{\overline{l}k}g^{\overline{q}p} \nabla_i (\eta_{\overline{l}}\overline{\eta_{\overline{k}}}) \nabla_{\overline{\jmath}}(\eta_{\overline{q}}\overline{\eta_{\overline{p}}})  = \sum_{i=1}^n |\overline{\eta_{\overline{1}}}\nabla_i \eta_{\overline{1}} + \eta_{\overline{1}}\overline{\nabla_{\overline{\imath}}\eta_{\overline{1}}}|^2 \\
\leq & |\eta|_{g \otimes h}^2 \left( (1+\lambda) |\nabla_1 \eta_{\overline{1}}|^2 +(1+\lambda^{-1}) |\nabla_{\overline{1}}\eta_{\overline{1}}|^2 \right) \\
&+|\eta|_{g \otimes h}^2 \left( \frac{3}{2}\sum_{i=2}^n |\nabla_i \eta_{\overline{1}}|^2 +  3 \sum_{i=2}^n |\nabla_{\overline{\imath}} \eta_{\overline{1}}|^2 \right)
\end{aligned}
\end{equation}
Combining \eqref{eq--dbar star 0}--\eqref{eq--nabla square} yields
\begin{equation}
    \begin{aligned}
        \frac{1}{2}|\nabla |\eta|_{g \otimes h}^2|_{g}^2 \leq |\eta|_{g \otimes h}^2 & \left( \frac{(1+\lambda)(n-1)}{n} \sum_{j=1}^n|\nabla_j \eta_{\overline{\jmath}}|^2 + (1+\lambda^{-1})|\nabla_{\overline{1}}\eta_{\overline{1}}|^2\right.\\
        &\left.+\frac{3}{2}\sum_{1\leq i\neq j \leq n} (|\nabla_i \eta_{\overline{\jmath}}|^2 + |\nabla_{\overline{\imath}}\eta_{\overline{\jmath}}|^2)\right).
    \end{aligned}
\end{equation} Choosing $\lambda=\frac{n}{n-1}$ and observing that at $p$ we have
\[
|\nabla \eta|_{g \otimes h}^2
= \sum_{i,j=1}^n \left( |\nabla_i \eta_{\overline{\jmath}}|^2
+ |\nabla_{\overline{\imath}}\eta_{\overline{\jmath}}|^2 \right),
\]
we obtain 
$$\frac{1}{2}|\nabla |\eta|_{g \otimes h}^2|_{g}^2\leq \frac{2n-1}{n}|\eta|_{g \otimes h}^2 |\nabla\eta|_{g \otimes h}^2.$$
\end{proof}

\begin{lemma}\label{lem--kato for holo gradient}
    Let $(M,\omega)$ be a K\"ahler manifold and $(L,h)$ be a Hermitian holomorphic line bundle with $\Theta_h=\lambda\omega$ for some $\lambda>0$. For any $u\in H^0(M,L)$, we have
    \begin{equation}\label{eq--improved kato for nablau}
        |\nabla^{1,0}|\nabla u|^2|^2\leq |\nabla u|^2\left(|\nabla^{1,0}\nabla u|^2+2\lambda|u||\nabla^{1,0}\nabla u|+\lambda^2|u|^2\right).
    \end{equation}
    \end{lemma}
\begin{proof} Using the fact that $u$ is a holomorphic section and that $\Theta=\lambda \omega$, we have $$\nabla_i\nabla_{\bar j}u=0, \text{ and } \nabla_{\bar k}\nabla_j u=-\lambda g_{j\bar k}u.$$ Then we can compute
    \begin{align*} 
|\nabla^{1,0}|\nabla u|_{g\otimes h}^{2}|_g^{2}= & h g^{\overline{\jmath}i}\nabla_{i}(g^{\overline{l}k}\nabla_{k}u\overline{\nabla_{l}u})\nabla_{\overline{\jmath}}(g^{\overline{q}p}\nabla_{p}u\overline{\nabla_{q}u})\\
= & h g^{\overline{\jmath}i}g^{\overline{l}k}g^{\overline{q}p}\left(\nabla_{i}\nabla_{k}u\overline{\nabla_{l}u}-\nabla_{k}u\overline{[\nabla_{l},\nabla_{\overline{\imath}}]u}\right)\left(\nabla_{p}u\overline{\nabla_{j}\nabla_{q}u}-\overline{\nabla_{q}u}[\nabla_{p},\nabla_{\overline{\jmath}}]u\right)\\
= & h g^{\overline{\jmath}i}g^{\overline{l}k}g^{\overline{q}p}\left(\nabla_{i}\nabla_{k}u\overline{\nabla_{l}u}- \lambda g_{i\overline{l}}\bar{u}\nabla_{k}u\right)\left(\nabla_{p}u\overline{\nabla_{j}\nabla_{q}u}-\lambda g_{p\overline{\jmath}}u\overline{\nabla_{q}u}\right)\\
\leq & h g^{\overline{\jmath}i}g^{\overline{l}k}g^{\overline{q}p}\nabla_{i}\nabla_{k}u\overline{\nabla_{j}\nabla_{q}u}\nabla_{p}u\overline{\nabla_{l}u}+|\nabla u|_{g\otimes h}^{2}(2|\nabla \nabla u|_{g\otimes h}\cdot \lambda|u|_h+\lambda^2|u|_h^{2})\\
\leq & |\nabla u|_{g\otimes h}^{2}\left(|\nabla^{1,0} \nabla u|_{g\otimes h}^{2}+2\lambda|u|_h |\nabla^{1,0} \nabla u|_{g\otimes h}+\lambda^2|u|_h^{2}\right),
\end{align*}
which completes the proof of \eqref{eq--improved kato for nablau}.
\end{proof}

\subsection{Cut-off functions}
\label{sec:cutoff}
In this section, we discuss three types of cut-off functions on a singular K\"ahler--Ricci shrinker. 

First, we recall the existence of suitable cut-off functions exhausting the regular set, as proved in \cite[Lemma~15.27]{bamler3} and \cite[Proposition~1.6]{CMZ2024}.
\begin{lemma}\label{lem:chicutoff} There exist $\chi_{\epsilon}\in C^{\infty}(\mathcal{R}_{-1})$
such that the following hold:
\begin{enumerate}
    \item $\chi_{\epsilon}|_{\{r_{\operatorname{Rm}}\leq\epsilon\}}\equiv0$, $\chi_{\epsilon}|_{\{r_{\operatorname{Rm}}\geq2\epsilon\}}\equiv1$,
$0\leq\chi_{\epsilon}\leq1$ everywhere;
\item for any $\sigma\in(0,4)$, $x_0 \in \mathcal{X}_{-1}$, and $r\in(0,\infty)$,
\begin{equation*} \lim_{\epsilon\searrow0}\int_{B(x_0,r)\cap \mathcal{R}_{-1}}|\nabla\chi_{\epsilon}|^{4-\sigma}dg_{-1}=0. \end{equation*} 
\end{enumerate}
\end{lemma}

Second, we then construct cut-off functions that exhaust compact subsets of $\mathcal{X}$; this is straightforward. We fix $\rho\in C^{\infty}(\mathbb{R})$
such that $\rho|_{(-\infty,1]}\equiv1$ and $\rho|_{[2,\infty)}\equiv0$,
and define $\rho_{r}:\mathcal{X}\to[0,\infty)$ by
\begin{equation}\label{rhocutoff}
    \rho_{r}(x):=\rho(r^{-2}f_{t}(x)).
\end{equation}
Because $\mathcal{L}_{\partial_{t}-\frac{1}{2}\nabla f}f=0$, we have
$\mathcal{L}_{\partial_{t}-\frac{1}{2}\nabla f}\rho_{r}=0$ on $\mathcal{R}$. Since $f_t$ is proper and the scalar curvature is non-negative, using \eqref{eq--relation between nabla f and f}, we obtain the following.
\begin{lemma}\label{lem--rho-cutoff}
There exists a constant $C$
such that for all $r\in[1,\infty)$ and $x\in\mathcal{X}_{[-1,-\frac{1}{2}]}$, 
\begin{equation}
\begin{aligned}
r^2(|\nabla\rho_{r}|^2+|\Delta\rho_{r}|)\leq C.
\end{aligned}
\end{equation}
 We also have 
\begin{equation}\label{eq-heat equ for rho}
    (\partial_{t}-\Delta)\rho_{r}^{2}=-2(\rho_{r}\Delta_{f}\rho_{r}+|\nabla^{1,0}\rho_{r}|^{2}).
\end{equation}
\end{lemma}

Finally, we construct cut-off functions that are more localized in nature and can be regarded as analogues of the cut-off functions constructed for non-collapsed Ricci limit spaces in \cite{CC-almostrigidity}. The proof follows the same strategy as that in \cite[Theorem 1.3]{BZI}.

\begin{lemma}
\label{lem:cutoff} For any $D<\infty$, there exists $C=C(D,W)>0$
such that the following holds. For any $x_{0}\in B(p,D)$ and $r\in(0,C^{-1}]$,
there exists $\varphi\in C^{\infty}(\mathcal{R}_{-1})\cap C^{0}(\mathcal{X}_{-1})$
such that the following hold:

\begin{enumerate}
    \item $\varphi|_{B(x_{0},r)}\equiv1$, $\supp(\varphi)\subseteq B(x_{0},Cr)$,

    \item $r^{2}(|\nabla\varphi|^{2}+|\Delta\varphi|)\leq C$. 
\end{enumerate}
\end{lemma}

\begin{proof}
It is sufficient to prove the case where $x_{0}\in\mathcal{R}_{-1}$. For $x\in\mathcal{R}_{-1}$, we define
\[
v(x):=K(x;x_{0}(-1-r^{2})).
\]

 We consider the curve in the space-time $\gamma(t):=x(t)=\phi_{t}(x)$ for
$t\in[-1-r^{2},-1]$, where $\phi_t$ are given by \eqref{eq-self-diff}. Then using the fact that $R_{{g}_t}(\gamma(t))=\frac{1}{|t|}R_{{g}_{-1}}(x)$ and $|\dot{\gamma}(t)|_{{g}_{t}}^{2}=\frac{1}{2|t|}|\nabla^{{g}_{-1}}f|_{{g}_{-1}}^{2}(x)$ and 
\eqref{eq--relation between nabla f and f}
we can estimate: 
\begin{equation*}
    \begin{split}
\mathcal{L}(\gamma)&= \int_{-1-r^{2}}^{-1}\sqrt{-1-t}\left(R_{{g}_t}(\gamma(t))+|\dot{\gamma}(t)|_{g_{t}}^{2}\right)dt\\
&\leq  r\int_{-1-r^{2}}^{-1}\left(R_{{g}_{-1}}(x)+\frac12|\nabla^{{g}_{-1}}{f}|_{{g}_{-1}}^{2}(x)\right)dt\\
&=  r^{3}\left({f}(x)-W\right)\leq  \frac{r^3}{2} ({d}_{-1}^2(p,x)+C(n)),
\end{split}
\end{equation*}
where in the last inequality we used \cite[Theorem 1.1]{CMZ2024}.
 Therefore we have an upper bound on Perelman's reduced length $$\ell_{x}(x(-1-r^{2}))\leq\frac{r^{2}}{2}(d_{-1}^{2}(p,x)+C(n)),$$ so that \cite[Lemma 22.2]{bamler3} gives 
\begin{equation}\label{eq--lower bound of conjugate heat kernel}
    K(x;x(-1-r^{2}))\geq\frac{1}{(2\pi r^{2})^{n}}e^{-\frac{1}{2}d_{-1}^{2}(p,x)-C(n)}.
\end{equation}
Note that \cite[Theorem 7.2]{bamler1} implies an upper bound on $v(x)K(x;x(-1-r^2))$. Combining this upper bound on the product with the lower bound of $K(x;x(-1-r^2))$ in  \eqref{eq--lower bound of conjugate heat kernel}, we can obtain an upper bound for $v$ 
\begin{equation}
    v(x)\leq\frac{C(W)}{r^{2n}}e^{-\frac{1}{10r^{2}}d_{-1}^{2}(x,x_{0})+\frac{1}{2}d_{-1}^{2}(p,x)}.
\end{equation}
If we assume $r\leq10^{-5}$, then by the triangle inequality,
\begin{equation}\label{eq--upper bound of v}
    v(x)\leq \frac{C(W)}{r^{2n}}e^{-\frac{d_{-1}^{2}(x,x_{0})}{10r^{2}}+\frac{d_{-1}^{2}(p,x)}{2}}\leq \frac{C(W)}{r^{2n}}e^{-\frac{d_{-1}^{2}(x,x_{0})}{20r^{2}}+d_{-1}^{2}(p,x_{0})}.
\end{equation}

In addition, by \eqref{eq--lower bound of conjugate heat kernel}, we know that 
\begin{equation}
    v(x_0)\geq \frac{1}{(2\pi r^2)^n}e^{-\frac12 d^2_{-1}(p,x_0)},
\end{equation}
and then we can apply the gradient estimate \cite[Theorem 7.5]{bamler1}
\begin{equation}\label{eq--gradient bound of v}
    |\nabla\log v|\leq\frac{C(W)}{r}\sqrt{\log\left(\frac{C(W)}{r^{2n}v}\right)}
\end{equation}
to obtain a pointwise lower bound 
\begin{equation}\label{eq--lower bound of v}
    v(x)\geq\frac{1}{C(W)r^{2n}}e^{-\frac{1}{4}d_{-1}^{2}(p,x_{0})-\frac{d_{-1}^{2}(x,x_{0})}{r^{2}}}.
\end{equation}

If we choose $C:=10d_{-1}(p,x_{0})+2(\log C(W)+10)$,
it follows from \eqref{eq--upper bound of v} that for all $x\in\mathcal{X}_{-1}\setminus B(x_{0},C r)$,
we have
\begin{equation}\label{eq--upper bound of v-second}
v(x)r^{2n}\leq\frac{1}{20C(W)}e^{-\frac{1}{2}d_{-1}^{2}(p,x_{0})}. 
\end{equation}
Choose a cut-off function $\xi\in C^{\infty}(\mathbb{R})$ such that
$\xi|_{(-\infty,\frac{1}{5}]}\equiv0$ and $\xi|_{[\frac{1}{2},\infty)}\equiv1$ and we 
define 
\[
\varphi(x):=\xi\left(C(W)r^{2n}v(x)e^{\frac{1}{2}d_{-1}^{2}(p,x_{0})+1}\right).
\]
By \eqref{eq--lower bound of v}, we know that $\varphi|_{B(x_{0}, r)}\equiv1$, and by \eqref{eq--upper bound of v-second} we know that $\varphi|_{\mathcal{X}_{-1}\setminus B(x_{0},C r)}\equiv0$. Moreover by \eqref{eq--gradient bound of v} and the choice of $\xi$, we obtain 
\[
r|\nabla\varphi|(x)\leq C(W,d_{-1}(p,x_{0})).
\]
Since $\mathcal{X}_{-1}$ has locally bounded scalar curvature, using the smooth convergence of the conjugate heat kernel and passing the pointwise estimate \cite[Lemma 3.1]{BZI} to the limit, we obtain
\[
r^{2}|\Delta\varphi|(x)\leq C(W,d_{-1}(p,x_{0})).
\]
\end{proof}

\subsection{The heat equation for holomorphic pluri-anticanonical sections}

In the following, all geometric quantities depend on $t$, but for simplicity of notation, we omit the explicit $t$-dependence. 
  
\begin{prop}\label{prop--heat equation for 1-form}
Suppose $\eta_{-1}\in \mathcal{A}^{0,1}(\mathcal{R}_{-1},K_{\mathcal{R}_{-1}}^{-\ell})$ with $\pp \eta_{-1}=0$ and extend $\eta_{-1}$ to $\mathcal{R}$ by demanding $\mathcal{L}_{\partial_t-\frac{1}{2}\nabla f}\eta=0$. Then 
\begin{equation}\label{L2evolution}
    (\partial_{t}-\Delta)|\eta|_{g_{t}\otimes h_{t}^{\ell}}^{2}=-|\nabla \eta|^{2}+2\Re\langle \eta,\pp\pp_f^{\ast}\eta\rangle+\mathrm{Im}\langle\mathcal{L}_{J\nabla f}\eta,\eta\rangle+\frac{(n-2)\ell}{|t|}|\eta|^{2}+\ell\Delta_ff|\eta|^2.
\end{equation}
\end{prop}

\begin{proof}
Letting $\Theta=\frac{\ell}{|t|}\omega$ denote the curvature of the metric $h_t^\ell$ on $K_{\mathcal{R}_{t}}^{-\ell}$,
our hypothesis that on  $\eta$ is $\pp$-closed and Lemma \ref{lem-bochner formula} yield
\begin{equation}
    -(\nabla^{0,1})^{\ast}\nabla^{0,1}\eta=(\Ric+\Theta)(\eta)+\mathcal{L}_{\nabla^{0,1}f}\eta-\overline{\partial}\overline{\partial}_f^{\ast}\eta,
\end{equation}
and 
\begin{equation}
    \begin{aligned}
        -(\nabla^{1,0})^{\ast}\nabla^{1,0}\eta= & -\left(\partial^{\ast}\partial+\partial\partial^{\ast}\right)\eta\\
= & -\left(\overline{\partial}^{\ast}\overline{\partial}+\overline{\partial}\overline{\partial}^{\ast}\right)\eta+\Theta(\eta)-\text{tr}_{\omega}(\Theta)\eta\\
= & -\pp\pp_f^{\ast}\eta+\mathcal{L}_{\nabla^{0,1}f}\eta+\Theta(\eta)-\text{tr}_{\omega}(\Theta)\eta,
    \end{aligned}
\end{equation}
so that
\begin{align*}
\Delta|\eta|_{g_{t}\otimes h^{\ell}_t}^{2}= & |\nabla\eta|^{2}-\langle(\nabla^{0,1})^{\ast}\nabla^{0,1}\eta,\eta\rangle-\langle\eta,(\nabla^{1,0})^{\ast}\nabla^{1,0}\eta\rangle\\
= & |\nabla\eta|^{2}+(\Ric+\Theta)(\eta,\eta)+\langle\mathcal{L}_{\nabla^{0,1}f}\eta,\eta\rangle+\langle\eta,\mathcal{L}_{\nabla^{0,1}f}\eta\rangle\\
 & -2\Re\langle\eta,\pp\pp_f^{\ast}\eta\rangle+\Theta(\eta,\eta)-\text{tr}_{\omega}(\Theta)|\eta|^{2}\\
= & |\nabla\eta|^{2}+(\Ric+2\Theta)(\eta,\eta)-\text{tr}_{\omega}(\Theta)|\eta|^{2}\\
 & -2\Re\langle\eta,\pp\pp_f^{\ast}\eta\rangle+2\mathrm{Re}\langle\mathcal{L}_{\nabla^{0,1}f}\eta,\eta\rangle.
\end{align*}
Note that we have 
\begin{equation}
    \partial_t\log h_t=\frac{n}{|t|}-\partial_tf_t-R.
\end{equation}
Combining this with \eqref{eq--shrinker equation}--\eqref{eq-scalar curvature and laplacian}, we know that 
\begin{equation}
    \partial_t\log h_t=\Delta_ff.
\end{equation}
Then we obtain that 
\begin{equation} \partial_{t}|\eta|_{g_{t}\otimes h^{\ell}_t}^{2}=2\text{Re}\langle\partial_{t}\eta,\eta\rangle+\Ric(\eta,\eta)+\ell(\Delta_{f}f)|\eta|^{2},
\end{equation}
and hence
\begin{align*}
(\partial_{t}-\Delta)|\eta|_{g \otimes h^{\ell}}^{2}= & -|\nabla \eta|^{2}+2\Re\langle\eta,\pp\pp_f^{\ast}\eta\rangle+2\text{Re}\langle\partial_{t}\eta-\mathcal{L}_{\nabla^{0,1} f}\eta,\eta\rangle+\frac{(n-2)\ell}{|t|}|\eta|^{2}+\ell \Delta_ff|\eta|^2.
\end{align*}
Recall that $\nabla^{0,1}f=\frac{1}{2}(\nabla f+\ii J\nabla f)$. Since $\eta$ satisfies $\mathcal{L}_{\partial_t-\frac{1}{2}\nabla f}\eta=0$, we obtain \eqref{L2evolution}.
\end{proof}

\subsection{Local $L^\infty$-bound and H\"ormander $L^2$-estimate}
\label{sec:localboundedness}
The following estimate is a key ingredient in the proof of H\"ormander's $L^2$-estimate on a singular K\"ahler--Ricci shrinker. Its proof relies on the improved Kato inequality from Lemma \ref{lem-improved kato}.

We use the same notation as in Section~\ref{Singular Kahler-Ricci shrinkers}. 
In particular, we implicitly use the Hermitian metric $h$ on $K^{-1}_{\mathcal{R}_{-1}}$ induced by the volume form 
\(
(2\pi)^{-n} e^{-f} \frac{\omega^n}{n!}.
\)
\begin{prop} \label{prop:localbound}
For $\ell\in \mathbb N_{\geq 0}$, suppose $\eta\in\mathcal{A}^{0,1}(\mathcal{R}_{-1},K_{\mathcal{R}_{-1}}^{-\ell}),$ $\varphi \in C_c^{\infty}(\mathcal{R}_{-1})$ satisfy
\begin{equation} \int_{\mathcal{R}_{-1}} |\eta|^2 e^{-\varphi} d\mu_{-1}<\infty, \qquad \operatorname{supp}(\overline{\partial}\eta)\cup\operatorname{supp}(\overline{\partial}_{f+\varphi}^{\ast}\eta) \cup \operatorname{supp}(\varphi)\subset\subset\mathcal{R}_{-1}.  \end{equation} Then 
\[
|\eta|\in L^{\infty}_{\rm{loc}}(\mathcal{X}_{-1}).
\]
\end{prop}

\begin{proof}
 We extend $\eta$ and $\varphi$  to $\mathcal{R}$ by demanding $\mathcal{L}_{\partial_t-\frac{1}{2}\nabla f}\eta=0$ and $\mathcal{L}_{\partial_t-\frac{1}{2}\nabla f}\varphi=0$. As discussed in Section \ref{Singular Kahler-Ricci shrinkers}, the assumption in this proposition implies that there exists $\epsilon_0>0$ such that 
$$\left(\text{supp}(\overline{\partial}\eta)\cup\text{supp}(\overline{\partial}_{f}^{\ast}\eta) \cup \text{supp}(\varphi)\right)\cap \mathcal{X}_{[-1,-\frac12]} \subseteq \{ r_{\operatorname{Rm}} \geq 4\epsilon_0\}.$$

Let $\rho_r$ be the cut-off function defined by \eqref{rhocutoff}. Let $\chi_{\epsilon}'$ be the cut-off function constructed in Lemma \ref{lem:chicutoff}, and let 
\begin{equation} \chi_{\epsilon}:=\chi_{\epsilon}'(1-\chi'_{2\epsilon_0}).
\end{equation}
Then $(\chi_{\epsilon})_{\epsilon\in(0,1]}$ satisfies
$$\text{supp}(\chi_{\epsilon})\subseteq\{\epsilon\leq  r_{\operatorname{Rm}}\leq 4\epsilon_0\},$$
$$\chi_{\epsilon}|_{\{2\epsilon\leq r_{\operatorname{Rm}}\leq \epsilon_0\}\cap B(p,r)}\equiv1$$
for any $r\in(0,\infty)$ when $\epsilon=\epsilon(r)>0$ is sufficiently
small. Moreover for each $\sigma>0$ and $r\in(0,\infty)$,
\begin{equation} \label{etacutoff}
\begin{aligned}  \lim_{\epsilon\searrow0}\int_{\mathcal{R}_{-1}\cap \{r_{\operatorname{Rm}}\leq \epsilon_0\} \cap B(x_{0},r)}|\nabla\chi_{\epsilon}|^{4-\sigma}dg_{-1}&=0, \\
\limsup_{\epsilon \searrow 0} \int_{\mathcal{R}_{-1}\cap \{r_{\operatorname{Rm}}>\epsilon_0\}\cap B(x_0,r)} |\nabla \chi_{\epsilon}|^{4-\sigma} dg_{-1}&\leq C(W,\epsilon_0,r). 
\end{aligned}
\end{equation}
We extend $\chi_{\epsilon}$ to functions on $\mathcal{R}$ by
$\mathcal{L}_{\partial_{t}-\frac{1}{2}\nabla f}\chi_{\epsilon}=0$. Then we have 
\begin{equation}\label{eq-heat equ for chi}
    (\partial_{t}-\Delta)\chi_{\epsilon}^{2}=-2(\chi_{\epsilon}\Delta_{f}\chi_{\epsilon}+|\nabla^{1,0}\chi_{\epsilon}|^{2}).
\end{equation}

Arguing as in \eqref{eq--difference between lie-derivative and connection} and using the fact that $\nabla^{1,0}f$ is a holomorphic vector field, we have 
\begin{equation}\label{eq-1,0difference between lie-derivative and connection}
    \nabla^{1,0}f\lrcorner \nabla \eta=\mathcal{L}_{\nabla^{1,0}f}\eta + \ell(\Delta f -|\nabla^{1,0}f|^2)\eta.
\end{equation}
Combining \eqref{eq--difference between lie-derivative and connection} and \eqref{eq-1,0difference between lie-derivative and connection}, we obtain
\begin{equation} \label{eq:Liederivativeidentity} \text{Im}\langle \mathcal{L}_{J\nabla f}\eta,\eta\rangle = \text{Im} \langle \nabla_{J\nabla f}\eta,\eta\rangle - \nabla^2f(\eta,\overline{\eta})-\ell \Delta_f f |\eta|^2,\end{equation}
where in local holomorphic coordinates,
$$\nabla^2f(\eta,\overline{\eta}):=g^{\overline{q}p}g^{\overline{\jmath}k} \eta_{\overline{q}}\overline{\eta_{\overline{k}}}\nabla_p \nabla_{\overline{\jmath}}f.$$
Wherever a smooth time-dependent family $\eta_{t}\in\mathcal{A}^{0,1}(\mathcal{R}_{t})$
satisfies $\mathcal{L}_{\partial_{t}-\frac{1}{2}\nabla f}\eta_{t}=0$, $\overline{\partial}\eta_t=0$, and $\overline{\partial}_f^{\ast}\eta_t=0$, we can combine \eqref{L2evolution} with \eqref{eq:Liederivativeidentity} to obtain the following estimate on $\mathcal{X}_{[-1,-\frac12]}\cap \supp(\rho_{r})\cap \supp(\chi_{\epsilon_0})$
\begin{align*}
(\partial_{t}-\Delta)|\eta|^{2}= & -|\nabla \eta|^{2}+\text{Im}\langle\mathcal{L}_{J\nabla f}\eta,\eta\rangle+\frac{(n-2)\ell}{|t|}|\eta|^{2}\\
\leq & -|\nabla \eta|^{2}+C(r)\cdot|\eta|\cdot|\nabla \eta|-\nabla^{2}f(\eta,\overline{\eta})+ 2n\ell|\eta|^2.
\end{align*}
If we set $\beta:=\frac{2n-1}{2n}$, then $\beta\in(0,1)$ and that $2-\beta<\frac{2n}{2n-1}$. Thus we may use the improved Kato inequality Lemma \ref{lem-improved kato}
to estimate
\begin{align*}
(\partial_{t}-\Delta)|\eta|^{\beta}=
& \frac{\beta}{2}|\eta|^{\beta-2}(\partial_{t}-\Delta)|\eta|^{2}+\frac{\beta(2-\beta)}{4}|\eta|^{\beta-4}|\nabla^{1,0}|\eta|^2|^{2}\\
\leq & \frac{\beta}{2}|\eta|^{\beta-2}\biggl(-|\nabla\eta|^{2}+C(r)\cdot|\nabla\eta|\cdot|\eta|-\nabla^2f(\eta,\overline{\eta})+2n\ell|\eta|^2 \\
&\quad\quad\quad\quad\quad +\frac{(2-\beta)(2n-1)}{2n}|\nabla\eta|^{2}\biggr).
\end{align*}
We then use the Cauchy--Schwarz to obtain
\begin{equation}
\begin{aligned}
      (\partial_t-\Delta)|\eta|^{\beta} \leq & -\frac{1}{C_0}|\eta|^{\beta-2}|\nabla\eta|^{2}+C(r)|\eta|^{\frac{\beta}{2}}\cdot|\eta|^{\frac{\beta-2}{2}}|\nabla\eta|-\frac{\beta}{2}|\eta|^{\beta-2}\nabla^2f(\eta,\overline{\eta})+\beta n\ell|\eta|^{\beta}\\
\leq & -\frac{1}{C_0}|\eta|^{\beta-2}|\nabla\eta|^{2}+C(r,\ell)|\eta|^{\beta}-\frac{\beta}{2}|\eta|^{\beta-2}\nabla^2 f(\eta,\overline{\eta})
\end{aligned}
\end{equation}
for some $C_0=C_0(n,r)<\infty$. Since we allow constants to depend on $r$, $\ell$, and the dimension $n$, we will not specify this dependence explicitly in what follows whenever a quantity depends only on these parameters. On the other hand, we will use the notation $C_i$ for certain constants, since later auxiliary parameters will be chosen depending on $C_i$.

Fix $t_{0}=-\frac{1}{2}$ and $y\in\mathcal{R}_{t_{0}} \cap \{ r_{\operatorname{Rm}}\leq \frac{1}{2}\epsilon_0\} \cap \{\rho_{r}=1\}$. In the following, let $d\nu_{y;t}$ denote the conjugate heat kernel measure based at the point $y$ and we write 
\begin{equation}
  d\nu_{y;t}=K(y;\cdot)dg_{t}=(2\pi\tau_{y})^{-n}e^{-f_{y}}dg_{t}.
\end{equation}
For a given $t\in [-1,-\frac{1}{2}]$, we use the fact that $\nabla f$ is locally bounded on $\mathcal{X}$ to obtain
\begin{align*}
\frac{d}{dt}\int_{\mathcal{R}_{t}}|\eta|^{\beta}\chi_{\epsilon}^{2}\rho_{r}^{2}d\nu_{y;t}= & \int_{\mathcal{R}_{t}}(\partial_{t}-\Delta)(|\eta|^{\beta}\chi_{\epsilon}^{2}\rho_{r}^{2})d\nu_{y;t}\\
\leq & \int_{\mathcal{R}_{t}}\left(-\frac{1}{C_0}|\eta|^{\beta-2}|\nabla\eta|^{2}+C|\eta|^{\beta}+\frac{\beta}{2}|\eta|^{\beta-2}\nabla^2 f(\eta,\overline{\eta})\right)\chi_{\epsilon}^{2}\rho_{r}^{2}d\nu_{y;t}\\
 & +C\int_{\mathcal{R}_{t}}|\eta|^{\frac{\beta-2}{2}}|\nabla\eta|\cdot|\eta|^{\frac{\beta}{2}}\left(\chi_{\epsilon}^{2}\rho_{r}|\nabla\rho_{r}|+\rho_{r}^{2}\chi_{\epsilon}|\nabla\chi_{\epsilon}|\right)d\nu_{y;t}\\
 & +C\int_{\mathcal{R}_{t}}|\eta|^{\beta}(|\nabla\chi_{\epsilon}|^{2}\rho_{r}^{2}+\chi_{\epsilon}^2\rho_{r}|\Delta\rho_{r}|+\chi_{\epsilon}^2|\nabla\rho_{r}|^{2})d\nu_{y;t}\\
& - 2\int_{\mathcal{R}_t} |\eta|^{\beta} \rho_r^2 \chi_{\epsilon}\Delta \chi_{\epsilon}d\nu_{y;t}
\end{align*}
We then integrate by parts the terms involving $\nabla^2 f$ and $\Delta \chi_{\epsilon}$ and use Cauchy--Schwarz and the fact that $\nabla f$ is locally bounded to obtain the estimate
\begin{equation}\label{eq--three terms for derivative}
    \begin{aligned} 
\frac{d}{dt}\int_{\mathcal{R}_{t}}|\eta|^{\beta}\chi_{\epsilon}^{2}\rho_{r}^{2}d\nu_{y;t}\leq & -\frac{1}{2C_0}\int_{\mathcal{R}_{t}}|\eta|^{\beta-2}|\nabla\eta|^{2}\chi_{\epsilon}^{2}\rho_{r}^{2}d\nu_{y;t} + C\int_{\mathcal{R}_t} |\eta|^{\beta}\chi_{\epsilon}^2\rho_r^2 d\nu_{y;t}\\  & 
+C_1\int_{\mathcal{R}_{t}}|\eta|^{\beta}\chi_{\epsilon}^{2}\rho_{r}^{2}|\nabla\log K(y;\cdot)|d\nu_{y;t} \\
 & +C\int_{\mathcal{K}_{t,\epsilon,r}}|\eta|^{\beta}\left(1+|\nabla\chi_{\epsilon}|^{2}+|\nabla\log K(y;\cdot)|^2\right)d\nu_{y;t},
    \end{aligned}
\end{equation}
where we have set 
$$\mathcal{K}_{t,\epsilon,r}:=\mathcal{R}_{t}\cap(\text{supp}(\nabla\chi_{\epsilon})\cup\text{supp}(\nabla\rho_{r})).$$

Then we will use Lemma \ref{lem--integration by part trick} to estimate the term in the second line of \eqref{eq--three terms for derivative}. 
Letting $\tau_{y}(t) := t_0 -t$, Perelman's differential Harnack inequality (Theorem \ref{thm-harnack inequality}) states that
\begin{align*}
0\geq & \tau_{y}(R+2\Delta f_{y}-|\nabla^{1,0} f_{y}|^{2})+f_{y}-n.
\end{align*}
Combining this with $R \geq 0$ and $f_y \geq -C(W,r)$ (which follows from \cite[Theorem 7.1]{bamler1} and its singular version \cite{CMZ2024}), we obtain
\begin{align*} \tau_y |\nabla^{1,0} f_y|^2 &\leq -\tau_y R - 2\tau_y (\Delta f_y - |\nabla^{1,0} f_y|^2) -f_y+n \\
&\leq C(W,r)-2\tau_y (\Delta f_y - |\nabla^{1,0} f_y|^2),
\end{align*}
or equivalently
\begin{equation}\label{eq--gradient bound by laplacian}
    |\nabla^{1,0} \log K(y;\cdot)|^2 K(y;\cdot)\leq 2\Delta K(y;\cdot) + \frac{C(r,W)}{\tau_y}K(y;\cdot).
\end{equation}

Therefore, we can apply Lemma \ref{lem--integration by part trick} with $u=K(y;\cdot)$.  Choose $\epsilon$ in Lemma \ref{lem--integration by part trick} sufficiently small so that $$\epsilon C_1   \le \frac{1}{4C_0},$$where $C_0$ and $C_1$ are the constants in \eqref{eq--three terms for derivative}.
Then combining \eqref{eq--three terms for derivative} and Lemma \ref{lem--integration by part trick}, taking $u=K(y;\cdot)$ and $\epsilon=\epsilon_1\sqrt{\tau_y}$ yields
\begin{align*} \frac{d}{dt}\int_{\mathcal{R}_{t}}|\eta|^{\beta}\chi_{\epsilon}^{2}\rho_{r}^{2}d\nu_{y;t} \leq &\frac{C}{\sqrt{\tau_y}} \int_{\mathcal{R}_t} |\eta|^{\beta} \chi_{\epsilon}^2 \rho_r^2 d\nu_{y;t} \\ 
 & + C\sqrt{\tau_y}\int_{\mathcal{K}_{t,\epsilon,r}}|\eta|^{\beta}\left(1+|\nabla\chi_{\epsilon}|^{2}+|\nabla\log K(y;\cdot)|^2\right)d\nu_{y;t}
 \\ 
 \leq & \frac{C}{\sqrt{\tau_y}} \int_{\mathcal{R}_t} |\eta|^{\beta} \chi_{\epsilon}^2 \rho_r^2 d\nu_{y;t} \\ &+
 C\left(\int_{\mathcal{K}_{t,\epsilon,r}}|\eta|^{\beta p}d\nu_{y;t}\right)^{\frac{1}{p}}\left(1+\left(\int_{\mathcal{K}_{t,\epsilon,r}}|\nabla\chi_{\epsilon}|^{2q}d\nu_{y;t}\right)^{\frac{1}{q}}\right)\\
 & +C\sqrt{\tau_y}\left(\int_{\mathcal{K}_{t,\epsilon,r}}|\eta|^{\beta p}d\nu_{y;t}\right)^{\frac{1}{p}}\left(\int_{\mathcal{K}_{t,\epsilon,r}}|\nabla\log K(y;\cdot)|^{2q}d\nu_{y;t}\right)^{\frac{1}{q}}
\end{align*}
where $\frac{1}{p}+\frac{1}{q}=1$. Since $\beta\in (0,1)$, we can choose $p,q$ such that 
\begin{equation}\label{eq--choice of q}
    \beta p<2 \text{ and } 2q<4.
\end{equation}

For all $t\in[-1,t_{0})$, on $\mathcal{K}_{t,\epsilon,r} \cap \{r_{\operatorname{Rm}}\geq \epsilon_0\}$, we know that  by choice of $y$, if $\epsilon =\epsilon(y)>0$ is sufficiently small,
\begin{equation}
    d_{t}(y(t),\mathcal{K}_{t,\epsilon,r} \cap \{r_{\operatorname{Rm}}\ge \epsilon_0\})\geq \frac{1}{2}\epsilon_0.
\end{equation}
 Therefore we have (by passing \cite[Theorem 7.2]{bamler1} to the limit)
\begin{equation}\label{eq--bound on regular region}
    \begin{aligned}
        (2\pi|t|)^{-n}e^{-f}&\geq C(r),\\
K(y;\cdot)\leq\frac{C(W)}{\tau_y^{n}}\exp\left(-\frac{\epsilon_{0}^{2}}{C(W)\tau_y}\right)&\leq C(W,\epsilon_0)\exp\left(-\frac{\epsilon_{0}^{2}}{C(W)\tau_y}\right).
    \end{aligned}
\end{equation}
 Similarly, on $\mathcal{K}_{t,\epsilon,r}\cap \{r_{\operatorname{Rm}}<\epsilon_0\}$, for $\epsilon =\epsilon(y,r,\epsilon_0,W)>0$ sufficiently small, we have
\begin{equation}\label{eq--bound on less regular region}
(2\pi|t|)^{-n}e^{-f}\geq C(r), \qquad K(y;\cdot)\leq C(W,y)\exp\left(-\frac{1}{C(W,y)\tau_y}\right).
\end{equation}
Thus if we take $\epsilon =\epsilon(y,r,\epsilon_0,W)>0$ sufficiently small, then
\begin{align} \label{eq:boundedcutofferror}
\begin{split} \int_{\mathcal{K}_{t,\epsilon,r}}|\nabla\chi_{\epsilon}|^{2q}d\nu_{y;t} & \leq C(W,\epsilon_0,r) \int_{\mathcal{K}_{t,\epsilon,r}\cap \{r_{\operatorname{Rm}}\geq \epsilon_0\}} |\nabla \chi_{\epsilon}|^{2q}dg_{t}\\&\quad  + C(W,y,r)\int_{\mathcal{K}_{t,\epsilon,r} \cap \{r_{\operatorname{Rm}}<\epsilon_0\}} |\nabla \chi_{\epsilon}|^{2q} dg_{t} \\
&\leq 1+C(W,\epsilon_0,r).
\end{split}
\end{align}
   
Recalling from \eqref{eq--bound on regular region} that $K(y;\cdot) \leq C(W,\epsilon_0)$ on $\mathcal{K}_{t,\epsilon,r} \cap \{r_{\operatorname{Rm}}\geq \epsilon_0\}$, we can apply Lemma \ref{lem:conjugategradient} with $r$ replaced by $\min\{\sqrt{\tau_y},r_{\operatorname{Rm}}(z)\}$ to obtain
\begin{align}
    \label{eq:pointwiseconjugatebound} 
\begin{split}
|\nabla_z \log K(y;z)|^{2q}K(y;z) \leq & \frac{C}{\min\{\tau_y^q,r_{\operatorname{Rm}}^{2q}(z)\}} \left( 1 + \log \left( \frac{C(W,\epsilon_0)}{K(y;z)} \right) \right)^{2q}K(y;z) \\ \leq & \frac{C(W,\epsilon_0)}{\min\{\tau_y^q,r_{\operatorname{Rm}}^{2q}(z)\} }
\end{split}
\end{align}
on $\mathcal{K}_{t,\epsilon,r} \cap \{r_{\operatorname{Rm}}\geq \epsilon_0\}$. Combining this with \cite[proof of Proposition 3.11]{CHM25} yields  
\begin{equation}\label{conjugate heat kernel 1}
\int_{(\mathcal{K}_{t,\epsilon,r} \cap \{ r_{\operatorname{Rm}}\geq \epsilon_0\})\cup (\operatorname{supp}(\nabla \rho_r)\cap \mathcal{R}_t)} |\nabla \log K(y;\cdot )|^{2q}d\nu_{y;t} \leq \frac{C(W,\epsilon_0,r)}{\tau_y^q},
\end{equation}
 where $q$ is the index satisfying \eqref{eq--choice of q}.
Since $q<2$, we can choose $\sigma>1$ such that $2q\sigma<4$. We may use H\"older's inequality and apply a similar estimate on $\mathcal{K}_{t,\epsilon,r} \cap \{r_{\operatorname{Rm}}<\epsilon_0\}$ to obtain
\begin{equation}\label{conjugate heat kernel 2}
\begin{aligned}
     & \int_{\operatorname{supp}(\nabla \chi_{\epsilon})\cap \{\rho_r=1\} \cap \{0<r_{\operatorname{Rm}}<\epsilon_0\}} |\nabla \log K(y;\cdot )|^{2q}d\nu_{y;t} \hspace{-40 mm}\\ &\leq \left( \int_{\mathcal{K}_{t,\epsilon,r} \cap \{ r_{\operatorname{Rm}}<\epsilon_0\}} |\nabla \log K(y;\cdot )|^{2q\sigma}d\nu_{y;t} \right)^{\frac{1}{\sigma}}\left( \nu_{y;t}(\{\operatorname{supp}(\nabla \chi_{\epsilon})\cap \{\rho_r=1\} \cap \{0<r_{\operatorname{Rm}}<\epsilon_0\}) \right)^{\frac{\sigma-1}{\sigma}}\\
   &\leq C(W,\epsilon_0)\left(\int_{\mathcal{K}_{t,\epsilon,r}} r_{\operatorname{Rm}}(z)^{-2q\sigma}dg_{t} + \frac{C(W,\epsilon_0,r)}{\tau_y^{q\sigma}}\right)^{\frac{1}{\sigma}} \epsilon^{\frac{\sigma-1}{\sigma}}\\
    &\leq \frac{C(r,W,y)}{\tau_y^{q}} \epsilon^{\frac{\sigma-1}{\sigma}}.
\end{aligned}
\end{equation}

By combining \eqref{conjugate heat kernel 1} and \eqref{conjugate heat kernel 2} and again insisting that $\epsilon =\epsilon(y,r,\epsilon_0,W)>0$ is sufficiently small, we therefore obtain
\begin{equation}
\left(\int_{\mathcal{K}_{t,\epsilon,r}}|\nabla \log K(y;\cdot)|^{2q}d\nu_{y;t}\right)^{\frac{1}{q}} \leq \frac{C(W,\epsilon_0,r)}{\tau_y}.
\end{equation} 
By \eqref{eq--bound on regular region}, \eqref{eq--bound on less regular region}, $\mathcal{L}_{\partial_t-\frac{1}{2}\nabla f}g_t =- \frac{1}{|t|}g_t$, and $\mathcal{L}_{\partial_t - \frac{1}{2}\nabla f}|\eta| = \frac{1}{2|t|}|\eta|$, we also have
\begin{equation}\label{holder control beta-p norm}
\int_{\mathcal{K}_{t,\epsilon,r}}|\eta|^{\beta p}d\nu_{y;t}\leq C(\epsilon_{0},r,W)\int_{\mathcal{K}_{t,\epsilon,r}}|\eta|^{\beta p}dg_{t}\leq C(W,\epsilon_{0},r) \left( \int_{\mathcal{R}_{-1}}|\eta|^{2}d\mu_{-1} \right)^{\frac{\beta p}{2}}.
\end{equation}

Again combining expressions, we get
\begin{equation*}
    \frac{d}{dt} \left( e^{C(r,W,C_0)\sqrt{\tau_y}}\int_{\mathcal{R}_t} |\eta|^{\beta} \chi_{\epsilon}^2 \rho_r^2 d\nu_{y;t} \right) \leq C(W,\epsilon_{0},r)(1+\frac{1}{\sqrt{\tau_y}}) \left( \int_{\mathcal{R}_{-1}} |\eta|^2 d\mu_{-1} \right)^{\frac{\beta}{2}}.
\end{equation*}
We may thus integrate from $t=-1$ to $t=t_0$,  and use that $d\nu_{y;-1}\leq C(r,t_0)d\mu_{-1}$
on $\text{supp}(\rho_{r})\cap \mathcal{R}_{-1}$ to estimate 
\begin{align*}
|\eta|^{\beta}(y)\leq & C(W,\epsilon_{0},r,t_0)\left(\int_{\mathcal{R}_{-1}}|\eta|^{2}d\mu_{-1}\right)^{\frac{\beta}{2}}+C(W,\epsilon_{0},r)\int_{\mathcal{R}_{-1}}|\eta|^{\beta}\rho_{r}^{2}d\mu_{-1}\\
\leq & C(W,\epsilon_{0},r,t_0)\left(\int_{\mathcal{R}_{-1}}|\eta|^{2}d\mu_{-1}\right)^{\frac{\beta}{2}}.
\end{align*}
Since $\mathcal{L}_{\partial_t-\frac12\nabla f}|\eta|^2=\frac{1}{|t|}|\eta|^2$, we obtain that for any $r>0$, $|\eta_{-1}|$ is bounded on a neighborhood of $\mathcal{S} \cap \{\rho_r=1\}$. Because $\eta_{-1}$ is smooth on $\mathcal{R}_{-1}$, it follows that $|\eta_{-1}|$ is in fact bounded on $\{\rho_r=1\}\cap \mathcal{X}_{-1}$.
\end{proof}
\begin{remark}
  Instead of using Proposition \ref{prop--heat equation for 1-form},  we can also use the following formula and then proceed as in the proof of Lemma \ref{lem:C1estimates} given below to give a proof of Proposition \ref{prop:localbound}.  Let $\eta\in \mathcal{A}^{0,1}(\mathcal{R}_{-1},K_{\mathcal{R}_{-1}}^{-\ell})$ with $\pp \eta=0$ and $\pp^*_f\eta=0$ and let $\nabla_f$ denote the Chern connection with respect to the Hermitian metric $h^\ell e^{-f}$:
\begin{equation}
    (\mathcal{L}_{\p_t-\frac12 \nabla f}-\Delta)|\eta|^2_{h^\ell e^{-f}}=\frac{|\eta|^2}{|t|}-|\nabla_f\eta|^2-\frac{(2-n)\ell+1}{|t|}|\eta|^2-\nabla^2f(\eta,\bar\eta)+\Delta f|\eta|^2
\end{equation}
\end{remark}

In the following, we let $L := K_{\mathcal{R}_{-1}}^{-1}$ and, for $\ell \in \mathbb{N}_{\geq 0}$, denote by 
$h_{\ell} := h^{\otimes \ell}$ the induced Hermitian metric on $L^\ell$, and set 
$\omega_{\ell+1} := (\ell+1) \omega_{-1}$. 
We remark that in the following computations it is understood that, for $L^\ell$-valued forms, we always use the Hermitian metric $h_\ell$ and the rescaled K\"ahler metric $\omega_{\ell+1}$. 
When there is no risk of confusion, we omit explicit reference to these metrics.
\begin{prop} \label{prop:hormander} Suppose $\ell \in \mathbb{N}_{\geq 0}$ and $\varphi \in C_c^{\infty}(\mathcal{R}_{-1})$ satisfy 
\begin{equation*}
   \sqrt{-1} \partial \overline{\partial} \varphi \geq c\omega_{\ell+1}
\end{equation*}
for some $c>-1$. Then for any $\alpha \in \mathcal{A}_c^{0,1}(\mathcal{R}_{-1},L^{\ell})$ satisfying $\overline{\partial} \alpha =0$, there exists $w\in L_{\operatorname{loc}}^2(\mathcal{R}_{-1},L^{\ell})$ satisfying $\overline{\partial}w=\alpha$ in the sense of distributions, and 
$$\int_{\mathcal{R}_{-1}} |w|_{h_\ell}^2 e^{-(\varphi+f)} \omega_{\ell+1}^n \leq \frac{1}{1+c}\int_{\mathcal{R}_{-1}} |\alpha|_{\omega_{\ell+1}\otimes h_{\ell}}^2 e^{-(\varphi+f)} \omega_{\ell+1}^n.$$
\end{prop}
\begin{proof} We first show that for any $\eta \in \mathcal{A}_c^{0,1}(\mathcal{R}_{-1},L^{\ell})$, we have
$$\left| \int_{\mathcal{R}_{-1}} \langle \eta,\alpha \rangle e^{-(\varphi+f)} \omega_{\ell+1}^n \right| \leq \left( \int_{\mathcal{R}_{-1}} |\overline{\partial}_{f+\varphi}^{\ast} \eta|^2 e^{-(\varphi+f)} \omega_{\ell+1}^n \right)^{\frac{1}{2}} \left( \int_{\mathcal{R}_{-1}} |\alpha|^2 e^{-(\varphi+f)} \omega_{\ell+1}^n \right)^{\frac{1}{2}}.$$
We first write $\eta = \eta^{(1)} + \eta^{(2)}$, where $\overline{\partial} \eta^{(1)}=0$ and $\int_{\mathcal{R}_{-1}} \langle \eta^{(2)},\alpha\rangle e^{-(\varphi+f)} \omega_{\ell+1}^n = 0$ for any $L^2$ $(0,1)$-form $\alpha$ on $\mathcal{R}_{-1}$ which satisfies $\overline{\partial}\alpha=0$ in the sense of distributions. In particular, we have $\overline{\partial}_{f+\varphi}^{\ast} \eta^{(2)}=0$, so that outside $\text{supp}(\eta) \cup \text{supp}(\varphi)$, 
$$\overline{\partial}_f^{\ast} \eta^{(1)} = \overline{\partial}_f^{\ast} \eta - \overline{\partial}_f^{\ast} \eta^{(2)}=0.$$
In addition, 
$$\int_{\mathcal{R}_{-1}} |\eta^{(1)}|^2 e^{-(\varphi+f)} \omega_{\ell+1}^n \leq \int_{\mathcal{R}_{-1}} |\eta|^2 e^{-(\varphi+f)} \omega_{\ell+1}^n<\infty.$$
Thus, $\eta^{(1)}$ satisfies the hypotheses of Proposition \ref{prop:localbound}, hence for any $r>0$, we can find $A(r)<\infty$ such that \begin{equation}
    \sup_{\mathcal{R}_{-1}\cap \{f \leq 2r^2\}} |\eta^{(1)}|_{h_{\ell}} \leq A(r).
\end{equation} 
 The Bochner--Kodaira--Nakano formula in Lemma \ref{lem-bochner formula}, combining with the K\"ahler--Ricci shrinker equation then yields
$$(\overline{\partial} \overline{\partial}_{f+\varphi}^{\ast} + \overline{\partial}_{f+\varphi}^{\ast} \overline{\partial})(\chi_{\epsilon} \rho_r \eta^{(1)}) = (\nabla^{0,1})_{f+\varphi}^{\ast} \nabla^{0,1} (\chi_{\epsilon} \rho_r \eta^{(1)}) + (\omega_{\ell+1} + \sqrt{-1}\partial \overline{\partial} \varphi)(\chi_{\epsilon} \rho_r \eta^{(1)}).$$ Here $\chi_{\epsilon}$ is the cut-off function constructed in Lemma \ref{lem:chicutoff} and $\rho_r$ is the cut-off function from \eqref{rhocutoff} that appeared in the proof of Proposition \ref{prop:localbound}.
We can therefore integrate by parts against $\chi_{\epsilon}\rho_r \eta^{(1)}$ to obtain
$$\int_{\mathcal{R}_{-1}} (|\overline{\partial} (\chi_{\epsilon} \rho_r \eta^{(1)})|^2 +|\overline{\partial}_{f+\varphi}^{\ast}(\chi_{\epsilon} \rho_r \eta^{(1)})|^2)e^{-(\varphi+f)}\omega_{\ell+1}^n \geq (1+c)\int_{\mathcal{R}_{-1}} |\chi_{\epsilon} \rho_r \eta^{(1)}|^2 e^{-(\varphi+f)}\omega_{\ell+1}^n.$$
Recalling that $\overline{\partial}_{f+\varphi}^{\ast} \eta^{(1)} = \overline{\partial}_{f+\varphi}^{\ast} \eta$, for any $\sigma>0$, we can estimate
\begin{equation}
    \begin{aligned}
        \int_{\mathcal{R}_{-1}} |\overline{\partial}_{f+\varphi}^{\ast} (\chi_{\epsilon} \rho_r \eta^{(1)})|^2 e^{-(\varphi+f)} \omega_{\ell+1}^n \leq & (1+\sigma)\int_{\mathcal{R}_{-1}} \chi_{\epsilon}^2 \rho_r^2 |\overline{\partial}_{f+\varphi}^{\ast}\eta^{(1)}|^2e^{-(\varphi+f)} \omega_{\ell+1}^n  \\ &+  C(\sigma) \int_{\mathcal{R}_{-1}} (|\nabla \chi_{\epsilon}|_{\omega_{\ell+1}}^2 \rho_r^2+ \chi_{\epsilon}^2 |\nabla \rho_r|_{\omega_{\ell+1}}^2)|\eta^{(1)}|^2 e^{-(\varphi+f)} \omega_{\ell+1}^n \\
\leq & (1+\sigma) \int_{\mathcal{R}_{-1}} |\overline{\partial}_{f+\varphi}^{\ast} \eta|^2 e^{-(\varphi+f)} \omega_{\ell+1}^n\\
&+ C(\sigma)A(r)^2\int_{\{f \leq 2r^2\}}|\nabla \chi_{\epsilon}|^2 e^{-(\varphi+f)}\omega_{\ell+1}^n\\ &+ \frac{C(\sigma)}{r} \int_{\mathcal{R}_{-1}} |\eta|^2 e^{-(\varphi+f)} \omega_{\ell+1}^n,
    \end{aligned}
\end{equation}
and similarly using $\pp \eta^{(1)}=0$, we obtain
\begin{align*} \int_{\mathcal{R}_{-1}} |\overline{\partial} (\chi_{\epsilon} \rho_r \eta^{(1)})|^2 e^{-(\varphi+f)} \omega_{\ell+1}^n \leq&  2A(r)^2\int_{\{f \leq 2r^2\}} |\nabla \chi_{\epsilon}|^2 e^{-(\varphi+f)} \omega_{\ell+1}^n\\
&+ \frac{2}{r} \int_{\mathcal{R}_{-1}} |\eta|^2 e^{-(\varphi+f)} \omega_{\ell+1}^n.
\end{align*}
Combining expressions yields
\begin{align*} \int_{\mathcal{R}_{-1}} \chi_{\epsilon}^2 \rho_r^2 |\eta^{(1)}|^2 e^{-(\varphi+f)} \omega_{\ell+1}^n \leq &(1+\sigma)\int_{\mathcal{R}_{-1}} |\overline{\partial}_{f+\varphi}^{\ast} \eta|^2 e^{-(\varphi+f)} \omega_{\ell+1}^n\\
&+ C(\sigma,\ell)A(r)^2 \int_{\{f\leq 2r^2\}} |\nabla \chi_{\epsilon}|^2 e^{-\varphi}dg_{-1} \\ &+ \frac{C(\sigma)}{r} \int_{\mathcal{R}_{-1}} |\eta|^2 e^{-(\varphi+f)} \omega_{\ell+1}^n.\end{align*}
By the dominated convergence theorem, we may take $\epsilon \searrow 0$, $r\to \infty$, and then $\sigma \to 0$ in order to obtain
$$\int_{\mathcal{R}_{-1}} |\eta^{(1)}|^2 e^{-(\varphi+f)}\omega_{\ell+1}^n \leq \frac{1}{1+c} \int_{\mathcal{R}_{-1}} |\overline{\partial}_{f+\varphi}^{\ast} \eta|^2 e^{-(\varphi+f)} \omega_{\ell+1}^n.$$
Combining this with Cauchy's inequality yields
\begin{align*} \left| \int_{\mathcal{R}_{-1}} \langle \eta , \alpha\rangle e^{-(\varphi+f)}\omega_{\ell+1}^n \right| & = \left| \int_{\mathcal{R}_{-1}} \langle \eta^{(1)}, \alpha\rangle e^{-(\varphi+f)} \omega_{\ell+1}^n \right| \\
&\leq \left( \int_{\mathcal{R}_{-1}} |\eta^{(1)}|^2 e^{-(\varphi+f)} \omega_{\ell+1}^n \right)^{\frac{1}{2}} \left( \int_{\mathcal{R}_{-1}} |\alpha|^2 e^{-(\varphi+f)} \omega_{\ell+1}^n \right)^{\frac{1}{2}} \\ &\leq \left(\frac{1}{1+c}\int_{\mathcal{R}_{-1}} |\overline{\partial}_{f+\varphi}^{\ast} \eta|^2 e^{-(\varphi+f)} \omega_{\ell+1}^n \right)^{\frac{1}{2}} \left( \int_{\mathcal{R}_{-1}} |\alpha|^2 e^{-(\varphi+f)} \omega_{\ell+1}^n \right)^{\frac{1}{2}}.
\end{align*}
It follows that the linear map 
$$\overline{\partial}_{f+\varphi}^{\ast} \mathcal{A}_c^{0,1}(\mathcal{R}_{-1},L^{\ell}) \to \mathbb{C}, \quad \overline{\partial}_{f+\varphi}^{\ast} \eta \mapsto \int_{\mathcal{R}_{-1}} \langle \eta, \alpha \rangle_{h_\ell} e^{-(\varphi+f)} \omega_{\ell+1}^n$$
is well-defined and bounded, so by the Hahn--Banach extension theorem and the Riesz isomorphism theorem, there exists $w \in L^2(\mathcal{R}_{-1},L^{\ell})$ such that $$\int_{\mathcal{R}_{-1}} |w|_{h_{\ell}}^2 e^{-(\varphi+f)} \omega_{\ell+1}^n \leq \frac{1}{1+c}\int_{\mathcal{R}_{-1}} |\alpha|_{\omega_{\ell+1}\otimes h_{\ell}}^2 e^{-(\varphi+f)} \omega_{\ell+1}^n$$ and 
$$\int_{\mathcal{R}_{-1}} \langle  \overline{\partial}_{f+\varphi}^{\ast} \eta, w\rangle_{h_{\ell}} e^{-(\varphi+f)} \omega_{\ell+1}^n = \int_{\mathcal{R}_{-1}} \langle \eta,\alpha \rangle_{\omega_{\ell+1}\otimes h_{\ell}} e^{-(\varphi+f)} \omega_{\ell+1}^n$$
for all $\eta \in \mathcal{A}_c^{0,1}(\mathcal{R}_{-1},L^{\ell})$. In other words, $\overline{\partial} w= \alpha$ in the sense of distributions.
\end{proof}

\subsection{$C^0$ and $C^1$ estimates for holomorphic sections on a singular K\"ahler--Ricci shrinker}
\label{sec:C0C1}
We now establish the $C^0$ and $C^1$ bounds for pluri-anticanonical holomorphic sections in the following lemma. 
Among other ingredients, we will use the improved Kato inequality in Lemma~\ref{lem--kato for holo gradient} and the cut-off function constructed in Lemma~\ref{lem:cutoff}. 
The improved Kato inequality is used to overcome the difficulty that the cut-off functions $\chi_\epsilon$ in Lemma~\ref{lem:chicutoff} have gradients which are only in $L^{4-\sigma}$ rather than $L^4$, while the cut-off function $\varphi$ constructed in Lemma~\ref{lem:cutoff} allows us to localize the estimate.
\begin{lemma} \label{lem:C1estimates}
For any $D<\infty$, there exists $C=C(D,W)<\infty$ such that the
following hold for any $\ell\in\mathbb{N}_{\geq 0}$, $u\in H^{0}(\mathcal{R}_{-1},K_{\mathcal{R}_{-1}}^{-\ell})$,
$x_{0}\in B(p,D)$, and $r\in(0,C^{-1}]$:
\begin{enumerate}
    \item \label{lem:C1estimates1} $\sup_{B_g(x_{0},r) \cap \mathcal{R}_{-1}}|u|^2_{h_{\ell}}\leq\frac{C(\ell r^2+1)^4}{((\ell+1) r^2)^n}\int_{B_g(x_{0},Cr) \cap \mathcal{R}_{-1}}|u|_{h_{\ell}}^{2}\omega_{\ell+1}^{n}$,
    \item \label{lem:C1estimates2} $\sup_{B_g(x_{0},r) \cap \mathcal{R}_{-1}}|\nabla  u|^2_{\omega_{\ell+1}\otimes h_{\ell}}\leq\frac{Ce^{\ell r^2}}{((\ell+1) r^2)^n}\int_{B_g(x_{0},Cr)\cap \mathcal{R}_{-1}}|u|_{h_{\ell}}^{2}\omega_{\ell+1}^{n}.$
\end{enumerate}
\end{lemma}
\begin{remark}
  In applications later in the paper, we will apply the above lemma to the rescaled metric $\omega_{\ell+1}$. 
It is therefore natural to normalize the $L^2$-norm of a holomorphic section to be 1 on the unit ball with respect to $\omega_{\ell+1}$, which corresponds to the choice $r^2=\ell^{-1}$. 
Consequently, the constant on the right-hand side is independent of $\ell$.
\end{remark}
\begin{proof}
We extend $u$ to all of $\mathcal{R}$ by $\mathcal{L}_{\partial_{t}-\frac{1}{2}\nabla f}u=0$.
Then because $\mathcal{L}_{\partial_{t}-\frac{1}{2}\nabla f}h=0$,
we have
\begin{equation}\label{eq--heat equation for |u|^2}
       \mathcal{L}_{\partial_{t}-\frac{1}{2}\nabla f}|u|_{h_{\ell}}^{2}=0
\end{equation}
and 
\begin{equation}\label{eq--lie derivative of nabla}
\mathcal{L}_{\partial_{t}-\frac{1}{2}\nabla f}|\nabla u|_{\omega_{\ell+1}\otimes h_{\ell}}^{2}=\frac{1}{|t|}|\nabla u|_{\omega_{\ell+1}\otimes h_{\ell}}^{2}.
\end{equation}
Fix an $r_0$, depending only on $D$ and $W$, such that Lemma \ref{lem:cutoff} holds and for $r\leq r_0$, let $\varphi\in C^{\infty}(\mathcal{R}_{-1})$ be the cut-off function constructed there
and let $\chi_{\epsilon}\in C^{\infty}(\mathcal{R}_{-1})$ be as in
Lemma \ref{lem:chicutoff}. For $x\in \mathcal{R}_{-1}$ and $t<0$, we denote by $x(t)$ the intersection point of $\mathcal{R}_t$ with the integral curve of $\partial_t - \frac{1}{2}\nabla f$ through $x$. Write $t':=-1+r^2$ and $x':=x(t')$ for $x\in B(x_0,r/2)\cap \mathcal{R}_{-1}$. 

 \ref{lem:C1estimates1}. Let $v=|u|_{h_{\ell}}^{\frac14}$. Then using $|\nabla^{1,0}|u|^2|^2=|\nabla u|^2|u|^2$ and \eqref{eq:bochnerlemma1}, we have
\begin{equation}
       (\mathcal{L}_{\partial_{t}-\frac{1}{2}\nabla f}-\Delta)v^{2}\leq -|\nabla v|^2+\frac{nl}{4\tau}v^2.
\end{equation}  For any $t\in[-1,t')$, we set $\tau=|t|$. For any $\phi \in C_c^{\infty}(\mathcal{R}_{[-1,t']})$, we can compute

 \begin{align*} \frac{d}{dt} \int_{\mathcal{R}_t} \phi d\nu_{x';t} & = \int_{\mathcal{R}_t} (\partial_t - \Delta)\phi d\nu_{x';t} = \int_{\mathcal{R}_t} (\partial_t - \Delta)\phi d\nu_{x';t} + \int_{\mathcal{R}_t} \mathcal{L}_{-\frac{1}{2}\nabla f} \left( \phi K(x';\cdot)\frac{\omega_t^n}{n!}\right)\\
 &= \int_{\mathcal{R}_t} (\mathcal{L}_{\partial_t - \frac{1}{2}\nabla f}-\Delta)\phi d\nu_{x';t} - \int_{\mathcal{R}_t} \phi \left( \mathcal{L}_{\frac{1}{2}\nabla f} \log K(x';\cdot) + \Delta f\right)d\nu_{x';t}.
 \end{align*}
Taking $\phi = v^2 \chi_{\epsilon}^2 \varphi^2$, it follows that $I(t):=\int_{\mathcal{R}_t} v^2 \chi_{\epsilon}^2 \varphi^2 d\nu_{x';t}$ satisfies
\begin{equation}\label{eq--time derivative}
    \begin{aligned}
        I'(t)=&\int_{\mathcal{R}_{t}}(\mathcal{L}_{\partial_{t}-\frac{1}{2}\nabla f}-\Delta)(v^{2}\chi_{\epsilon}^{2}\varphi^{2})d\nu_{x';t}\\
&-\int_{\mathcal{R}_{t}}(v^{2}\chi_{\epsilon}^{2}\varphi^{2})(\Delta f + \frac{1}{2} \langle \nabla f,\nabla \log K(x';\cdot)\rangle )d\nu_{x';t}.
    \end{aligned}
\end{equation}
Then, using \eqref{eq--heat equation for |u|^2}, the fact that $\nabla f$ is locally bounded, and integrating by parts the term involving $\nabla^2 f$, we can estimate
\begin{align*}
 I'(t)\leq & \int_{\mathcal{R}_{t}}\chi_{\epsilon}^{2}(y)\varphi^{2}(y)\left(\frac{ n\ell}{\tau}v^{2}-|\nabla v|^{2}\right)(y)d\nu_{x';t}(y)\\
 & +C\int_{\mathcal{R}_{t}}(|\nabla v|\chi_{\epsilon}\varphi)(y)\left( |v|\chi_{\epsilon}|\nabla\varphi|+|v|\varphi|\nabla\chi_{\epsilon}| \right)(y)d\nu_{x';t}(y) - \int_{\mathcal{R}_t} (v^2 \varphi^2 \Delta \chi_{\epsilon}^2)(y)d\nu_{x';t}(y)\\
 & +C \int_{\mathcal{R}_{t}}v^{2}(y)\left(\chi_{\epsilon}^{2}\varphi|\Delta\varphi|+\chi_{\epsilon}^{2}|\nabla\varphi|^{2}+\chi_{\epsilon}\varphi |\nabla\chi_{\epsilon}| \cdot |\nabla\varphi|\right)(y)d\nu_{x';t}(y)\\
 & + C(D) \int_{\mathcal{R}_t} (v^2 \chi_{\epsilon}^2 \varphi^2)(y)   (1+|\nabla \log K(x';\cdot)|(y))d\nu_{x';t}(y).
\end{align*}

Then, using the basic properties of cut-off functions, applying Cauchy--Schwarz, 
and integration by parts to $\Delta \chi_\epsilon^2$, we can continue the estimate by
\begin{equation}\label{eq-time derivative many terms}
\begin{aligned}
     I'(t) \leq & -\frac{1}{2}\int_{\mathcal{R}_{t}}(\chi_{\epsilon}^{2}\varphi^{2}|\nabla v|^{2})d\nu_{x';t}+C\frac{\ell}{\tau} \int_{\mathcal{R}_t}\chi_{\epsilon}^2 \varphi^2 v^2 d\nu_{x';t}+C\int_{\mathcal{R}_{t}}(|\nabla \varphi|^2\chi_{\epsilon}^2+|\nabla\chi_{\epsilon}|^2\varphi^2)v^{2}d\nu_{x';t}\\
      &+(\frac{n\ell}{\tau}+C(D)r^{-2})\int_{\mathcal{R}_{t}\cap \supp(\nabla\varphi)}v^2\chi_\epsilon^2d\nu_{x';t}+C(D)\int_{\mathcal{R}_t} (v^2 \chi_{\epsilon}^2 \varphi^2) |\nabla \log K(x';\cdot)| d\nu_{x';t}\\
 & +2\int_{\mathcal{R}_{t}}\langle\nabla\chi_{\epsilon},\nabla(v^{2}\chi_{\epsilon}\varphi^{2}K(x';\cdot))\rangle dg_{t}\\
\leq & -\frac{1}{4}\int_{\mathcal{R}_{t}}(\chi_{\epsilon}^{2}\varphi^{2}|\nabla v|^{2})d\nu_{x';t} + C(D)\int_{\mathcal{R}_t} (v^2 \chi_{\epsilon}^2 \varphi^2) |\nabla \log K(x';\cdot)| d\nu_{x';t}\\
&+C(D)\left( \ell+\frac{1}{r^2}\right)\left(\int_{\text{supp}(\nabla \varphi)\cap\mathcal{R}_{t}}v^{2}d\nu_{x';t}+\int_{\text{supp}(\varphi\nabla \chi_{\epsilon})\cap\mathcal{R}_{t}}v^{2}d\nu_{x';t}\right)\\
 & +C\int_{\text{supp}(\nabla\chi_{\epsilon})\cap\mathcal{R}_{t}}\left(|\nabla\chi_{\epsilon}|^{2}+|\nabla\log K(x';\cdot)|^{2}\right)v^{2}\varphi^{2}d\nu_{x';t} \\ 
\end{aligned}
\end{equation}
Writing $\tau' :=t'-t$, we can use Perelman's differential Harnack inequality and Lemma \ref{lem--integration by part trick} as in the proof of Proposition \ref{prop:localbound} to obtain 
\begin{equation}\label{eq--intermediate step}
\begin{aligned}
     C(D)\int_{\mathcal{R}_t} (v^2 \chi_{\epsilon}^2 \varphi^2) |\nabla \log K(x';\cdot)| d\nu_{x';t} \hspace{-50 mm}& \\\leq & \frac{1}{4}\int_{\mathcal{R}_{t}}(\chi_{\epsilon}^{2}\varphi^{2}|\nabla v|^{2})d\nu_{x';t}+\frac{C}{\sqrt{\tau'}}\int_{\mathcal{R}_t} v^2 \chi_{\epsilon}^2 \varphi^2 d\nu_{x';t} \\
     & +C\sqrt{\tau'}\int_{ \text{supp}(\nabla \chi_{\epsilon})\cap \mathcal{R}_t}  (|\nabla \log K(x';\cdot)|^2+|\nabla \chi_{\epsilon}|^2)v^2\varphi^2 d\nu_{x';t}\\
     &+\int |\nabla \chi_{\epsilon}|^2 v^2\varphi^2 d\nu_{x';t}+\sqrt{\tau'}\int_{\operatorname{supp}(\nabla \varphi)\cap \mathcal{R}_t} (|\nabla \log K(x';\cdot)|^2+|\nabla \varphi|^2)v^2 d\nu_{x';t}.
\end{aligned}
\end{equation}
Combining \eqref{eq-time derivative many terms} and \eqref{eq--intermediate step} and using Cauchy's inequality yields
\begin{equation}\label{eq-upper bound of derivative 5terms}
    \begin{aligned}
        I'(t) \leq & \frac{C}{\sqrt{\tau'}} I(t) + C \left( \ell+\frac{1}{r^2}\right)\int_{ \supp ({\nabla \varphi})\cap\mathcal{R}_{t}} v^2 d\nu_{x';t}\\ 
&+ C\int_{ \text{supp}(\nabla \chi_{\epsilon})\cap \mathcal{R}_t}  |\nabla \log K(x';\cdot)|^2 v^2\varphi^2 d\nu_{x';t}\\
&+ C\int_{\mathcal{R}_t} |\nabla \chi_{\epsilon}|^2 v^2\varphi^2 d\nu_{x';t}+C \left(\ell+\frac{1}{r^2}\right)\int_{\text{supp}(\nabla\chi_{\epsilon})\cap\mathcal{R}_{t}} v^2 d\nu_{x';t}\\&+\sqrt{\tau'}\left( \int_{\operatorname{supp}(\nabla \varphi)\cap \mathcal{R}_t} |\nabla \log K(x';\cdot)|^3 d\nu_{x';t}\right)^{\frac{2}{3}}\left( \int_{\operatorname{supp}(\nabla \varphi)\cap \mathcal{R}_t} v^6 d\nu_{x';t} \right)^{\frac{1}{3}}.
    \end{aligned}
\end{equation}
For $\epsilon=\epsilon(x)>0$ sufficiently small, we have for $t\in [-1, t']$,
$$d_t(\text{supp}(\nabla\chi_{\epsilon}),x(t))\geq c(x),$$ which implies an upper bound on $K(x';\cdot)$ on $\supp(\nabla \chi_{\epsilon})$. As in \eqref{eq:pointwiseconjugatebound}, we can use Lemma \ref{lem:conjugategradient} to obtain 
\begin{equation}
    \int_{\supp(\nabla \chi_\epsilon)\cap \mathcal R_t}|\nabla \log K(x';\cdot)|^{2q}d\nu_{x';t}\leq \frac{C(W,x,q)}{(\tau')^{q}}
\end{equation}
for any $q\in [1,2)$. Similarly, by the construction of $\varphi$ and the fact that $x \in B(x_0, r/2)$ and $t\in [-1,t']$,
\begin{equation}
d_t\big(\operatorname{supp}(\nabla \varphi), x(t)\big) \geq C^{-1} r,
\end{equation}
allowing us to estimate
\begin{equation*}
    \int_{\supp(\nabla \varphi)\cap \mathcal{R}_t} |\nabla \log K(x';\cdot)|^3 d\nu_{x';t} \leq \frac{C(W,D)}{(\tau')^{\frac{3}{2}}}.
\end{equation*}
Applying H\"older's inequality to the third and fourth terms on the right-hand side of \eqref{eq-upper bound of derivative 5terms}, we obtain that there exists a constant $C$, independent of $\epsilon$, such that 
\begin{equation}
 \begin{aligned}
   \frac{d}{dt}\left(e^{C\sqrt{t'-t}}\int_{\mathcal{R}_{t}}v^{2}\chi_{\epsilon}^{2}\varphi^{2}d\nu_{x';t}\right)\leq &  C \left(\ell+\frac{1}{r^2}\right)\int_{ \supp ({\nabla \varphi})\cap\mathcal{R}_{t}} v^2 d\nu_{x';t}+\Psi(\epsilon|W,x)\\
   &+ \frac{C}{\sqrt{\tau'}}\left( \int_{\operatorname{supp}(\nabla \varphi)\cap \mathcal{R}_t} v^6 d\nu_{x';t} \right)^{\frac{1}{3}},
\end{aligned}
\end{equation}
where $\Psi(\epsilon|W,x)$ is a quantity satisfying $\lim_{\epsilon \to 0} \Psi(\epsilon|W,x)=0$ for any fixed $W,x$. Thus, integrating $t$ from $-1$ to $t'=-1+r^2$ and taking $\epsilon\searrow0$
gives
\begin{equation}\label{eq--conditional upper bound}
    \begin{aligned}
   v^2(x)=  v^{2}(x')\leq &e^{Cr}\int_{\text{supp}(\varphi)\cap\mathcal{R}_{-1}}v^{2}(y)d\nu_{x';-1}(y)\\
&+Ce^{Cr}(\ell r^2+1)\sup_{t\in [-1,t']}\int_{\supp(\nabla \varphi)\cap \mathcal{R}_t} v^2d\nu_{x';t}\\
&+Cre^{Cr}\sup_{t\in [-1,t']}\left( \int_{\operatorname{supp}(\nabla \varphi)\cap \mathcal{R}_t} v^6 d\nu_{x';t} \right)^{\frac{1}{3}}.
    \end{aligned}
\end{equation}
Then, using the upper bound for the conjugate heat kernel
\cite[Theorem 7.2]{bamler1} and arguing as in the proof of
\eqref{holder control beta-p norm}, noting that \(v=|u|^{1/4}\), we bound the last two
terms in \eqref{eq--conditional upper bound} from above by the \(L^2\)-integral
of \(|u|\). It follows that 
\begin{equation}
    |u|^2(x)\leq \frac{C(\ell r^2+1)^4}{r^{2n}}\int_{B(x_{0},Cr)}|u|_{h_{\ell}}^{2}\omega_{-1}^{n}.
\end{equation}

\ref{lem:C1estimates2}. In the proof below, we temporarily use $\omega_{-1}\otimes h^\ell$ to compute the norm of $\nabla u$. By Lemma \ref{lem:cutoff}, there exists $\widetilde{\varphi} \in C^{\infty}(\mathcal{R}_{-1})\cap C^0(\mathcal{X}_{-1})$ such that $\widetilde{\varphi}|_{B(x_0,r)} \equiv 1$, $\operatorname{supp}(\widetilde{\varphi})\subseteq B(x_0,C(D,W)r)$, and $r^2(|\nabla \widetilde{\varphi}|^2+|\Delta\widetilde{\varphi}|)\leq C(D,W)$. As before, extend to $\mathcal{R}_{-1}$ by $\mathcal{L}_{\partial_t-\frac{1}{2}\nabla f}\widetilde{\varphi}=0$. Because $|u|$ is locally bounded and $\mathcal{L}_{\partial_{t}-\frac{1}{2}\nabla f}((4\pi\tau)^{-n}e^{-f}\omega_{t}^{n})=0$,
we can estimate 
\begin{align*}
\frac{d}{dt}\int_{\mathcal{R}_{t}}|u|^{2}\widetilde{\varphi}^2\chi_{\epsilon}^{2}d\mu_{t}= & \int_{\mathcal{R}_{t}}(\mathcal{L}_{\partial_{t}-\frac{1}{2}\nabla f}-\Delta)(|u|^{2}\widetilde{\varphi}^{2}\chi_{\epsilon}^{2})d\mu_{t}\\
&-\int_{\mathcal{R}_t} |u|^2 \widetilde{\varphi}^2 \chi_{\epsilon}^2 ( \Delta f - \frac{1}{2} |\nabla f|^2 )(y)d\mu_t
\\
\leq & \int_{\mathcal{R}_{t}}\left(-|\nabla u|^{2}+\frac{n\ell}{\tau}|u|^{2}\right)\widetilde{\varphi}^{2}\chi_{\epsilon}^{2}d\nu_{t} \\
 & +C\int_{\mathcal{R}_{t}}|u|^{2}\left(\chi_{\epsilon}\widetilde{\varphi}^{2}\Delta\chi_{\epsilon}+\widetilde{\varphi}^{2}|\nabla\chi_{\epsilon}|^{2}+\frac{1}{r^{2}}\right)d\mu_{t}\\
 &+\frac{1}{\tau}\int_{\mathcal{R}_t} |u|^2 \widetilde{\varphi}^2 \chi_{\epsilon}^2 (f-n-W)d\mu_t\\
\leq & -\frac{1}{2}\int_{\mathcal{R}_{t}}|\nabla u|^{2}\widetilde{\varphi}^2d\mu_{t}+\left(\frac{n\ell}{\tau}+C(r')\right)\int_{\mathcal{R}_{t}}|u|^{2}\widetilde{\varphi}^{2}\chi_{\epsilon}^{2}d\mu_{t}\\
 & +C\int_{\text{supp}(\chi_{\epsilon})\cap\mathcal{R}_{t}}\widetilde{\varphi}^{2}|u|^{2}(|\nabla\chi_{\epsilon}|^{2}+|\nabla f|^{2})d\mu_{t}.
\end{align*}
By integrating from $t=-1$ to $t'>-1$ and taking $\epsilon\searrow 0$,
we obtain that  
\begin{equation*}
\int_{-1}^{t'}\int_{\mathcal{R}_{t} \cap B(x_0,r)}|\nabla u|^{2} d\mu_{t}dt\leq (C\ell+C(D,W))\int_{\mathcal{R}_{-1}\cap B(x_0,2Cr)}|u|^{2}d\mu_{-1}.
\end{equation*}

Next, using that $u$ is a holomorphic section, by Lemma \ref{lem--bochner for gradient} and the shrinker equation, we have
\begin{equation}
    \begin{aligned}
      \Delta|\nabla u|^{2}&= |\nabla^{1,0}\nabla u|^{2}+\Ric(\nabla u,\nabla u)-\frac{(n+2)\ell}{\tau} |\nabla u|^{2}+\frac{n\ell^2}{\tau^2}|u|^2\\
    &= |\nabla^{1,0}\nabla u|^{2}+\frac{1}{\tau}(1-\ell(n+2))|\nabla u|^{2}- h^{\ell} g^{\overline{\ell}i}g^{\overline{\jmath}k} \nabla_i \nabla_{\overline{\jmath}} f \nabla_k u \overline{\nabla_{\ell}u}+\frac{n\ell^2}{\tau^2}|u|^2.
    \end{aligned}
\end{equation}
Then we use the improved Kato inequality \eqref{eq--improved kato for nablau}
and use Cauchy's inequality to estimate
\begin{equation}\label{eq--laplacian}
    \begin{aligned}
         \Delta |\nabla u|^{\beta} =& \frac{\beta}{2} \left( |\nabla u|^{\beta-2}\Delta |\nabla u|^2 -  \frac{2-\beta}{2} |\nabla u|^{\beta-4}|\nabla^{1,0} |\nabla u|^2|^2  \right) \\
\geq & \frac{\beta}{2}\left( 1-\frac{2-\beta}{2}(1+\gamma)  \right) |\nabla^{1,0} \nabla u|^2 |\nabla u|^{\beta -2} +\frac{\beta}{2\tau}(1-\ell(n+2))|\nabla u|^{\beta} \\ &- \frac{\beta}{2}|\nabla u|^{\beta-2} h^\ell g^{\overline{l}i}g^{\overline{\jmath}k} \nabla_i \nabla_{\overline{\jmath}} f \nabla_k u \overline{\nabla_{l}u} \\& +\frac{\beta \ell^2}{2\tau^2}\left(n- \frac{2-\beta}{2}(1+\gamma^{-1})\right) |u|^2 |\nabla u|^{\beta-2}.
    \end{aligned}
\end{equation}
For any choice of $\beta \in \left(\frac{2}{n+1},1\right)$, we have
\(
\frac{2-\beta}{2(n-1)+\beta} < \frac{\beta}{2-\beta}.
\)
Thus, if we set
\[
\gamma := \frac{1}{2}\left(\frac{2-\beta}{2(n-1)+\beta} + \frac{\beta}{2-\beta}\right),
\]
then the coefficients of the first and last terms in \eqref{eq--laplacian} are both positive.  We therefore obtain the existence of $C_{\beta}<\infty$ such that $v:= |\nabla u|^{\frac{\beta}{2}}$ satisfies 
\begin{equation}\label{eq--bound on lapv}
    \begin{aligned}
        \Delta v^2 \geq &\frac{1}{C_{\beta}}|\nabla v|^2 + \frac{1}{C_{\beta}} |\nabla^{1,0} \nabla u|^{2} |\nabla u|^{\beta-2}\\& +\frac{\beta}{2\tau}(1-\ell(n+2))v^2- \frac{\beta}{2}|\nabla u|^{\beta-2} h^{\ell} g^{\overline{l}i}g^{\overline{\jmath}k} \nabla_i \nabla_{\overline{\jmath}} f \nabla_k u \overline{\nabla_{l}u}.
    \end{aligned}
\end{equation}
Then by \eqref{eq--time derivative}, $J(t):= \int_{\mathcal{R}_t} v^2 \chi_{\epsilon}^2 \varphi^2 d\nu_{x';t}$ satisfies 
\begin{align*} J'(t)
= & \int_{\mathcal{R}_t} (\partial_{t} - \mathcal{L}_{\frac{1}{2}\nabla f}-\Delta)(v^2\chi_{\epsilon}^2 \varphi^2) d\nu_{x';t} \\ &- \int_{\mathcal{R}_t} (v^2 \chi_{\epsilon}^2 \varphi^2)\left( \Delta f +\frac{1}{2} \langle \nabla f,\nabla \log K(x';\cdot)\rangle \right)d\nu_{x';t}.
\end{align*}
Using \eqref{eq--lie derivative of nabla}, \eqref{eq--bound on lapv} and the fact that $\nabla f$ is locally bounded, and integrating by parts the term involving $\Delta f$, we can bound 
\begin{equation}
    \begin{aligned}
       J'(t)\leq  & -\frac{1}{C_{\beta}} \int_{\mathcal{R}_t} (|\nabla^{1,0} \nabla u|^{2}|\nabla u|^{\beta-2}+|\nabla v|^2) \chi_{\epsilon}^2 \varphi^2 d\nu_{x';t}+\frac{\beta \ell}{2\tau}(n+1)J(t)\\
&+ \frac{\beta}{2}\int_{\mathcal{R}_t} \left( \chi_{\epsilon}^2 \varphi^2|\nabla ^h u|^{\beta-2} h g^{\overline{l}i}g^{\overline{\jmath}k} \nabla_i \nabla_{\overline{\jmath}}f \nabla_k u \overline{\nabla_{l}u} \right) d\nu_{x';t} \\
&+ C(D) \int_{\mathcal{R}_t} (v^2 \chi_{\epsilon}^2 \varphi^2)(1+|\nabla \log K(x';\cdot)|)d\nu_{x';t}\\
& +C\int_{\mathcal{R}_{t}}(|\nabla v|\chi_{\epsilon}\varphi)\left( |v|\chi_{\epsilon}|\nabla\varphi|+|v|\varphi|\nabla\chi_{\epsilon}| \right)d\nu_{x';t}\\& - \int_{\mathcal{R}_t} (v^2 \varphi^2 \Delta \chi_{\epsilon}^2)d\nu_{x';t}\\
 & +C \int_{\mathcal{R}_{t}}v^{2}\left(\chi_{\epsilon}^{2}\varphi\Delta\varphi+\chi_{\epsilon}^{2}|\nabla\varphi|^{2}+\chi_{\epsilon}\varphi |\nabla\chi_{\epsilon}| \cdot |\nabla\varphi|\right)d\nu_{x';t}
    \end{aligned}
\end{equation}
We integrate by parts to obtain
\begin{align*}
    \frac{\beta}{2}\int_{\mathcal{R}_t} & \left( \chi_{\epsilon}^2 \varphi^2|\nabla ^h u|^{\beta-2} h g^{\overline{l}i}g^{\overline{\jmath}k} \nabla_i \nabla_{\overline{\jmath}}f \nabla_k u \overline{\nabla_{l}u} \right)(y) d\nu_{x';t}(y) \\ \leq & C(D) \int_{\mathcal{R}_t} |\nabla u|^{\beta-1}(y) \left( |\nabla u|(\chi_{\epsilon}^2 \varphi|\nabla \varphi| + \varphi^2 \chi_{\epsilon} |\nabla \chi_{\epsilon}|) + \chi_{\epsilon}^2 \varphi^2 |\nabla \nabla u| \right)(y) d\nu_{x';t}(y) \\
    &+C(D) \int_{\mathcal{R}_t} (v^2 \chi_{\epsilon}^2 \varphi^2)(y)(1+|\nabla \log K(x';\cdot)|)(y)d\nu_{x';t}(y).
\end{align*}
From here, we combine expressions and proceed as in \ref{lem:C1estimates1}.
\end{proof}

\section{Proof of Theorem \ref{thm:complexstructure}}
\label{sec:Thm1Pf}

We continue to assume throughout this section that $(X,d)$ is a singular K\"ahler-Ricci shrinker obtained as an $\mathbb{F}$-limit of noncollapsed K\"ahler-Ricci flows, as in Section \ref{sec;convergence setting}(I).

\subsection{Normal complex variety structure}
\label{sec:normal}

After establishing the H\"ormander $L^2$ estimate for pluri-anticanonical sections (Proposition~\ref{prop:hormander}) and the $C^1$ estimate for holomorphic pluri-anticanonical sections (Lemma~\ref{lem:C1estimates}), the normal complex variety structure will follow from the same argument as in \cite{DS2}. See also \cite{LiuSz1,hallgrenKahlerRicciTangentFlows2023} for closely related work. 

Throughout this section, we continue to use the notation $L=K_{\mathcal{R}_{-1}}^{-1}$, $h=(2\pi)^{-n}e^{-f}\frac{\omega^n}{n!}$, $h_{\ell} = h^{\otimes \ell}$, $\omega_\ell = \ell \omega_{-1}$. 

\begin{lemma} \label{lem:perturbation}
    For any $D\geq 1$ and $\epsilon>0$, there exist $\zeta=\zeta(D,W,\epsilon)>0$ and $A_0=A_0(D,W,\epsilon) \geq 1$ such that the following holds. Suppose $x_0 \in B(p,D)$, $\ell \in \mathbb{N}$ is positive, $U \subseteq \mathcal{X}_{-1}$ is a neighborhood of $x_0$, and $v\in C_c^{\infty}(\mathcal{R}_{-1},L^{\ell})$ satisfies the following for some $A\geq A_0$ with $B_{g_{\ell}}(x_0,2A)\subseteq B_g(p,2D)$:
\begin{enumerate}
    \item \label{lem:perturbation1} $\int_{\mathcal{R}_{-1}} |v|_{h_{\ell}}^2 \frac{\omega_{\ell}^n}{n!} < (1+\zeta)(2\pi)^n$,

    \item \label{lem:perturbation2} $\sup_{U} \left| e^{-\frac{1}{2}d_{g_{\ell}}^2(x_0,\cdot)} - |v|_{h_{\ell}}^2 \right| < \zeta$,

    \item \label{lem:perturbation3} $\int_{\mathcal{R}_{-1}} |\overline{\partial}v|_{\omega_{\ell}\otimes h_{\ell}}^2 \omega_{\ell}^n < \zeta$,

    \item \label{lem:perturbation4} $\sup_U |\overline{\partial} v|_{\omega_{\ell} \otimes h_{\ell}}^2 < \zeta$,

    \item \label{lem:perturbation5} $B_{g_{\ell}}(x_0,2A) \cap \{ r_{\operatorname{Rm}}^{g_{\ell}} \geq \zeta\} \subseteq U$,

    \item \label{lem:perturbation6} $\operatorname{supp}(v)\subseteq B_{g_{\ell}}(x_0,A)$,
\end{enumerate}
Then there exists $u\in H^0(\mathcal{R}_{-1},L^{\ell})$ satisfying the following:

\begin{enumerate}[label=(\alph*)]
    \item $\int_{B_{g}(p,4D)\cap \mathcal{R}_{-1}} |u|_{h_{\ell}}^2 \frac{\omega_{\ell}^n}{n!} < (1+\epsilon)(2\pi)^n$,

    \item $\sup_{B_{g}(p,4D) \cap \mathcal{R}_{-1}} \left| e^{-\frac{1}{2}d_{g_{\ell}}^2(x_0,\cdot)} - |u|_{h_{\ell}}^2 \right|<\epsilon.$
\end{enumerate}
\end{lemma}
\begin{proof}
We may apply Proposition \ref{prop:hormander} with $\varphi=0$ and $c=0$, using assumption \ref{lem:perturbation3}, and the inequalities
\begin{equation*}
    \inf_{\mathcal{R}_{-1}}f \geq W, \qquad \sup_{B_g(p,4D)} f \leq C(D,W)
\end{equation*}
in order to conclude the existence of an $L_{\operatorname{loc}}^2$ section $w$ of $L^\ell$ satisfying $\overline{\partial} w = \overline{\partial} v$ and
\begin{equation}\label{eq:wsmallinL2}
\begin{aligned}
\int_{B_{g}(p,4D)\cap \mathcal{R}_{-1}} |w|_{h_{\ell}}^2 \,\omega_{\ell}^n 
\leq & C(D,W)\int_{\mathcal{R}_{-1}} |w|_{h_{\ell}}^2 e^{-f}\omega_{\ell+1}^n
\\ \leq &C(D,W)\int_{\mathcal{R}_{-1}} |\overline{\partial} v|_{\omega_{\ell+1}\otimes h_{\ell}}^2 e^{-f}\,\omega_{\ell+1}^n
\\ \leq &
C(D,W)\zeta.
\end{aligned}
\end{equation}
By \eqref{eq:wsmallinL2} and \ref{lem:perturbation1}, $u:=v-w$ satisfies $\overline{\partial} u=0$ on $\mathcal{R}_{-1}$ and 
\begin{equation} \label{eq:L2boundforu}
    \int_{B_g(p,4D) \cap \mathcal{R}_{-1}} |u|_{h_{\ell}}^2 \frac{\omega_{\ell}^n}{n!} \leq (1+C(W,D)\zeta)(2\pi)^n.
\end{equation}
Take $\epsilon'>0$ to be determined. As a consequence of \cite[Theorem 2.30]{bamler3} (see also \cite[Proposition 3.11]{CHM25}), for any $z\in B_{g}(p,4D) \cap \mathcal{R}_{-1}$, there exists $z' \in B_{g_{\ell}}(z,C(W,D)(\epsilon')^{\frac{1}{4n}})$ such that $r_{\operatorname{Rm}}^{g_{\ell}}(z')\geq 2\epsilon'$. 

First suppose $z\in B_{g_{\ell}}(x_0,\frac{3}{2}A)\cap \mathcal{R}_{-1}$, so that $B_{g_{\ell}}(z',\epsilon') \subseteq U \cap B(p,4D)$ by \ref{lem:perturbation5}. By \ref{lem:perturbation4}, \eqref{eq:wsmallinL2}, and local elliptic regularity, we then have
\begin{equation} \label{eq:wsmallonU}
    |w|_{h_{\ell}}(z') \leq \Psi(\zeta|\epsilon',D,W),
\end{equation}
where $\lim_{\zeta \to 0}\Psi(\zeta|\epsilon',D,W)=0$ for any fixed $\epsilon',D,W$.
We now apply Lemma \ref{lem:C1estimates}, replacing $x_0$ with $z$ and $r$ with $4C(W,D)\ell^{-\frac{1}{2}}$ to obtain
\begin{equation} \label{eq:C1estimateappliedtou} \sup_{B_{g_{\ell}}(z,4C(W,D)(\epsilon')^{\frac{1}{4n}})} |\nabla|u|_{h_{\ell}}|_{\omega_{\ell}}\leq 2\sup_{B_{g}(z,4C(W,D)\ell^{-\frac{1}{2}})} |\nabla u|_{\omega_{\ell+1} \otimes h_{\ell}} \leq C(D,W).
\end{equation}
 Integrating \eqref{eq:C1estimateappliedtou} along a path from $z$ to $z'$, applying \eqref{eq:wsmallonU}, and combining with \ref{lem:perturbation2} yields
\begin{align} 
\begin{split} \label{eq:bddubdddistance}
\left| |u|_{h_{\ell}}(z) -e^{-\frac{1}{4}d_{g_{\ell}}^2(x_0,z)} \right|&\leq \left| |u|_{h_{\ell}}(z') -e^{-\frac{1}{4}d_{g_{\ell}}^2(x_0,z')} \right|+C(D,W)(\epsilon')^{\frac{1}{4n}}  \\ &\leq \left| |v|_{h_{\ell}}(z') -e^{-\frac{1}{4}d_{g_{\ell}}^2(x_0,z')} \right|+\Psi(\epsilon'|D,W)+\Psi(\zeta|\epsilon',D,W)\\
&\leq \Psi(\epsilon'|D,W,A)+\Psi(\zeta|\epsilon',D,W).
\end{split}
\end{align}
By \eqref{eq:L2boundforu} and \eqref{eq:bddubdddistance}, we can thus choose $\epsilon'=\epsilon'(\epsilon,D,W,A)$ so that
\begin{equation*}
    \sup_{B_{g_{\ell}}(x_0,\frac{3}{2}A)}\left| e^{-\frac{1}{2}d_{g_{\ell}}^2(x_0,\cdot)} - |u|_{h_{\ell}}^2 \right|<\epsilon
\end{equation*}
if $\zeta = \zeta(\epsilon,D,W,A)>0$ sufficiently small. 

Suppose instead $z\in (B_g(p,4D)\cap \mathcal{R}_{-1})\setminus B_{g_{\ell}}(x_0,\frac{3}{2}A)$, so that $v=0$ on $B_{g_{\ell}}(z,\frac{1}{2}A)$ by \ref{lem:perturbation6}. We can then apply Lemma \ref{lem:C1estimates}\ref{lem:C1estimates1}, replacing $u$ with $w$, $x_0$ with $z$, and $r$ with $\ell^{-\frac{1}{2}}$, to obtain 
\begin{equation*}
    |u|_{h_{\ell}}(z)=|w|_{h_{\ell}}(z)\leq C(D,W)\zeta.
\end{equation*}
Because $e^{-\frac{1}{2}d_{g_{\ell}}^2(x_0,z)}\leq e^{-\frac{1}{2}A^2}$, the claim follows if we also choose $A_0=A_0(\epsilon)$ sufficiently large.
\end{proof}

\begin{lemma} \label{lem:peak}
    For any $D <\infty$, $x_0 \in B_g(p,D)$, and $\epsilon>0$, there exists $\ell \in \mathbb{N}$ with $\ell \geq \epsilon^{-1}$ and some $u \in H^0 (\mathcal{R}_{-1},L^{\ell})$ satisfying the following:

\begin{enumerate}
    \item $\sup_{x\in B_{g}(p,2D)\cap \mathcal{R}_{-1}} \left| |u|_{h_{\ell}}^2(x) - e^{-\frac{1}{2}d_{g_{\ell}}^2(x_0,x)} \right|<\epsilon$,
    \item $\int_{B_{g}(p,2D)\cap \mathcal{R}_{-1}} |u|_{h_{\ell}}^2 \frac{\omega_{\ell}^n}{n!} <(1+\epsilon)(2\pi)^n $.
\end{enumerate}
In particular, there exists $\ell_0, r_0>0$ and a holomorphic section $u\in H^0(\mathcal{R}_{-1}, L^{\ell_0})$  such that 
\begin{equation*}
    \inf_{B(x_0,r_0)}|u|_{h_{\ell_0}}^2 \geq \frac{1}{2}.
\end{equation*}
\end{lemma}
\begin{proof} Suppose by way of contradiction that there exist $x_0 \in \mathcal{X}_{-1}$, $D<\infty$, and $\epsilon>0$ such that sufficiently large $\ell \in \mathbb{N}$ do not satisfy the claim. By \cite{bamler3}, there exist $\ell_i \to \infty$ such that $\ell_i \in \mathbb{N}$, 
\begin{equation*} (\mathcal{X}_{-1},\ell_i^{\frac{1}{2}} d_{-1},x_0) \to (C(Y),d_{C(Y)},x_{\ast})\end{equation*}
in the pointed Gromov--Hausdorff sense, where $(C(Y),d_{C(Y)})$ is a metric cone with vertex $x_{\ast}$. Moreover, the regular part of $C(Y)$ has the structure of a K\"ahler manifold $(\mathcal{R}_{C(Y)},J_{C(Y)},\omega_{C(Y)})$ equipped with a holomorphic Hermitian line bundle $(L_{C(Y)},h_{C(Y)})$ satisfying $\Theta_{h_{C(Y)}}=\omega_{C(Y)}$ and
\begin{align*}
   & (\mathcal{X}_{-1}, \ell_i^{\frac{1}{2}} d_{-1},x_0,\mathcal{R}_{-1},\ell_i\omega_{-1},J,K_{\mathcal{R}_{-1}}^{-\ell_i},h^{\ell_i})\longrightarrow \\
& \hspace{10 mm} (C(Y),d_{C(Y)},x_{\ast},\mathcal{R}_{C(Y)},\omega_{C(Y)},J_{C(Y)},L_{C(Y)},h_{C(Y)})
\end{align*}
in the sense of pointed polarized K\"ahler singular spaces (as in \cite[Section 2.1]{DS1} and \cite[Definition 5.5]{hallgrenKahlerRicciTangentFlows2023}). We can therefore apply the argument of \cite{DS1} (see also the proof of \cite[Lemma 5.8]{hallgrenKahlerRicciTangentFlows2023} to obtain $v_i\in C_c^{\infty}(\mathcal{R}_{-1},L^{\ell_i})$ satisfying the hypotheses of Lemma \ref{lem:perturbation} where we replace $\ell$ with $\ell_i$, $\zeta$ with a sequence $\zeta_i \searrow 0$, and $A$ with some $A_i \to \infty$. We may therefore apply Lemma \ref{lem:perturbation} to obtain a contradiction.
\end{proof}

We now fix $x_0 \in \mathcal{X}_{-1}$, $r_0,\ell_0>0$, and the $u\in H^0(\mathcal{R}_{-1},L^{\ell_0})$ as in Lemma \ref{lem:peak}. 

\begin{corollary} \label{cor:proper}
There exist constants $N$, $\rho_0\in (0,r_0)$, and $\ell_1,\ldots,\ell_N \in \mathbb{N}$,   $C<\infty$, and $s_j \in H^0(\mathcal{R}_{-1},L^{\ell_0\ell_j})$ such that 
\begin{equation*}
    F:= (F^1,\ldots,F^N):=\left( \frac{s_1}{u^{\ell_1}},\ldots,\frac{s_N}{u^{\ell_N}} \right): B(x_0,r_0)\to \mathbb{C}^N
\end{equation*}
satisfies the following:
\begin{equation*}
    |\nabla F| \leq C, \qquad \inf_{\partial B(x_0,r_0)}\sup_{1\leq j\leq N}|F^j|> \frac{1}{2}, \qquad \sup_{B(x_0,\rho_0)}\sup_{1\leq j\leq N}|F^j| \leq \frac{1}{100}.
\end{equation*}
\end{corollary}
\begin{proof}
Given Lemma \ref{lem:peak}, the argument proceeds exactly as in \cite[Proof of Theorem 1.2]{DS1}, \cite[Proposition 2.1]{DS2}, and \cite[Proof of Proposition 16]{CDSII}.
\end{proof}

We now define $\Omega := F^{-1}(B(0^N,\frac{1}{4N}))$, so that the restriction $F:\Omega \to B:=B(0^N,\frac{1}{4N})$ is a proper map and is holomorphic when restricted to $\Omega\cap \mathcal{R}_{-1}$, and (as in \cite[proof of Proposition 2.1]{DS2}) $F$ has finite fibers. 

\begin{lemma} \label{lem:separation}
\begin{enumerate}
    \item For any $x_0,x_1 \in \Omega$, there exists $m \in \mathbb{N}^{\times}$ and $s\in H^0(\mathcal{R}_{-1},L^{\ell_0 m})$ such that $\int_{\Omega \cap \mathcal{R}_{-1}}|s|_{h^{\ell_0 m}}^2 \omega^n<\infty$ and 
    \begin{equation*}
        \frac{s}{u^m}(x_0)\neq \frac{s}{u^m}(x_1).
    \end{equation*}

    \item For any $x_0 \in \Omega\cap \mathcal{R}_{-1}$, there exist $m\in \mathbb{N}^{\times}$ and $v_1,\ldots,v_n \in H^0(\mathcal{R}_{-1},L^{\ell_0m})$ such that $\int_{\Omega \cap \mathcal{R}_{-1}} |v_j|_{h^{\ell_0 m}}^2 \omega^n <\infty$ and 
    \begin{equation*}
        \left(\frac{v_1}{u^m},...,\frac{v_n}{u^m} \right)
    \end{equation*}
    is a holomorphic embedding from some neighborhood of $x_0$ into $\mathbb{C}^n$. 
\end{enumerate}
\end{lemma}
\begin{proof}
The argument proceeds as in \cite[Propositions 4.6 and 4.7]{DS1}; see also \cite[Proposition 5.9]{hallgrenKahlerRicciTangentFlows2023}, using Lemma \ref{lem:peak} and Proposition \ref{prop:hormander}.
\end{proof}

\begin{prop} \label{prop:complexvariety} After possibly adding additional components to $F$, $F:\Omega \to \mathbb{C}^N$ is a homeomorphism onto its image $Z \subseteq \mathbb{C}^N$, which is a normal complex analytic subvariety of $\mathbb{C}^N$. Moreover, $F$ maps $\mathcal{R}_{-1} \cap \Omega$ into the regular part $Z^{\operatorname{reg}}$ of $Z$. 
\end{prop}
\begin{proof}
    Given Corollary \ref{cor:proper} and Lemma \ref{lem:separation}, the proof proceeds as in \cite[Propositions 5.13 and 5.14]{hallgrenKahlerRicciTangentFlows2023} (which are based on arguments in \cite{DS2,LiuSz2}). In our setting, we use Lemma \ref{lem:C1estimates} to conclude the local Lipschitz properties of $L^2$ holomorphic sections and hence the Lipschitz property of the map $F:\Omega\rightarrow \mathbb C^N$. 
\end{proof}

\subsection{Further regularity results}
\label{sec:orbifold}

We now show that $\omega_{-1}$ has bounded local potentials and is a K\"ahler current on $X=\mathcal{X}_{-1}$. 

\begin{prop}\label{prop:bddlocal}
     For any $x_0 \in \mathcal{X}_{-1}$, there exists $r>0$ and $\varphi \in \mathrm{PSH}( B(x_0,r))\cap L^{\infty}(B(x_0,r))$ such that $\omega_{-1}=\sqrt{-1}\partial \overline{\partial}\varphi $ on $B(x_0,r)$. Moreover, there exists a neighborhood $U$ of $x_0$ and a holomorphic embedding $U \hookrightarrow \mathbb{C}^N$ such that
\begin{equation}
\omega_{-1} \ge C_p^{-1}\omega_{\mathbb{C}^N}\big|_{U}.
\end{equation}
\end{prop} 
\begin{proof} Lemma \ref{lem:peak} yields $k\in \mathbb{N}$, $r>0$, and $s\in H^0(\mathcal{R}_{-1},L^{k})$ such that $\int_{B(x_0,2r) \cap \mathcal{R}_{-1}}|s|_{h^k}^2 \omega^n <\infty$ and $|s|_{h^k}^2 \geq \frac{1}{2}$ on $B(x_0,r)\cap \mathcal{R}_{-1}$. By Lemma \ref{lem:C1estimates}, $\varphi:= -\frac{1}{k}\log |s|_{h^k}^2$ extends to a continuous function on $B(x_0,r)$. Moreover, $\sqrt{-1}\partial \overline{\partial}\varphi = \omega_{-1}$ on $B(x_0,r)\cap \mathcal{R}$, so by \cite{FoNa}, $\varphi$ is plurisubharmonic on all of $B(x_0,r)$ (in the sense that it extends to a plurisubharmonic function on local embeddings into $\mathbb{C}^N$). 

The existence of a local holomorphic embedding follows from Proposition~\ref{prop:complexvariety}. The lower bound of $\omega_{-1}$ by the restriction of the Euclidean metric follows from the $C^1$-estimate for holomorphic pluri-anticanonical sections Lemma~\ref{lem:C1estimates}, together with the lower bound for peaked pluri-anticanonical sections Lemma~\ref{lem:peak}.
\end{proof}

Next, we prove that the metric and algebraic singular sets of $\mathcal{X}_{-1}$ coincide, using a modification of the argument in \cite{LiuSz1}. In the following, we let $X^{\reg}$ denote the regular part of $X$ in the complex analytic sense. By Proposition \ref{prop:complexvariety}, we know that $\mathcal{R}_{-1}\subset X^\reg$.

\begin{prop} \label{prop:regularset} $X^{\operatorname{reg}}=\mathcal{R}_{-1}$.
\end{prop}
\begin{proof} Suppose by way of contradiction that $x \in X^{\operatorname{reg}} \setminus \mathcal{R}_{-1}$, and let $(u_i)_{i=1}^n$ be holomorphic coordinates on a neighborhood $U$ of $x$. We may argue as in \cite[Proposition 6.9]{hallgrenKahlerRicciTangentFlows2023} to then conclude that the holomorphic $(n,0)$-form $du_1 \wedge \cdots \wedge du_n$ has strictly positive Lelong number at $x$:
\begin{equation*}
    \lim_{r \to 0} \inf_{y\in B(x,r)\cap \mathcal{R}_{-1}} \frac{\log |du_1 \wedge \cdots \wedge du_n|_{\omega_{-1}}(y)}{\log|u(y)|}>0.
\end{equation*}
After possibly shrinking $U$, we may let $\varphi$ be a K\"ahler potential for $\omega_{-1}$ near $x$ as in Proposition \ref{prop:bddlocal}. Then
\begin{equation*}
    v := \log |du_1 \wedge \cdots \wedge du_n|_{\omega_{-1}}^2 + f - \varphi,
\end{equation*}
satisfies \begin{equation*} \sqrt{-1}\partial \overline{\partial} v = \Ric(\omega_{-1}) + \sqrt{-1}\partial \overline{\partial} (f-\varphi)=0
\end{equation*} 
on $U\cap \mathcal{R}_{-1}$. Because $U\setminus \mathcal{R}_{-1}$ has Hausdorff dimension at most $2n-4$ with respect to the Euclidean metric in the coordinates $(u_i)_{i=1}^n$, and because $v$ is pluriharmonic on $U \cap \mathcal{R}_{-1}$, it follows from \cite[Th\'eor\`eme 2]{GrRe} that $v$ extends to a pluriharmonic function on all of $U$, which is in particular bounded. Because $f$ and $\varphi$ are bounded on $U$, this implies $\log |du_1 \wedge \cdots \wedge du_n|_{\omega_{-1}}$ is bounded on $U \cap \mathcal{R}_{-1}$, contradicting the fact that $du_1 \wedge \cdots \wedge du_n$ has positive Lelong number.
\end{proof}

\begin{prop} \label{prop:klt} $X$ has klt singularities. Moreover, at every point of $X$, the tangent cone is unique, and is a K\"ahler Ricci-flat cone on an affine variety with log terminal singularities.
\end{prop}
\begin{proof}
The property that  $X$ has log terminal singularities follows from Lemma \ref{lem:peak}. Indeed, at each point $x\in X$, there exists $\ell=\ell(x)>0$ and a holomorphic section $u$ of $K_{X}^{-\ell}$ in neighborhood $U_x$ of $x$ such that on $U_x\cap X^{\reg}$, we have 
\begin{equation}\label{eq--two sided bound}
    C^{-1}\leq |u|_{h^{\ell}}\leq C.
\end{equation}
Let $s$ be the dual holomorphic section of $u$ in $H^0(U_x, K_X^\ell)$. By the definition of $h^\ell$ and \eqref{eq--two sided bound}, we know that on $U_x$,
\begin{equation}
C^{-1}\omega_X^n \le \left((\ii)^{n^2\ell}s\wedge \bar s\right)^{\frac{1}{\ell}} \le C\omega_X^n .
\end{equation}
Since $\omega_X^n$ has finite local volume, by \cite[Lemma 6.4]{EGZ} we know that $X$ has klt singularities near $x$. Moreover, there exists $p=p(x)>1$ such that for some locally smooth K\"ahler metric $\omega_0$ on $U_x$, we have
\begin{equation}
    \frac{\omega_X^n}{\omega_0^n} \in L^p(U_x,\omega_0^n).
\end{equation}

By \cite[Theorem 2.18]{bamler3}, we know that tangent cones exist at every point of $X$ in the Gromov–Hausdorff sense, and that the convergence is smooth on the regular part. We note that the arguments used in this paper also apply to any tangent cone of $X$. Therefore, we conclude that every tangent cone of $X$ is a K\"ahler Ricci-flat cone over an affine variety with log terminal singularities (see also \cite{hallgrenKahlerRicciTangentFlows2023}). The uniqueness of the tangent cone then follows from the same argument as in \cite[Section 3]{DS2}.
\end{proof}

In the following, we show that $\omega_X$ is a smooth orbifold metric on the orbifold locus $X^{\mathrm{orb}}$ of $X$. We first recall the definition and note that, by \cite[Lemma~5.8]{graf2020} and \cite[Proposition 9.3]{GKK}, if $X$ is a complex analytic variety with klt singularities, then $X \setminus X^{\mathrm{orb}}$ is contained in an analytic subvariety of codimension at least $3$.

\begin{definition}
\begin{enumerate}
    \item $X^{\mathrm{orb}}$ denotes the set of $x \in X$ which admit a neighborhood biholomorphic to a neighborhood of the origin of the form
\[
\mathbb{C}^{n} / G,
\]
where $G$ is a finite subgroup of $GL(n,\mathbb{C})$ that does not contain any reflections.
\item A K\"ahler form $\omega$ on $X^{\mathrm{orb}}$ is called a smooth orbifold K\"ahler metric if, for every orbifold chart modeled on $\mathbb{C}^n/G$ with local covering map 
\[
p:\mathbb{C}^n \to \mathbb{C}^n/G,
\]
the pullback $p^*\omega$ extends to a smooth K\"ahler form on a neighborhood of the origin in $\mathbb{C}^n$.
\end{enumerate}
\end{definition}

\begin{prop} \label{prop:orbifoldreg}
    $\omega_X$ is a smooth orbifold metric on the orbifold locus $X^{\mathrm{orb}}$ of $X$.
\end{prop}
\begin{proof}
Since we can pull back the geometric quantities and the cut-off functions to the local uniformizing chart while preserving the same estimates, we are able to apply the arguments of Lemma \ref{lem:C1estimates} to the standard coordinates on $\mathbb{C}^n$ in any local uniformizing chart $p$ to obtain gradient estimates for holomorphic functions with respect to the pull-back of the shrinker metric $\pi^*\omega_X$ and the Hermitian metric $p^*(e^{-f})$. Since $\nabla f$ is locally bounded on $X$, this yields gradient estimates for holomorphic functions and, consequently, a lower bound for $\omega_X$:
\begin{equation}\label{eq-lower bound of omega_X}
p^*\omega_X\geq C^{-1}\omega_{\mathbb C^n}.
\end{equation}
Let $\varphi$ denote a bounded local potential for $\omega_X$. Then the shrinker equation implies that, for some neighborhood $B\subseteq \mathbb{C}^n$ of the origin, we have
\begin{equation} \label{eq:shrinkereqinchart}
\sqrt{-1}\partial\bar\partial\left(\log\left(\frac{p^*\omega_X^n}{\omega_{\mathbb{C}^n}^n}\right) + p^*(\varphi - f)\right)=0
\end{equation}
on $Z$, where the complement of $Z:= p^{-1}(X^{\operatorname{reg}})$ has complex codimension at least $2$ in $B$. By \cite[Th\'eor\`eme 2]{GrRe}, the above function extends to a pluriharmonic function on $B$. In particular, it is bounded. Combining this fact with \eqref{eq-lower bound of omega_X} therefore yields 
\begin{equation*}
  C^{-1}\omega_{\mathbb C^n}\leq  p^*\omega_X\leq C\omega_{\mathbb C^n}
\end{equation*}
on $B$ in the sense of currents (and pointwise on $B\setminus Z$).

In the following, we omit the obvious pullback notation $p^*$ and work on a local uniformizing chart. From local elliptic regularity applied to $\Delta_{\omega_{\mathbb{C}^n}}\varphi = \operatorname{tr}_{\omega_{\mathbb{C}^n}}(\omega_X)$, we know that $$\varphi \in W^{2,p}\cap C^{1,\alpha}$$ for every $p \in (1,\infty)$ and $\alpha \in (0,1)$. By Hartog's theorem, the lift $V$ of $\frac{1}{2}\nabla f$ extends to a real-holomorphic vector field on a neighborhood of the origin. By $\sqrt{-1}\partial \overline{\partial}(V\varphi)=\mathcal{L}_{V}\omega$ and \eqref{eq:shrinkereqinchart}, we have
\begin{equation}\label{eq--complex monge-ampere}
    (\sqrt{-1}\partial\bar\partial \varphi)^n
    =
    e^{V\varphi+f_0-\varphi}\,\omega_{\mathbb C^n}^n,
\end{equation}
on $Z$, where $f_0$ is a smooth function in a neighborhood of the origin in $\mathbb C^n$. Since $V$ and $f_0$ are smooth and $\varphi \in C^{1,\alpha}$, we know the right-hand side of \eqref{eq--complex monge-ampere} belongs to $C^{0,\alpha}$ for every $\alpha \in (0,1)$. By \cite[Theorem 3.2]{zeng2016} (see also \cite{chen-he,wang2012c,LLZ} for a related discussion), we conclude that $\varphi \in C^{2,\alpha}$ for every $\alpha \in (0,1)$. A standard bootstrapping argument using elliptic regularity then implies that $\varphi$ is in fact smooth.
\end{proof}

Theorem \ref{thm:complexstructure} follows from the above results, as we now summarize.

\begin{proof}[Proof of Theorem \ref{thm:complexstructure}] 
We first consider $X$ arising under assumption \ref{assume:KRF} of Section \ref{sec;convergence setting}. In this case, 
Claim \ref{thm:complexstructure3} is exactly Proposition \ref{prop:regularset}. 
Assertions \ref{thm:complexstructure1} and \ref{thm:complexstructure4} follow from Proposition \ref{prop:klt}, while \ref{thm:complexstructure2} follows from Proposition \ref{prop:bddlocal}. Finally, \ref{thm:complexstructure5} is just Proposition \ref{prop:orbifoldreg}.

Now suppose instead that $X$ arises under assumption \ref{assume:KRS} of Section \ref{sec;convergence setting}. 
By \cite[Theorem 21]{LW20}, Perelman's differential Harnack inequality holds for any (smooth) complete shrinking Ricci soliton, without the assumption of bounded curvature. By the convergence of heat kernels from \cite[Theorem 6.12(4a)]{LW24}, this passes to the limit, ensuring that the conclusion of Theorem \ref{thm-harnack inequality} also holds in this setting. Similarly, the heat kernel bounds used in the proofs of Proposition \ref{prop:bddlocal} and Lemma \ref{lem:C1estimates} hold on $X$ by passing to the limit, where we use \cite[Theorem 5]{LW20} instead of \cite[Theorem 7.1]{bamler1}, and \cite[Theorem 1.3]{LW24} instead of \cite[Theorem 7.2]{bamler1}.
The existence of cutoff functions as in Lemma \ref{lem:chicutoff} follows from the construction in \cite[Lemma 15.27]{bamler3}, using \cite[Theorem 6.23]{LW24}. The conclusion of Lemma \ref{lem:cutoff} holds by \cite[Theorem 3.6]{HLW21}. From these considerations, the proofs of Proposition \ref{prop:bddlocal}, Proposition \ref{prop:hormander}, and Lemma \ref{lem:C1estimates} also apply in this setting. Using \cite[Theorem 6.23]{LW24} instead of \cite[Theorem 2.30]{bamler3} in the proof of Lemma \ref{lem:perturbation}, we then obtain the conclusion of Lemma \ref{lem:perturbation} as well. 

By \cite[p. 1686]{LW24}, any tangent cone $\mathcal{C}$ based at a point of $X$ is a singular space in the sense of \cite[Definition 2.15]{bamler3}. Moreover, $\mathcal{C}$ is a metric cone, the regular part $\mathcal{R}_{\mathcal{C}}\subseteq \mathcal{C}$ of $\mathcal{C}$ is a Ricci-flat K\"ahler-cone, and $\mathcal{C}\setminus \mathcal{R}_{\mathcal{C}}$ has Minkowski codimension at least four. As a consequence, the conclusions of Lemma \ref{lem:peak} hold, and from here the remaining claims of Subsection \ref{sec:normal} proceed exactly as in the case of assumption \ref{assume:KRF}. Similarly, we note that the ingredients of \cite[Proposition 6.9]{hallgrenKahlerRicciTangentFlows2023} are all available under assumption \ref{assume:KRS}, hence the results of Subsection \ref{sec:orbifold} hold as well, and the claim follows.  
\end{proof}

\section{Local Equivariant Embeddings and Long-Time Pseudolocality}
\label{sec:equivariant}
Throughout this section, we fix a singular shrinking K\"ahler--Ricci soliton $(X,d)$ arising as in either assumption \ref{assume:KRF} or \ref{assume:KRS} of Section \ref{sec;convergence setting}.
Let $(X^{\operatorname{reg}},\omega,J)$ denote the K\"ahler structure on the regular set, and $f\in C_{\operatorname{loc}}^{0,1}(X)\cap C^{\infty}(X^{\operatorname{reg}})$ denote the soliton potential function. Throughout this section, we also assume $f$ is not a constant, or equivalently that $\nabla f$ does not vanish identically on $X^{\operatorname{reg}}$. By \cite[proof of Theorem 1.3]{halljian}, $J\nabla f$ is a complete vector field on $X^{\operatorname{reg}}$, and the one-parameter family of holomorphic isometries $(\sigma_s)_{s\in \mathbb{R}}$ on $X^{\operatorname{reg}}$ generated by $J\nabla f$ extends uniquely to a one-parameter family of isometries of $(X,d)$. 
By \cite[proof of Proposition 6.12]{hallgrenKahlerRicciTangentFlows2023}, the closure of $\{\sigma_s \: : \: s\in \mathbb{R}\}$ in the isometry group of $(X^{\operatorname{reg}},\omega)$ is a real $m$-dimensional torus $\mathbb{T}$, and each $\sigma \in \mathbb{T}$ acts by biholomorphic isometries on $X^{\operatorname{reg}}$. As a consequence of the fact that $X$ admits a normal complex analytic variety structure, we know that $\mathbb{T}$ acts on $X$ by biholomorphisms.

Since the action of $\mathbb{T}$ preserves $X^{\operatorname{reg}}$ and consists of biholomorphisms, it admits a natural lift to each line bundle $L^{\ell} = K_{X^{\operatorname{reg}}}^{-\ell}$.
Consequently, it induces an action of $\mathbb{T}$ on $H^0(X^{\operatorname{reg}}, L^{\ell}) = H^0(X, L^{\ell})$.
More precisely, for $\zeta \in \mathbb{T}$ and $u \in H^0(X^\reg,L^{\ell})$, for $x \in X^{\operatorname{reg}}$, we define
\begin{equation*}
    (\zeta \cdot u)(x) := \zeta_*\big(u(\zeta^{-1}(x))\big).
\end{equation*}
We identify elements of $\rm{Lie}(\mathbb T)$, the Lie algebra of $\mathbb{T}$, with holomorphic Killing vector fields on $X^{\mathrm{reg}}$. Then the $\mathbb T$-action on $H^0(X^{\operatorname{reg}},L^{\ell})$ induces a Lie derivative operator $\mathcal{L}_\xi$ on $H^0(X^{\operatorname{reg}},L^{\ell})$ for each $\xi \in \mathrm{Lie}(\mathbb{T})$.

\begin{lemma} \label{lem:decomposition} Given $\ell\in \mathbb N_{\geq 0}$ and $u \in H^0(X,L^{\ell})$, there exist $u_{\alpha}\in H^0(X,L^\ell)$ for each character $\alpha$ of $\mathbb{T}$ such that the following hold:

\begin{enumerate}
    \item \label{lem:decomposition1} $\mathcal{L}_{\xi}u_{\alpha}=\sqrt{-1}\langle \alpha,\xi\rangle u_{\alpha}$ for any $\xi\in \rm{Lie}(\mathbb T)$, 

    \item \label{lem:decomposition2} for any $\mathbb{T}$-invariant relatively compact set $K\subset X$,
\[
\lim_{j\to\infty}\sup_{K\cap X^{\reg}}\left|u-\sum_{|\alpha|\le j}u_\alpha\right|_{h^\ell}=0.
\]
\end{enumerate}
\end{lemma}
\begin{proof} Let $\mu_{\mathbb{T}}$ denote the Haar measure of $\mathbb{T}$. For each $u\in H^0 (X,L^{\ell})$ and each character $\alpha$ of $\mathbb{T}$, define
   \begin{equation*}
       u_{\alpha}(z) := \int_{\mathbb{T}} (\zeta \cdot u)(z)\zeta^{-\alpha}d\mu_{\mathbb{T}}(\zeta)
   \end{equation*}
   for $z\in X^{\reg}$, so that the sections $u_{\alpha} \in H^0(X,L^\ell)$ satisfy $\zeta \cdot u_{\alpha} =\zeta^{\alpha} u_{\alpha}$. In particular, \ref{lem:decomposition1} holds. 
   
   Let $U \subset \subset X$ be $\mathbb{T}$-invariant, and let $\mathcal{H}$ denote the Hilbert space of $L^2$ sections of $L^\ell|_U$. Because the regular representation of $\mathbb{T}$ on $\mathcal{H}$ is unitary, and $u_{\alpha}$ are exactly the projections onto the weight spaces of the $\mathbb{T}$-action on $\mathcal{H}$, a consequence of the Peter-Weyl theorem (see \cite[Theorem III.5.10]{BrtD}) yields
\begin{equation*}
    \lim_{j \to \infty} \int_U \left| u - \sum_{|\alpha|\leq j} u_{\alpha}\right|_{h^\ell}^2 \omega^n =0.
\end{equation*}
Then local elliptic regularity gives
   \begin{equation*}
       \lim_{j\to \infty} \sup_{\mathcal{K}} \left| u-\sum_{|\alpha|\leq j} u_{\alpha}\right|_{h^\ell} = 0
   \end{equation*}
for any compact subset $\mathcal{K} \subseteq U\cap X^{\reg}$. For any $y_0 \in U$, we can find a neighborhood $U_{y_0} \subseteq U$ of $y_0$ and some locally bounded holomorphic section $s\in H^0(U_{y_0}\cap X^{\operatorname{reg}},L^\ell)$ with $\inf_{U_{y_0} \cap X^{\operatorname{reg}}}|s|_{h^\ell}>0$. Writing $\widehat{u} := s^{-1}u$ and $\widehat{u}_{\alpha} := s^{-1}u_{\alpha}$, it follows that $\widehat{u},\widehat{u}_{\alpha}$ are holomorphic functions on $U$ which satisfy
\begin{equation*}
    \lim_{j\to \infty} \sup_{\mathcal{K}} \left| \widehat{u} - \sum_{|\alpha|\leq j} \widehat{u}_{\alpha} \right| =0
\end{equation*}
for each compact subset $\mathcal{K} \subseteq U_{y_0}\cap X^{\operatorname{reg}}$. Moreover, any compact subset of $U$ can be covered by finitely many such neighborhoods $U_{y_0}$. By \cite[p.~340]{DS2}, using the fact that a singular point lies in a holomorphic disk whose boundary is a fixed distance away from the singular set, a simple maximum modulus argument shows that the convergence is also uniform on compact subsets $\mathcal{K}\subseteq U$. This implies \ref{lem:decomposition2}.
\end{proof}

\begin{prop} \label{prop:equivariant} Any $x_0 \in X$ possesses a $\mathbb{T}$-invariant neighborhood $U \subseteq X$ and a holomorphic map $G:U \to \mathbb{C}^K$ for some $K\in \mathbb{N}^{\times}$, such that $G$ is $\mathbb{T}$-equivariant with respect to some linear action of $\mathbb{T}$ on $\mathbb{C}^K$, and such that $G$ restricts to an embedding on a neighborhood of $x_0$.  
\end{prop}
\begin{proof}
By Lemma \ref{lem:peak}, we can choose $s \in H^0(X,L^{\ell_0})$, $r>0$ such that $\inf_{B(x_0,r)}|s|_{h^{\ell_0}}>0$. By Lemma \ref{lem:decomposition}\ref{lem:decomposition2} and Lemma \ref{lem:C1estimates}\ref{lem:C1estimates2}, we can shrink $r$ in order to guarantee the existence of $k\in \mathbb{Z}^n$ such that $\inf_{B(x_0,r)}|s_k|_{h^{\ell_0}}>0$. By Proposition \ref{prop:complexvariety}, there exist $u_j \in H^0(X,L^{\ell_0\ell_j})$, $1\leq j\leq N$, such that (after again shrinking $r>0$)
\begin{equation*}
   (F^1,\ldots,F^N):= \left( \frac{u_1}{s_k^{\ell_1}},\ldots,\frac{u_N}{s_k^{\ell_N}} \right) : B(x_0,r) \to \mathbb{C}^N
\end{equation*}
is a holomorphic embedding. Because $|s_k|_{h^{\ell_0}}$ is invariant under the action of $\mathbb{T}$, $(F^1,...,F^N)$ defines a holomorphic map $U:= \mathbb{T}\cdot B(x_0,r) \to \mathbb{C}^N$. Moreover, Lemma \ref{lem:decomposition}\ref{lem:decomposition2} implies that for $M\in \mathbb{N}^{\times}$ sufficiently large, 
\begin{equation*}
   G:= (G_1,\ldots,G_{NM}):= \left( \frac{u_{1,1}}{s_k^{\ell_1}},\ldots,\frac{u_{1,M}}{s_k^{\ell_1}},\ldots,\frac{u_{N,1}}{s_k^{\ell_N}},\ldots,\frac{u_{N,M}}{s_k^{\ell_N}} \right) : B(x_0,r) \to \mathbb{C}^{MN}
\end{equation*}
defines a holomorphic embedding, where each $u_{i,j}$ is equal to $u_{i,\alpha}$ for some character $\alpha$ of $\mathbb{T}$. By Lemma \ref{lem:decomposition}\ref{lem:decomposition1}, each $G_j$ satisfies $\zeta \cdot G_j = \zeta^{k_j}G_j$ for some $k_j \in \mathbb{Z}^n$. In other words, $G$ is $\mathbb{T}$-equivariant with respect to a linear action on $\mathbb{C}^{MN}$.
\end{proof}

\begin{prop}\label{prop:holovectorfields} $X$ admits a holomorphic $(\mathbb{C}^*)^m$-action for some $m \ge 1$. More precisely, 
\begin{enumerate}
    \item \label{prop:holovectorfields1} $\nabla f$ is a complete vector field on $X^{\mathrm{reg}}$, and its flow extends to biholomorphisms of $X$;
    \item \label{prop:holovectorfields2} Holomorphic vector fields in $J\mathrm{Lie}(\mathbb{T})$ are complete on $X^{\reg}$, and their flows extend to biholomorphisms of $X$.
\end{enumerate} 
\end{prop} 
\begin{proof}
   \ref{prop:holovectorfields1} 
   This roughly follows by arguing as in \cite[Section~2.3]{DS2}. Since $|\nabla f|$ grows at most linearly with the distance to a fixed point on $X$, the distance from any integral curve of $\nabla f$ to a fixed point on $X$ is bounded whenever the time interval is finite. Then, since $X$ is a complete metric space, any integral curve $\nabla f$ over a finite time interval remains in a compact subset of $X$. Therefore, in order to show that $\nabla f$ is complete on $X^\reg$,  it suffices to show that for any point $q \in X \setminus X^{\mathrm{reg}}$, there exists a neighborhood $\Omega$ of $q$ in $X$ and $\epsilon > 0$ such that the flow $\phi_t$ of $\nabla f$ is well-defined on $\Omega \cap X^{\mathrm{reg}}$ for all $t \in [-\epsilon,\epsilon]$. We can embed a neighborhood $\Omega$ of $q$ into $\mathbb{C}^N$. Since $\nabla f$ is a holomorphic vector field, after possibly shrinking $\Omega$, we may assume that there exists a holomorphic vector field $\xi$ on a neighborhood $U$ of $\Omega$ inside $\mathbb{C}^N$ such that for any $x \in \Omega \cap X^{\mathrm{reg}}$, we have 
    \begin{equation}
        \xi(x)=\nabla f(x).
    \end{equation}
Then there exists a neighborhood $V \subset U$ of $q$ in $\mathbb{C}^N$ such that $\xi$ generates a local holomorphic flow $\phi_t : V \to U$ for $t \in [-\epsilon,\epsilon]$. We claim $\phi_t$ maps $V\cap \Omega\cap X^{\reg} $ into $U\cap\Omega\cap X^{\reg}$. Since $\phi_t$ is a biholomorphism, it is enough to show that $\phi_t$ maps $V\cap \Omega $ into $U\cap \Omega$. Since they are analytic subvarieties, it suffices to show that if a holomorphic function $u$ on $U$ vanishes on $U \cap \Omega$, then it also vanishes on $\phi_t(V \cap \Omega)$. Given such a $u$, since $\xi$ is tangential to $\Omega\cap X^{\reg}$, we have $\xi(u)=0$ on $U\cap \Omega\cap X^{\reg}$ and hence $\xi(u)=0$ on $U\cap \Omega$ by normality of $X$. This implies that on $U\cap \Omega$, for $t\in [-\epsilon,\epsilon]$,
\begin{equation}
    \phi_t^*(u)=0.
\end{equation} That is, $u$ vanishes on $\phi_t(V\cap \Omega)$. Therefore by shrinking $\Omega$, we obtain that $\phi_t$ is well-defined on $\Omega\cap X^{\reg}$ for $t\in [-\epsilon,\epsilon]$.

The previous argument shows that near any point $q \in X$, the flow of $\nabla f$ for a short time is locally given by the restriction of the flow in the ambient space via a local holomorphic embedding. Therefore, it extends continuously to $q$. By a standard compactness argument, we then know that for any compact subset $K \subset X$, there exists $\epsilon = \epsilon(K) > 0$ such that for all $|t| \leq \epsilon$, the flow $\phi_t$ of $\nabla f$ extends continuously to all of $K$. Then we show that for any given time $T \in \mathbb{R}$, the flow $\phi_T$ of $\nabla f$ 
extends to a continuous map from $X$ to itself. To see this, note that since $\nabla f$ has linear growth, for any fixed point $q \in X$, we can find a neighborhood $\Omega$ of $q$ in $X$ such that
\[
\bigcup_{|t| \le |T|} \phi_t(\Omega \cap X^{\reg})
\]
is contained in a compact subset $K$ of $X$. Then choose $N$ sufficiently large such that $\frac{|T|}{N}\leq \epsilon(K)$ and write 
\begin{equation}
\phi_T = \underbrace{\phi_{T/N} \circ \cdots \circ \phi_{T/N}}_{N \text{ times}}.
\end{equation} 
Since each $\phi_{T/N}$ extends continuously to a map defined on $K$ and for each $j \in \{0,\cdots, N\}$,
\begin{equation*}
    \phi_{jT/N}(\Omega\cap X^{\reg})\subset K,
\end{equation*} we obtain that $\phi_T$ extends as a continuous map defined on $\Omega$ and hence $\phi_T$ extends to a continuous map defined on $X$.
Then the normality of $X$, together with the fact that the flow is biholomorphic when restricted to $X^\reg$, implies that the flow of $\nabla f$ is a biholomorphism from $X$ to itself.

\ref{prop:holovectorfields2} Let $V_0=\nabla f$ and $V \in J\mathrm{Lie}(\mathbb{T})$.
Let $\gamma$ be an integral curve of $V$ with maximal defining interval $(-T_1, T_2)$, i.e. for $t\in (-T_1,T_2)$,  \begin{equation}\label{eq--integral curve}
      \dot{\gamma}(t)=V_{\gamma(t)}, \quad \gamma(0)\in X^{\reg}.
  \end{equation} 
We show that $T_2 = \infty$ below; the proof that $T_1 = \infty$ is similar. Suppose by way of contradiction that $T_2<\infty$, set $x:=\gamma(T_2-\frac{1}{2}\min\{1,T_1+T_2\})$, and consider $\phi_t(x)$ for $t\in (-\infty, 0]$, where as before $\phi_t$ denotes the flow generated by $\nabla f$. We have 
  \begin{equation}
      \partial_tf(\phi_t(x))=|\nabla f|^2(\phi_t(x)).
  \end{equation}Since $f$ is proper and bounded from below, for $t\in (-\infty, 0]$, $\phi_t(x)$ stays in a compact subset of $X$ and $\int_{-\infty}^0|\nabla f|^2(\phi_t(x))\, dt<\infty$. 
Combining a compactness argument with Lemma~\ref{lem:separation} and Proposition~\ref{prop:complexvariety}, we conclude that for any compact subset $K \subset X$, there exists a holomorphic embedding $K \hookrightarrow \mathbb{P}^N$. Restricting the Fubini--Study metric then yields a smooth K\"ahler metric $\omega_0$ on $K$. By the gradient estimate as before, we obtain the existence of $c_0>0$ such that
\begin{equation}
\omega \geq c_0 \omega_0 \text{ on } K.
\end{equation} 
Since $V_0=\nabla f$ can be extended to a local holomorphic vector field in the ambient space, $|V_0|_{\omega_0}$ is locally Lipschitz, and
\begin{equation}
\int_{-\infty}^0 |V_0|_{\omega_0}^2(\phi_t(x))\, dt
\le
c_0^{-1}\int_{-\infty}^0 |V_0|^2(\phi_t(x))\, dt < \infty .
\end{equation}
It follows that
\begin{equation}\label{eq--limit is a critical point}
    \lim_{t\to -\infty} |V_0|_{\omega_0}(\phi_t(x)) = 0 .
\end{equation}
Therefore, there exists a sequence $t_i\rightarrow -\infty$ such that $\phi_{t_i}(x)\rightarrow x_0$ where $V_0=\nabla f$ vanishes at $x_0$. Since $V\in J \mathrm{Lie}(\mathbb T)$, $V$ also vanishes at $x_0$. Fix a ball $B$ with respect to a smooth K\"ahler metric on $X$ around $x_0$. If $\delta$ is sufficiently small, then the flow of $V$ starting at $y$ will exist and stay in $B$ for at least time $1$ since $V$ vanishes at $x_0$. Therefore the argument for (i) implies that the flow of $V$ is well-defined for any point in $B$ up to time 1. Because $[\nabla f,V]=0$, if we then flow forward by $\nabla f$, this segment inside $B$ is moved to an integral curve of $V$ existing for time $1$ and starting at $x$. But then this contradicts the assumption that $T_2<\infty$ is the maximal defining time of \eqref{eq--integral curve}.

We now show that for any $T\in \mathbb{R}$, the flow $\phi_T^V$ of $V$ on $X^{\operatorname{reg}}$ extends continuously to $X$. 
Fix a point $q\in X$. We claim that there exists a sequence $t_i\to -\infty$ such that 
$\phi_{t_i}(q)$ converges to a point $x_0\in X$ where $\nabla f$ vanishes.
If $q\in X^{\reg}$, this follows from the argument above. In general, since $\nabla f$ is locally bounded, the function 
$t\mapsto f(\phi_t(q))$ is locally Lipschitz. In particular, its derivative exists for almost 
every $t$. Using an approximation argument and the fact that $|\nabla f|_{\omega_0}$ is 
continuous, we obtain that for almost every $t\in (-\infty,0)$,
\begin{equation}\label{eq--derivative upper bound}
\frac{d}{dt}f(\phi_t(q))\geq c_0 |\nabla f|_{\omega_0}^2(\phi_t(q)).
\end{equation}
Consequently,
\[
\int_{-\infty}^0 |\nabla f|_{\omega_0}^2(\phi_t(q))\, dt < \infty .
\]
It follows that there exists a sequence $t_i\to -\infty$ such that 
$\phi_{t_i}(q)$ converges to a point $x_0\in X$, where $\nabla f$ vanishes. Since $V$ vanishes at $x_0$, the argument in (i) implies that there exists a neighborhood 
$U$ of $x_0$ such that the flow $\phi_T^V$ extends continuously to $U$. 
Moreover using the continuity of $\phi_{t_i}$, we know that there exists a neighborhood $U_q$ of $q$ and $t_i$ such that
\[
\phi_{t_i}(U_q)\subset U .
\]
Because $[\nabla f, V]=0$, the corresponding flows commute, i.e.
\begin{equation}
\phi_T^V  = \phi_{-t_i} \circ \phi_T^V \circ \phi_{t_i}.
\end{equation}
Using this relation, we conclude that $\phi_T^V$ extends continuously to $U_q$. 
Since $q$ was arbitrary, it follows that $\phi_T^V$ extends continuously to the whole space $X$.
\end{proof}

Recall that we let $\mathcal{X}$ denote the metric flow corresponding to $X$ (see Section~\ref{Singular Kahler-Ricci shrinkers}).

\begin{corollary} \label{cor:timelikecomplete} The maximal existence time of any integral curve of $\partial_{\mathfrak{t}}$ is $(-\infty,0)$. 
\end{corollary}
\begin{proof} By Proposition \ref{prop:holovectorfields} and Proposition \ref{prop:regularset}, $\tau \nabla f \in \mathfrak{X}(\mathcal{R})$ is a complete vector field. Because $\tau (\partial_{\mathfrak{t}}-\nabla f)$ is a complete vector field by \cite[Theorem 15.69]{bamler3}, and because 
\begin{align*}
    [\tau \nabla f,\tau(\partial_{\mathfrak{t}}-\nabla f)]=0,
\end{align*}
the claim follows from the fact that the sum of complete commuting vector fields is complete. 
\end{proof}

\begin{proof}[Proof of Theorem \ref{thm:pseudolocality}] Suppose by way of contradiction that for some $D,W<\infty$, there exist compact K\"ahler--Ricci flows $(M_i^n,(g_{i,t})_{t\in [-\delta_i^{-1},0]})$, $\delta_i \searrow 0$ and $x_{0,i} \in M_i$ such that $\mathcal{N}_{x_{0,i},0}(1)\geq -W$, $(x_{0,i},0)$ is $(\delta_i,1)$-selfsimilar, but also that there exist $(x_i,t_i) \in P^{\ast}(x_{0,i},0;D) \cap (M_i \times [-D^2,-D^{-2}])$ satisfying $r_{\operatorname{Rm}}(x_i,t_i)\geq \sigma_0$ but
\begin{equation*}
   \lim_{i \to \infty} r_{\operatorname{Rm}}(x_i,t_i')=0  
\end{equation*}
for some $t_i'\in [-D^2,-D^{-2}]$. After passing to a subsequence, we may assume that 
$$(M_i,(g_{i,t})_{t\in [-\delta_i^{-1},0]},(\nu_{x_{0,i},0;t})_{t\in [-\delta_i^{-1},0]})\xrightarrow[i\to \infty]{\mathbb{F},\mathfrak{C}}(\mathcal{X},(\nu_t)_{t\in (-\infty,0)})$$
uniformly on compact time intervals with respect to some correspondence $\mathfrak{C}$, where $\mathcal{X}$ is a K\"ahler metric soliton. After passing to a subsequence, we can assume $t_i \to t_{\infty} \in [-D^2,-D^{-2}]$, and \cite[Lemma 4.2]{FlHa} gives a conjugate heat flow $(\mu_t)_{t\in (-\infty,t_{\infty})}$ such that $\lim_{t \nearrow t_{\infty}}\operatorname{Var}(\mu_t)=0$ and
\begin{align*}
    (\nu_{x_{i},t_i;t})_{t\in [-\delta_i^{-1},t_i]}\xrightarrow[i\to \infty]{\mathfrak{C}} (\mu_t)_{t\in (-\infty,t_{\infty})}.
\end{align*}
By arguing as in \cite[Proof of Lemma 26.1]{bamler3}, for each $\delta>0$, we can find $t^{\ast} \in (t_{\infty}-\delta,t_{\infty})$ such that for any $H_{2n}$-center $z_{\infty} \in \mathcal{X}_{t^{\ast}}$ of $(\mu_t)$, and any sequence $z_i \in M_i$ such that $(z_i,t^\ast)\xrightarrow[i\to \infty]{\mathfrak{C}}z_{\infty}$, we have $d_{t^\ast}(z_i,x_i) < C(\sigma_0)\sqrt{\delta}$. By choosing $\delta>0$ sufficiently small, and using $r_{\operatorname{Rm}}^{g_i}(x_i,t_i) \geq \sigma_0$, we can ensure that $B(z_{\infty},\frac{\sigma_0}{2}) \subseteq \mathcal{R}_{t^{\ast}}$. By Corollary \ref{cor:timelikecomplete},  $\mathfrak{t}\partial_{\mathfrak{t}} \in \mathfrak{X}(\mathcal{R})$ is complete, hence $B(z_{\infty},\frac{\sigma_0}{2})((-\infty,0)) \subseteq \mathcal{R}$ is unscathed. As in \cite[Proof of Lemma 26.1]{bamler3}, the smooth convergence implies a uniform (in $i\in \mathbb{N}$) curvature bound on 
\begin{align*}
    B(z_i,t^{\ast},\frac{\sigma_0}{4})\times[-2D^2,-\frac{1}{2}D^{-2}] \supseteq B(x_i,t^{\ast},\frac{\sigma_0}{8})\times [-2D^2,-\frac{1}{2}D^{-2}].
\end{align*}
In particular, $\liminf_{i \to \infty}r_{\operatorname{Rm}}^{g_i}(x_i,t_i')>0$, a contradiction. 
\end{proof}

\section{Properties of the moment map}
\label{sec:moment}
As in the previous section, let $X$ be a singular K\"ahler--Ricci shrinker arising as in either assumption \ref{assume:KRF} or \ref{assume:KRS} of Section \ref{sec;convergence setting}. By Section \ref{sec:equivariant}, $X$ admits a holomorphic and Killing action by a torus $\mathbb{T}$. In the following, we typically identify each element $\xi \in \lie(\mathbb{T})$ with the corresponding real holomorphic Killing vector field on $X$.
Since $\omega$ lies in the cohomology class of the anticanonical bundle $-K_X$, and the $\mathbb{T}$-action admits a canonical lift to $-K_X$, we can show that the $\mathbb{T}$-action admits a moment map.
\begin{lemma}
    There exists a smooth map $\mu:X^{\reg}\rightarrow \lie(\mathbb T)^*$ such that for any $\eta\in \lie(\mathbb T)$, we have 
    \begin{equation}
   \iota_{\eta}\omega=-d \langle \mu,\eta\rangle.
    \end{equation}
\end{lemma}
\begin{proof}
   We define a Hamiltonian potential $u_{\eta}$ for $\eta$, defined on $X^\reg$, by
\begin{equation}\label{eq--def of moment map}
    u_{\eta} = \frac{1}{2}\big(\Div(J\eta) - J\eta(f)\big)
              = \frac{1}{2} \frac{\mathcal{L}_{J\eta}(e^{-f}\omega^n)}{e^{-f}\omega^n}.
\end{equation}Then using the shrinker equation,  we obtain that
\begin{equation}
    \iota_{\eta}\omega=-du_{\eta}.
\end{equation}Clearly $u_\eta$ depends linearly on $\eta$ and hence we have a well-defined moment map $\mu$ on $X^{\reg}$ given by $\langle \mu,\eta\rangle := u_{\eta}$.
\end{proof}

\begin{lemma} \label{lem-vector field bounded}
For any $\eta \in \operatorname{Lie}(\mathbb{T})$, we have
\begin{equation*}
|\eta|_{\omega}\in L^{\infty}_{\mathrm {loc}}(X). 
\end{equation*}
\end{lemma}
\begin{proof} 
Fix $x_0 \in X$, and use Proposition \ref{prop:equivariant} to choose a $\mathbb T$-invariant open set $B$ containing $x_0$ and relatively compact in $X$. Then we can take a $\mathbb T$-equivariant resolution $\pi:\widetilde{B}\rightarrow B$ \cite[Theorem 2.0.1]{wlodarczyk2008}, so that the holomorphic vector field $\eta$ lifts to a holomorphic vector field $\widetilde{\eta}$ on $\widetilde{B}$.
Then in the following, we work in a small neighborhood of $x_0$ and consider its preimage under the resolution map $\pi$. Let $\omega_B,\omega_{\widetilde{B}}$ be smooth K\"ahler metrics on $B,\widetilde{B}$, respectively. Because $X$ has klt singularities, there exist effective reduced and irreducible exceptional divisors $D_i$ and constants $a_i>-1$, $b_i\geq 0$ such that 
\begin{equation}
     (\pi^{\ast}\omega)^n = \prod |s_{D_i}|^{2a_i} \omega_{\widetilde{B}}^n, \quad (\pi^*\omega_B)^n=\prod |s_{D_i}|^{2b_i} \omega_{\widetilde{B}}^n.
\end{equation}
Since the norm of $\eta$ with respect to the smooth K\"ahler metric $\omega_B$ is always bounded, away from the exceptional locus of $\pi$, we can therefore estimate
\begin{equation}
       \pi^*(|\eta|_{\omega})=|\widetilde{\eta}|_{\pi^*\omega}\leq C \operatorname{tr}_{\pi^*\omega_{B}}(\pi^{\ast}\omega)\leq 
       C \frac{(\pi^{\ast}\omega)^n}{(\pi^*\omega_B)^n}\leq C \prod |s_{D_i}|^{-C}. 
\end{equation}
 Combining these facts and an integral computation in local holomorphic coordinates then yields (after possibly shrinking $B,\widetilde{B}$) for $\delta>0$ sufficiently small and for any $q\in (1,\infty)$
\begin{align*}
    \int_{X^{\operatorname{reg}}\cap B} (\log|\eta|_{\omega}^2)_+^q \omega^n\leq C_{\delta}\int_{\widetilde{B}} \prod |s_{D_i}|^{2a_i-\delta} \omega_{\widetilde{B}}^n<\infty.
\end{align*}

In local holomorphic coordinates, we use the Cauchy--Schwarz inequality to estimate
\begin{align*}
    \Delta \log |\eta|_\omega^2 &= g^{\overline{\jmath}i}\nabla_i \nabla_{\overline{\jmath}}\log(g_{p\overline{q}}\eta^p\overline{\eta^q}) = g^{\overline{\jmath}i}\nabla_i \left( \frac{g_{p\overline{q}}\eta^p \overline{\nabla_j \eta^q}}{g_{k\overline{l}}\eta^k\overline{\eta}^{l}} \right)\\
    &= \frac{1}{|\eta|_g^4}g^{\overline{\jmath}i}\left( |\eta|_g^2 g_{p\overline{q}}\nabla_i \eta^p \overline{\nabla_j \eta^q} - g_{p\overline{q}}g_{k\overline{l}} \eta^p \overline{\eta^l} \nabla_i\eta^k \overline{\nabla_l{\eta}^q} \right) - \frac{1}{|\eta|_g^2}g^{\overline{\jmath}i} g_{p\overline{q}}\eta^p \overline{[\nabla_j,\nabla_{\overline{\imath}}]\eta^q}\\
    &\geq -\frac{1}{|\eta|_g^2}  R_{p\overline{k}}\eta^p\overline{\eta^k}.
\end{align*}
Extending $\eta$ to $X^{\operatorname{reg}}\times (-\infty,0)$ by $\mathcal{L}_{\partial_t}\eta=0$, we therefore have
\begin{equation*} (\partial_t - \Delta)(\log |\eta|)_+^2 \leq -2|\nabla (\log|\eta|)_+|^2.
\end{equation*}
in the sense of distributions. Choose $r>0$ such that $B(x_0,4r)\subseteq B$. By Lemma \ref{lem:cutoff}, we can find $C_0 <\infty$ and $\varphi \in C^{\infty}(B(x_0,4)\cap X^{\operatorname{reg}})$ such that $\varphi|_{B(x_0,2C_0^{-1}r) \equiv 1}$ and $r^2(|\nabla \varphi|^2+|\Delta \varphi|)\leq C$.  
Then we can argue as in Lemma \ref{lem:C1estimates}\ref{lem:C1estimates1}, to obtain 
\begin{equation*} (\log|\eta|)_+^2(x) \leq C\sup_{[-1,-1-C^{-1}r^2]} \int_{\operatorname{supp}(\nabla \varphi) \cap X^{\operatorname{reg}}} (\log|\eta|)_+^2 \omega^n <\infty \end{equation*}
for all $x\in B(x_0,C_0^{-1}r)$. 
\end{proof}

\begin{theorem} \label{thm:convexityofmoment} The following hold:
\begin{enumerate}
    \item The moment map $\mu$ extends to a continuous map on $X$ and is proper;
    \item  $\mu(X)$ is a closed and convex subset of $\lie(\mathbb T)^*$.
\end{enumerate}
\end{theorem}
\begin{proof}
   By Lemma \ref{lem-vector field bounded}, $\mu$ is locally Lipschitz continuous on $X^{\operatorname{reg}}$, so extends to a continuous function on $X$. The properness of $\mu$ follows from the fact that the soliton potential $f$, which is the Hamiltonian potential of $J\nabla f$, is proper.
    The closedness of $\mu(X)$ is a consequence of the properness of $\mu$. We show that $\mu(X)$ is convex using the result in \cite{HH96}. We let 
    \begin{equation}
        X_{\max}:=\{x\in X^{\reg}\mid \dim \mathbb T\cdot x \text{ is maximal }\}.
    \end{equation} Since the $\mathbb T^{\mathbb C}$-action preserves $X^{\reg}$, \cite[Internal Convexity Theorem]{HH96} gives that $\mu(X_{\max})$ is convex. Because $X_{\max}$ is an open and dense subset of $X^{\reg}$, which is open and dense in $X$, as a consequence of the continuity of $\mu$, we know that $\mu(X_{\max})$ is a dense subset of $\mu(X)$, i.e.
    \begin{equation*}
\mu(X)=\overline{\mu(X_{\max})}.
    \end{equation*}Since the closure of a convex set is convex, it follows that $\mu(X)$ is convex.
\end{proof}

As a consequence of the convexity of $\mu(X)$, it follows from \cite[Lemma 3.3]{SunZhang} that one can perturb $J\nabla f$ to obtain an element in $\lie(\mathbb{T})$ which generates an $\mathbb S^1$-action and still admits a proper Hamiltonian potential:
\begin{equation*}
    \Lambda_{\mathbb Q}:=\{\eta\in \lie(\mathbb T)\mid u_{\eta} \text{  is proper and bounded below, $\eta$ generates an $\mathbb S^1$-action}\}\neq \emptyset.
\end{equation*}

\begin{lemma} \label{lem:constancy} Let $\eta \in \operatorname{Lie}(\mathbb{T})$, and let $Z \subseteq X$ be a connected component of the vanishing set of $\eta$, viewed as a holomorphic vector field on $X$. Then the restriction $u_{\eta}\vert_Z$ is constant. In particular, if $\eta \in \Lambda_{\mathbb{Q}}$, then the set of critical values of $u_{\eta}$ is discrete and each irreducible component of the analytic set $\{\eta=0\}$ is compact.
\end{lemma}
\begin{proof} Fix $x_0 \in Z$. By Proposition \ref{prop:equivariant}, we can find a $\mathbb{T}$-invariant neighborhood $U$ of $x_0$ and $s\in H^0(U,K_X^{-\ell})$ for some $\ell\in \mathbb{N}_{>0}$ such that $\mathcal{L}_{\eta} s = \sqrt{-1} \lambda s$ for some $\lambda \in \mathbb{R}$, and $\inf_{U} |s|_{h^\ell} \geq 1$. Then $\mathcal{L}_{J\eta}s=-\lambda s$, and $\varphi:= -\frac{1}{\ell}\log |s|_{h^\ell}^2$ is a Lipschitz continuous function on $U$ such that $\omega=\ii\partial\pp \varphi$. By the definition of the moment map \eqref{eq--def of moment map}, we obtain that on $U\cap X^{\operatorname{reg}}$, we have
\begin{equation*}
    u_{\eta}=
\frac12\left(-(J\eta)\varphi- \frac{1}{\ell} \frac{\langle \mathcal{L}_{J\eta} s,s\rangle_h + \langle s,\mathcal{L}_{J\eta}s\rangle_h}{|s|_h^2}\right) = -\frac12(J\eta)\varphi+\frac{\lambda}{\ell}.
\end{equation*}
Because $x_0$ is a fixed point of the $\mathbb{T}$-action, we can choose a neighborhood $V\subseteq U$ of $x_0$ such that $\exp(\sqrt{-1}t\eta)\cdot p \in U$ for all $p\in V$ and $|t| \leq 1$. For each $p\in V$, the function
\begin{equation*}
    \mathbb{D}\to \mathbb{R}, \qquad z\mapsto \varphi(\exp(\sqrt{-1}z\eta)\cdot p)
\end{equation*}
is a strictly plurisubharmonic function on $\mathbb{D}$ which depends only on $\operatorname{Re}(z)$, hence its restriction $\varphi_p$ to $(-1,1)$ is strictly convex. 
For any $p\in V \cap X^{\operatorname{reg}}$, $\varphi_p$ is smooth, and 
\begin{equation} \label{eq:slopeandmomentmap} \varphi_p'(0)=(J\eta)\varphi(p) = \frac{2\lambda}{\ell}-2u_{\eta}(p). \end{equation}
Now fix $p\in V\cap Z$, and choose $p_j \in V\cap X^{\operatorname{reg}}$ such that $\lim_{j\to \infty}p_j=p$. Because $\varphi$ is continuous, we then have $\varphi_{p_j}\to \varphi_{p}$ uniformly on $[-\frac{1}{2},\frac{1}{2}]$ as $j\to \infty$. On the other hand, $\varphi_{p}$ is constant, so because $\varphi_{p_j}$ are convex, we have 
\begin{equation*}
    \lim_{j\to \infty}\varphi_{p_j}'(0)=0.
\end{equation*} Recalling that $u_{\eta}$ is also continuous, we may therefore replace $p$ with $p_j$ in \eqref{eq:slopeandmomentmap} and take $j\to \infty$ to obtain
\begin{equation*}
    0 = \frac{2\lambda}{\ell} - 2u_{\eta}(p)
\end{equation*}
for all $p\in V\cap Z$. Thus $u_{\eta}|_Z$ is locally constant, but $Z$ is connected, so $u_{\eta}|_Z$ is constant. 

Now assume that in addition $\eta \in \Lambda_{\mathbb{Q}}$. Let $s_j$ be a sequence of distinct critical values of $u_{\eta}$ converging to some $s_{\infty} \in \mathbb{R}$. Choose $x_j \in u_{\eta}^{-1}(s_j) \cap \{\eta=0\}$. Since $u_{\eta}$ is proper, after passing to a subsequence we may assume that $x_j \to x_{\infty}$ in $X$. Then $\eta(x_{\infty}) = 0$.
Since the vanishing locus of $\eta$ is locally connected, all points $x_j$ lie in the same connected component for $j \gg 1$, and hence $s_j = u_{\eta}(x_j)$ are all equal. This contradicts the assumption that the $s_j$ are distinct critical values of $u_{\eta}$.
Therefore, the critical values of $u_{\eta}$ must form a discrete subset of $\mathbb{R}$. Since $u_\eta$ is constant on each irreducible component of $\{\eta=0\}$ and $u_\eta$ is proper, we know that each irreducible component of the analytic set $\{\eta=0\}$ is compact.
\end{proof}

\begin{prop}\label{prop:connected-levels}
    For any $\eta \in \Lambda_{\mathbb{Q}}$ and any  $r \in \mathbb{R}$, the level set $u_{\eta}^{-1}(r)$ is connected. In particular, $X$ has only one end.
\end{prop}
\begin{proof}

We consider 
\begin{equation*}
    X_{\max}(\eta):=\{x\in X^{\reg} \mid \eta(x)\neq 0\}
\end{equation*} and 
\begin{equation}\label{stable points in regular part}
    X^{s}_{\max}(\eta, r):=\{x\in X_{\max}(\eta)\mid \mathbb C^*\cdot x\cap u_{\eta}^{-1}(r)\neq \emptyset\}.
\end{equation}
Clearly $X_{\max}(\eta)$ is connected and irreducible and  hence $ \mathbb C^*$-irreducible, i.e. it cannot be written as the union of two non-empty proper closed $\mathbb C^*$-invariant analytic subsets. Then we can apply \cite[Isotropy Stratum Density Lemma]{HH96} to the $\mathbb C^*$-action on $X_{\max}(\eta)$ to obtain that the set  $X^{s}_{\max}(\eta, r)$, if non-empty, is an open, dense, and $\mathbb C^*$-irreducible  subset of $X_{\max}(\eta)$. Note that we have used the fact that the semi-stable points coincide with the stable points since the vector field $\eta$ is nowhere vanishing on $X_{\max}(\eta)$. As a consequence, we obtain that 
\begin{equation*}
    u_{\eta}^{-1}(r)\cap X_{\max}(\eta) \text{ is connected}.
\end{equation*}Indeed, if $ u_{\eta}^{-1}(r)\cap X_{\max}(\eta)$ had two disjoint connected components $E_1$ and $E_2$, then we could show that the sets 
\begin{equation*}
    U_i:=\{x\in X_{\max}(\eta)\mid \mathbb C^*\cdot x\cap E_i\neq \emptyset\}
\end{equation*}
would be disjoint, open and $\mathbb C^*$-invariant, such that $X^{s}_{\max}(\eta, r)=U_1\cup U_2$, contradicting the fact that $ X^{s}_{\max}(\eta, r)$ is $\mathbb C^*$-irreducible.

We first show the connectedness of the level sets for regular values of $u_\eta$.  It suffices to prove that $u_{\eta}^{-1}(r)$ coincides with the closure of $u_{\eta}^{-1}(r)\cap X_{\max}(\eta)$ in $X$. Suppose not. Then there exists a point $x_0 \in u_{\eta}^{-1}(r)$ which is not contained in the closure of $u_{\eta}^{-1}(r)\cap X_{\max}(\eta)$. Since $X_{\max}(\eta)$ is dense in $X$, it follows that $x_0$ must be a local maximum or minimum of $u_{\eta}$. We claim that this implies $\eta$ must vanish at $x_0$, and hence $r = u_\eta(x_0)$ would be a critical value, which contradicts with the assumption that $r$ is a regular value. The claim is clear if $x_0 \in X^{\reg}$. 
In general, we argue as in \eqref{eq--derivative upper bound}. Let $\phi_t^{J\eta}$ denote the flow generated by $J\eta$. For $t \in (-1,1)$, the function $t \mapsto u_{\eta}(\phi_t^{J\eta}(x_0))$ is Lipschitz and satisfies, for a.e. $t \in (-1,1)$,
\begin{equation*}
    \frac{d}{dt} u_{\eta}(\phi_t^{J\eta}(x_0)) \leq -c_0 \, |\eta|^2_{\omega_0}(\phi_t^{J\eta}(x_0)).
\end{equation*} Thus $x_0$ cannot be a local maximum or a local minimum of $u_{\eta}$, which yields a contradiction.

Next, we prove the connectedness of $u_{\eta}^{-1}(r)$ for a singular value $r$ of $u_\eta$. 
\begin{itemize}
    \item \textbf{Case 1:} $r=\min_X u_\eta$. Fix $\delta>0$ such that  there is no critical value of $u_{\eta}$ in $(r,r+\delta]$. Then for any point $x\in u_{\eta}^{-1}(r+\delta)$,
    \begin{equation*}
        \lim_{t\rightarrow \infty}u_\eta(\phi_t^{J\eta}(x))=r.
    \end{equation*} By applying the \L{}ojasiewicz inequality to a Hamiltonian potential of $J\eta$ with respect to a smooth K\"ahler metric on $X$, we know that 
    \begin{equation*}
        \lim_{t\rightarrow \infty}\phi_t^{J\eta}(x) \in u_\eta^{-1}(r) \text{ exists.}
    \end{equation*} Then if $u_\eta^{-1}(r)=E_1\sqcup E_2$, where $E_i$ are non-empty closed subsets and hence compact subsets of $X$, then
\begin{equation*}
      U_i:=  \{x\in u_{\eta}^{-1}(r+\delta)\mid \lim_{t\rightarrow \infty}\phi_t^{J\eta}(x) \in E_i\}
    \end{equation*} 
would be disjoint open subsets of $u^{-1}(r+\delta)$ whose union equals $u^{-1}(r+\delta)$, contradicting the connectedness of $u_\eta^{-1}(r+\delta)$. Therefore, $u_\eta^{-1}(\min_X  u_\eta)$ is connected.
\item \textbf{Case 2:} $r>\min_X u_\eta$. We fix $\delta>0$ such that $r-\delta>\min_X u_\eta$ and $r$ is the unique critical value of $u_\eta$ inside $[r-\delta,r+\delta]$. Then we show that for any $\epsilon\in (0, \delta]$, $u_\eta^{-1}([r-\epsilon,r+\epsilon])$ is connected. Suppose by way of contradiction this fails, so that there are non-empty compact subsets $E_1,E_2$ of $u_{\eta}^{-1}([r-\epsilon,r+\epsilon])$ such that 
\begin{equation*}
    u_\eta^{-1}([r-\epsilon,r+\epsilon])=E_1\sqcup E_2,
\end{equation*}
Since the set of critical values of $u_\eta$ is discrete and level sets of regular values are connected, after relabeling $E_i$ if needed, we must have 
\begin{equation*}
    \max_{E_1}u_{\eta}=r=\min_{E_2} u_{\eta},
\end{equation*}and hence 
\begin{equation*}
    u_\eta^{-1}([r-\epsilon, r))\subset E_1 \text{ and }u_\eta^{-1}((r,r+\epsilon])\subset E_2.
\end{equation*}Since $X^s_{\max}(\eta, r-\epsilon')\neq \emptyset$ for some $\epsilon'\in (0,\epsilon)$, by \cite{HH96}, we know that $X^s_{\max}(\eta, r-\epsilon')$ is dense in $X_{\max}(\eta)$. Therefore  there exists a point $x\in u_\eta^{-1}((r,r+\epsilon])\cap X_{\max}(\eta)$ such that 
\begin{equation}\label{eq-consequence of dense}
    (\mathbb C^*\cdot x) \cap u_\eta^{-1}(r-\epsilon')\neq \emptyset.
\end{equation}Consider the trajectory $x_t=\phi_t^{J\eta}(x)$. Since $u_{\eta}$ is monotone decreasing along $x_t$, \eqref{eq-consequence of dense} implies that there exists $A>0$ such that
\begin{equation*}
  x_0\in E_2, \: x_A\in E_1, \text{ and } \{x_t\mid t\in [0,A]\}\subseteq E_1\sqcup E_2,
\end{equation*}which contradicts with the connectedness of the interval $[0,A]$. Therefore we have proved that for any $\epsilon\in (0, \delta]$, $u_\eta^{-1}([r-\epsilon,r+\epsilon])$ is connected. 

We will now conclude the connectedness of  $u_\eta^{-1}(r)$ using
\begin{equation*}
    u_\eta^{-1}(r)=\bigcap_{\epsilon\in (0,\delta]}u_\eta^{-1}([r-\epsilon,r+\epsilon]).
\end{equation*}Indeed suppose $u_\eta^{-1}(r)=E_1\sqcup E_2$, where $E_i$ are non-empty compact subsets. Then there exist disjoint open subsets $U_i$ of $X$ such that $E_i\subset U_i$. Then for $\epsilon>0$ sufficiently small, we would have 
\begin{equation*}
    u_\eta^{-1}([r-\epsilon,r+\epsilon])\subset U_1\sqcup U_2.
\end{equation*} Then by connectedness of $u_\eta^{-1}([r-\epsilon,r+\epsilon])$, there exist $U_i$ such that $$U_i\cap u_\eta^{-1}([r-\epsilon,r+\epsilon])=\emptyset,$$ which gives a contradiction.
\end{itemize}
For $\eta \in \Lambda_{\mathbb{Q}}$, the function $u_{\eta}$ is proper and bounded below and we have shown that the level sets of $u_{\eta}$ are connected. Putting these facts together, we conclude that $X$ has only one end.
\end{proof}

For $\eta \in \Lambda_{\mathbb{Q}}$ and $s\in \mathbb{R}$, we define the corresponding semistable and stable sets by
\begin{align*}
     X^{s}(\eta,s) &:= \{x\in X \: | \: \mathbb{C}^{\ast}\cdot x \cap u_{\eta}^{-1}(s)\neq \emptyset \text{ and } \dim_{\mathbb C}(\mathbb C^*\cdot x)=1\},
    \\ X^{ss}(\eta,s) &:= \{x\in X \: | \: \overline{\mathbb{C}^{\ast}\cdot x } \cap u_{\eta}^{-1}(s) \neq \emptyset\}.
\end{align*}
By the argument for \eqref{eq--limit is a critical point}, we know that any point in $\overline{\mathbb{C}^{\ast}\cdot x}\setminus \mathbb C^*\cdot x$ is a critical point of $u_{\eta}$. Therefore if $s$ is a regular value of $u_{\eta}$, then we have 
\begin{equation*}
     X^{s}(\eta,s)=X^{ss}(\eta,s).
\end{equation*}

We now recall some notions from birational geometry that will be used below. We say a normal variety $X$ is of klt type if there exists an effective $\mathbb{Q}$-divisor $B$ on $X$ such that $(X,B)$ is a klt pair.
Let $Z$ be a normal projective variety. We say that $Z$ is of \emph{Fano type} 
if there exists an effective $\mathbb{Q}$-divisor $B$ such that $(Z,B)$ is klt 
and $-(K_Z + B)$ is ample.
 By \cite[Lemma-Definition 2.6]{PrSh}, $Z$ is of Fano type if and only if there exists 
an effective $\mathbb{Q}$-divisor $B_1$ such that $(Z,B_1)$ is klt and 
$-(K_Z + B_1)$ is big and nef. 

We refer to \cite[Section 2.3]{greb-projec} and \cite{HH99} for the definition of an analytic Hilbert quotient that we are going to use in the following.
\begin{theorem} \label{thm:regvalues}
For $\eta \in \Lambda_{\mathbb{Q}}$ and $r \in \mathbb{R}$, the analytic Hilbert quotient 
$$X_r:=X^{ss}(\eta,r)//\mathbb{C}^* $$ exists and is a compact normal K\"ahler space homeomorphic to $u_{\eta}^{-1}(r)/\mathbb{S}^1$. Moreover if $r$ is a regular value of $u_\eta$, then $X_r$ is a normal projective variety of klt type.
\end{theorem}
\begin{proof}
The first statement is essentially contained in \cite{HL94} except that there they assume the K\"ahler metric is smooth on each $\mathbb C^*$-orbit. Examining the proof there, they need that each point $x_0$ admits a $\mathbb C^*$-invariant Stein neighborhood, on which one can write $\omega=\ii\partial\pp\varphi$ for some function $\varphi$ whose restriction to each orbit $\phi_t^{J\eta}(x)$ for $\eta\in \lie(\mathbb T)$ is smooth and convex. In our setting, even though $\varphi$ is not smooth, by the lower bound of $\omega$ (Proposition \ref{prop:bddlocal}) and the approximation trick used before (see \eqref{eq--derivative upper bound}), we know that 
\begin{equation*}
    \varphi(\phi_t^{J\eta} (x))\in W^{2,\infty}_{\loc}(\mathbb R) \text{ and } \varphi(\phi_t^{J\eta} (x)) \text{ is convex,}
\end{equation*}
\begin{equation*}
      \frac{d}{dt}\varphi(\phi_t^{J\eta}(x))=-u_{\eta}(\exp( tJ\eta)\cdot x),
\end{equation*}
\begin{equation*}
    \frac{d^2}{dt^2}\varphi(\phi_t^{J\eta}(\cdot x))=-\frac{d}{dt}u_{\eta}(\phi_t^{J\eta}(x))\geq c_0(t,x) \, |\eta|^2_{\omega_0}(\phi_t^{J\eta}(x))
\end{equation*}
for some locally bounded function $c_0(t,x)$. These properties suffice to run the argument in \cite{HL94} to get the existence of the analytic Hilbert quotient $X_r$,  and the desired homeomorphism. In particular, $X_r$ is a compact normal complex space. Indeed, by \cite[(3.3)]{HL94}, for the analytic Hilbert quotient
\(
    \pi:X^{ss}(\eta,r)\longrightarrow X_r,
\)
one has
\[
    \mathcal O_{X_r}
    =
    \bigl(\pi_*\mathcal O_{X^{ss}(\eta,r)}\bigr)^{\mathbb C^*}.
\]
Here the right-hand side denotes the sheaf whose sections over an open set
\(U\subset X_r\) are given by
\[
    \bigl(\pi_*\mathcal O_{X^{ss}(\eta,r)}\bigr)^{\mathbb C^*}(U)
    =
    \mathcal O_{X^{ss}(\eta,r)}\bigl(\pi^{-1}(U)\bigr)^{\mathbb C^*},
\]
that is,
\[
    \mathcal O_{X^{ss}(\eta,r)}\bigl(\pi^{-1}(U)\bigr)^{\mathbb C^*}
    =
    \left\{
    h\in \mathcal O_{X^{ss}(\eta,r)}\bigl(\pi^{-1}(U)\bigr)
    \;\middle|\;
    h(\lambda\cdot x)=h(x)
    \text{ for all }\lambda\in\mathbb C^*
    \right\}.
\]
Therefore weakly holomorphic functions on \(X_r\) pull back to
\(\mathbb C^*\)-invariant weakly holomorphic functions on \(X^{ss}(\eta,r)\).
Since \(X^{ss}(\eta,r)\) is normal, these pullbacks extend holomorphically and  descend
to holomorphic functions on \(X_r\). Hence \(X_r\) is normal.

As in \cite[Section~3]{HHL94}, by choosing $\mathbb{S}^1$-invariant local potentials for $\omega$ and descending to the quotient, we obtain that $\omega$ descends to a closed $(1,1)$-current $\omega_r$ on $X_r$ whose local potentials are continuous. By the proof of the theorem in \cite[Section~2]{HHL94}, and by considering an $\mathbb{S}^1$-invariant local smooth K\"ahler form on $X$, we obtain locally defined strictly plurisubharmonic functions in the sense of perturbation on $X_r$; see \cite[Page 124]{HHL94} for the definition. Then in order to show $\omega_r$ is a positive $(1,1)$-current, we first consider $(X^{ss}(\eta,r)\cap X^{\reg})//\mathbb{C}^*$ and use the fact that $(X^{ss}(\eta,r)\cap X^{\mathrm{sing}})//\mathbb{C}^*$ is an analytic subset of $X_r$, together with the extension property of plurisubharmonic functions \cite[Chapter I--Theorem 5.24]{demaillyjean-pierreComplexAnalyticDifferential}, we obtain that the lower bound of $\omega$ from Proposition~\ref{prop:bddlocal} descends to the quotient. Consequently, the local potentials of $\omega_r$ are also strictly plurisubharmonic in the sense of perturbation. Then, by \cite[Theorem 5.3.1]{FN1980}, these local potentials can be realized as restrictions of strictly plurisubharmonic functions under local embeddings. Therefore, $X_r$ is a compact normal Kähler space by \cite[Theorem 1]{Varouchas}.

 We need to show that $X_r$ is a projective variety of klt type if $r$ is a regular value of $u_\eta$. Since $r$ is a regular value of $u_{\eta}$ and the set of critical values is discrete, there exists $\delta > 0$ such that the analytic quotients $X_s$ are biholomorphic for all $s \in (r-\delta, r+\delta)$. Therefore, by replacing $r$ with a nearby value if necessary, we may assume $r \in \mathbb{Q}$. Then we  recall the existence of slices proved in \cite{HL94}. For any point $x_0 \in u_{\eta}^{-1}(r)$, there exists a $G_{x_0}$-invariant subvariety $S$ through $x_0$ such that the natural map
\begin{equation*}
    \mathbb{C}^* \times_{\mathbb C^*_{x_0}} S \to X
\end{equation*}
is a biholomorphism onto its open image, where $\mathbb C^*_{x_0}$ denotes the stabilizer of $x_0$ under the $\mathbb{C}^*$-action. Therefore, a neighborhood of $x_0$ in the analytic quotient $X^{s}(\eta, r)/\mathbb{C}^*$ can be identified with a neighborhood of the finite quotient $S/G_{x_0}$. 
Since the quotient map $S \to S/G_{x_0}$ is finite, Remmert's proper mapping theorem implies that $S^{\mathrm{sing}}/G_{x_0}$ is an analytic subset of $S/G_{x_0}$. Consequently, $(X^{s}(\eta, r) \cap X^{\mathrm{sing}})/\mathbb{C}^*$ is an analytic subset of $X_r$ of codimension at least 2. The computation in 
    \cite[Lemma 3.4]{SunZhang} can be done on the locus $(X^{s}(\eta, r)\cap X^{\mathrm{reg}})/\mathbb C^*$, showing that $\omega_r$ is the curvature form for some $\mathbb Q$-line bundle on $(X^{s}(\eta, r)\cap X^{\mathrm{reg}})/\mathbb C^*$. By Hironaka's resolution of singularities \cite{BM,wlodarczyk2008}, there exists a projective resolution
\begin{equation*}
    \widetilde{X_r} \to X_r.
\end{equation*}
In particular, $\widetilde{X_r}$ is a compact smooth K\"ahler manifold \cite{fujiki1978}. Since $\omega_r$ is a globally defined positive (1,1)-current, by \cite[Proposition~3]{schumacher2017}, there exists a line bundle $L$ on $\widetilde{X_r}$ equipped with a singular Hermitian metric whose curvature current is semipositive and strictly positive on a nonempty open set. By \cite{boucksom2002volume}, using Demailly's regularization theorem, it follows that $L$ is a big line bundle. Hence $\widetilde{X_r}$ is Moishezon, and since it is K\"ahler, it is projective \cite{moishezon1967n}.
Consequently, $X_r$ is Moishezon. By \cite{greb-rational}, $X_r$ has rational singularities. It then follows from \cite{Namikawa} that $X_r$ is projective. Finally, by \cite[Theorem~3]{greb-projec}, $X^s(\eta,r)$ is quasi-projective and the $\mathbb{C}^*$-action is algebraic and $X^{s}(\eta,r)\rightarrow X_r$ is an algebraic Hilbert quotient \cite[Proposition 10.1]{greb-projec} since $r$ is a regular value. Therefore, by \cite{BGLM2024}, the quotient $X^s(\eta,r)/\mathbb{C}^*$ is of klt type.
\end{proof}

    The following result is likely well known, but we include it for completeness.

\begin{lemma}\label{lem-q-cartier}
   Let $X$ be a normal projective variety of klt type, and let 
$D \subset \Omega\subset X^{\reg}$ be a smooth divisor, where $\Omega $ is a Zariski open subset such that $X\setminus \Omega$ has complex codimension at least 2. Suppose that the line 
bundle $\mathcal{O}_{\Omega}(D)$ admits a Hermitian metric $h$ with bounded local potentials,
whose curvature current $\Theta_h$ satisfies the following property: for 
any point $x \in X$, there exists a neighborhood $U$ of $x$ and a function 
$\varphi \in \mathrm{PSH}(U)$ such that
\[
\Theta_h = \sqrt{-1}\,\partial \bar{\partial} \varphi \quad \text{on } U\cap \Omega.
\]
Then $D$ extends to a $\mathbb{Q}$-Cartier divisor on $X$.
\end{lemma}
\begin{proof}
    It suffices to show that, for any $U$, with $\varphi$ as in the statement, $\mathcal{O}_{U\cap \Omega}(D)$ extends to a $\mathbb{Q}$-Cartier divisor on $X$. Consider the Hermitian metric $h e^{\varphi}$ on $\mathcal{O}_{U\cap \Omega}(D)$. Then this Hermitian metric has vanishing curvature and hence defines a smooth Hermitian metric
   on $\mathcal{O}_{U\cap \Omega}(D)$. 
    Shrinking $U$ if necessary, by \cite{braunLocalFundamentalGroup2021}, we know that $\pi_1(U\cap X^{\reg})$ is finite. Since by the assumption that $X\setminus \Omega$ has codimension at least 2 in $X^{\reg}$, we have 
    \begin{equation*}
        \pi_1(U\cap \Omega)\simeq \pi_1(U\cap X^{\reg})
    \end{equation*} 
    is also finite. Therefore there exists a positive integer $k$ such that  $\mathcal{O}_{\Omega}(kD)$ admits a parallel section $s_U$ on $U\cap \Omega$ and hence nowhere vanishing on $U\cap \Omega$.

    Let $\iota: \Omega\rightarrow X$ denote the inclusion. Then 
    \begin{equation*}
        (\iota_*(\mathcal{O}_{\Omega}(kD)))^{**}
    \end{equation*}is a rank 1 reflexive sheaf on $X$ as a consequence of the normality of $X$.  Since $s_U$ is a nowhere-vanishing section of $\mathcal{O}_{\Omega}(kD)$ over $U\cap \Omega$, it extends to an invertible section of $(\iota_*(\mathcal{O}_{\Omega}(kD)))^{**}$ on $U$ by \cite[Lemma 5.1.10]{Ishii}. It follows that the sheaf $ (\iota_*(\mathcal{O}_{\Omega}(kD)))^{**}$ is defined by a Cartier divisor.
\end{proof}
\begin{prop}\label{prop-fano type}
    For $\eta \in \Lambda_{\mathbb{Q}}$ and $r$ a regular value of $u_\eta$ sufficiently close to $0$, the variety $$X_r:=X^s(\eta,r)/\mathbb{C}^* $$ is of Fano type. 
\end{prop}
\begin{proof}
 Fix $\delta > 0$ such that the interval $(-\delta, \delta)$ contains at most one critical value of $u_\eta$. In the following, we consider the case where $r < 0$, as the case $r > 0$ is analogous. Then for any $r \in (-\delta, 0)$, the varieties $X_r$ are isomorphic. Hence, we may assume that $r \in \mathbb{Q}$.
    
    The argument in \cite[Section 3]{SunZhang} and \cite[Section 3]{HNP} shows that on the locus 
    \begin{equation}
        (X^s(\eta,r)\cap X^{\reg})/\mathbb{C}^* 
    \end{equation} the $\mathbb Q$-line bundle $-(K_{X_r}+D_r)+rL_r$ admits a Hermitian metric whose curvature form extends to $X_r$ as a K\"ahler current with continuous potentials. Moreover we know that $(X^s(\eta,r)\cap X^{\reg})/\mathbb{C}^*$ is a Zariski open set in $X_r$ whose complement is the holomorphic image of a $\mathbb{C}^{\ast}$-invariant analytic subset of codimension at least 2, hence itself has codimension at least 2. Therefore we can apply Lemma \ref{lem-q-cartier} to obtain that it extends to $X_r$ as a $\mathbb Q$-Cartier divisor. By choosing two different values of $r$, and using the fact that $\mathbb Q$-line bundle $K_{X_r}+D_r$ and $L_r$ are indeed independent of $r$, we know that both 
       $K_{X_r}+D_r$
   and $L_r$ extend as $\mathbb Q$-Cartier divisors on $X_r$. Moreover, by the slicing theorem and the formula for the canonical bundle under finite group quotients, see for example \cite[Proposition 2.18]{moraga2021fano}, there exists an effective $\mathbb{Q}$-divisor $D'_r$ on $X_r$ such that $(X_r, D'_r)$ is a klt pair, and $D'_r$ coincides with $D_r$ on 
\(
\bigl(X^s(\eta,r)\cap X^{\mathrm{reg}}\bigr)/\mathbb{C}^*.
\)
In particular, $(X_r, D_r)$ is a klt pair.

By \cite[Theorem 1]{Varouchas}, $c_1(-(K_{X_r}+D_r)+rL_r)$ is a K\"ahler class. Since $r$ can be taken arbitrarily small, it follows that $-(K_{X_r}+D_r)$ is a nef class. By Proposition \ref{prop:birational} proved below, we know that the volume of the K\"ahler current in the class $c_1(-(K_{X_r}+D_r)+rL_r)$ has a uniform lower bound independent of $r\in (-\delta,0)$, we know that the class $-(K_{X_r}+D_r)$ is  big \cite[Theorem 2.12]{DP04}, hence $X_r$ is of Fano type.
\end{proof}

In what follows, we use the notion of proper modifications between reduced complex spaces; see \cite[Definition 2.1.17]{MaMar}.
\begin{prop} \label{prop:birational} Fix $\eta \in \Lambda_{\mathbb{Q}}$, and equip $X$ with the corresponding effective $\mathbb{C}^{\ast}$-action. Fix $s_- <s_0 <s_+$ such that $s_{\pm}$ are regular values of $u_{\eta}$ with $u_{\eta}^{-1}(s_{\pm})\neq \emptyset$, and $s_0$ is the unique critical value of $u_{\eta}$ in $(s_-,s_+)$. Then there are proper modifications
\begin{equation*}
    \pi_{\pm}: X_{s_{\pm}}\to X_{s_0} 
\end{equation*}
with the following properties.
\begin{enumerate}
    \item Let $Z_1,\ldots,Z_N$ be the irreducible components of $\{\eta=0\}$ which are contained in $u_{\eta}^{-1}(s_0)$. Then 
\begin{equation*} Z_i^{-} := \{x\in X^s(\eta,s_-) \: | \: \lim_{z\to 0} z\cdot x \in Z_i\}, \qquad Z_i^+ := \{x \in X^s(\eta,s_+) \: | \: \lim_{z\to \infty} z\cdot x\in Z_i\} \end{equation*}
are $\mathbb{C}^{\ast}$-invariant complex subvarieties of $X^s(\eta,s_-)$ and $X^s(\eta,s_+)$, respectively.
\item $\dim_{\mathbb C}Z_i^{\pm}<n$, $i=1,\ldots, N$.
\item $E_i^{\pm} := Z_i^{\pm}//\mathbb{C}^{\ast}$ can be identified with complex subvarieties of $X_{s_{\pm}}$, $Z_i$ can be viewed as subvarieties of $X^{ss}(\eta,s_0)//\mathbb{C}^{\ast}$, and with these identifications, we have $\pi_{\pm}(E_i^{\pm}) = Z_i$, and $\pi_{\pm}$ restrict to biholomorphisms
\begin{equation*}
    X_{s_{\pm}}\setminus \cup_{i=1}^N E_i^{\pm} \to X_{s_0}\setminus \cup_{i=1}^N Z_i.
\end{equation*}
\end{enumerate} 
\end{prop} 

\begin{proof} 
It suffices to prove the above statements for $X^s(\eta,s_-)$ (the case of $X^s(\eta,s_+)$ is similar). We first claim that $X^s(\eta,s_-) \subseteq X^{ss}(\eta,s_0)$. Let $(\phi_t^{J\eta})_{t\in \mathbb{R}}$ be the flow generated by $J\eta$. Then $t\mapsto u_{\eta}(\phi_t^{J\eta}(x))$ is strictly decreasing, hence we can set $s_{\infty}:= \sup_{t<0}u_{\eta}(\phi_t^{J\eta}(x))$. Because $X^s(\eta,s)=X^s(\eta,s_-)$ for all $s\in [s_-,s_0)$, we must have $s_{\infty} \geq s_0$. If $s_{\infty}>s_0$, then $x\in X^s(\eta,s_0)$. Suppose instead that $s_{\infty}=s_0$, so that there exist $z_i \to 0$ such that $x_{\infty} := \lim_{i \to \infty} z_i \cdot x \in u_{\eta}^{-1}(s_0)$. Then $x \in X^{ss}(\eta,s_0)\setminus X^s(\eta,s_0)$, and in particular, $X^s(\eta,s_-)\subseteq X^{ss}(\eta,s_0)$. 

For $x\in X^s(\eta,s_-)\setminus X^s(\eta,s_0)$, we now show that any subsequential limit $x_{\infty} = \lim_{i\to \infty} z_i \cdot x$ actually satisfies $x_{\infty} = \lim_{z\to 0} z\cdot x$ and $x_{\infty} \in Z_i$ for some $i\in \{1,\ldots,N\}$.
By \cite[Theorem 3.3.14]{HH99} and the fact that any point of $X$ admits a $\mathbb{T}^{\mathbb{C}}$-invariant Stein neighborhood \cite[Theorem 4.1.4]{HH99}, there is a $\mathbb{T}^{\mathbb{C}}$-invariant neighborhood $U$ of $x_{\infty}$ in $X$ which is $\mathbb{C}^{\ast}$-equivariantly embedded as a complex analytic subspace of a linear $\mathbb{T}^{\mathbb{C}}$ representation on some finite-dimensional complex vector space $V$. It follows that $z_i \cdot x\in U$ for sufficiently large $i\in \mathbb{N}$, hence $x\in U$ as well. Because $\mathbb{C}^{\ast}\cdot x \subseteq \{u_{\eta}\leq s_0\}$, it follows that $\mathbb{C}^{\ast} \to V$, $z\mapsto z\cdot x$ is bounded, so that the Riemann extension theorem gives $\lim_{z\to 0}z\cdot x =x_{\infty}$, and it follows that $x_{\infty}$ is a fixed point of the $\mathbb{C}^{\ast}$-action; that is, $x_{\infty} \in Z_i$ for some $i\in \{1,...,N\}$. 

Next, we prove that $Z_i^-$ are complex subvarieties of $X^s(\eta,s_-)$. By a diagonal argument, we obtain that $Z_i^-$ are closed subsets of $X^s(\eta,s_-)$, so it suffices to show that for any $x\in Z_i$, there exists a neighborhood $U$ of $x$ such that $U\cap Z_i^-$ is a complex subvariety. Letting $x_{\infty},U,V$ be as in the previous paragraph, we know that $Z_i \cap U$ corresponds to an irreducible component of the intersection of $U$ with the subspace of $V$ fixed by the $\mathbb{C}^{\ast}$ action. It follows that $U\cap Z_i^-$ is an irreducible complex analytic subset of $U$, hence $Z_i^-$ is an irreducible complex analytic subset of $X^s(\eta,s_-)$. Clearly $Z_i^-$ is also $\mathbb{C}^{\ast}$-invariant, and by \cite[Section 6.4]{H91} (see also \cite[Section 1]{HMP} and \cite[Section 2.3]{greb-projec}), we know the image of $Z_i^-$ under the quotient map $X^s(\eta,s_-)\to X_{s_-}$ is a complex subvariety $E_i^-\subseteq X_{s_-}$ isomorphic to $Z_i^-//\mathbb{C}^{\ast}$.

We now show that 
\begin{equation*}
    \dim_{\mathbb C} Z_i^{-}<n.
\end{equation*}
Suppose not. Then there exists a nonempty open subset $U \subseteq X$ contained in $Z_{i}^-$. For any $r>s_0$, we would then have
\begin{equation}\label{nonintersection property}
    U\cap X^{ss}(\eta, r)=\emptyset,
\end{equation}
hence for any $x\in U$, 
\begin{equation*}
    \overline{\mathbb C^{\ast}\cdot x}\cap u_\eta^{-1}(r)=\emptyset.
\end{equation*}
We can choose $r$ such that $X^s_{\max}(\eta, r)\neq \emptyset$, where $X_{\max}^s(\eta,r)$ is as in \eqref{stable points in regular part}. Then \eqref{nonintersection property} would contradict \cite[Isotropy Stratum Density Lemma]{HH96}, which says that $X^s_{\max}(\eta, r)$ is a dense subset of $X_{\max}$ and hence $X^{s}(\eta, r)$ is a dense subset of $X$.

The composition $X^s(\eta,s_-)\hookrightarrow X^{ss}(\eta,s_0) \to X^{ss}(\eta,s_0)//\mathbb{C}^{\ast}$ is $\mathbb{C}^{\ast}$-invariant, hence the universal mapping property of $X^s(\eta,s_-)//\mathbb{C}^{\ast}$ gives a holomorphic map 
\begin{equation*} \pi_-:X_{s_-}=X^s(\eta,s_-)//\mathbb{C}^{\ast} \to X^{ss}(\eta,s_0)//\mathbb{C}^{\ast}=X_{s_0}.\end{equation*}
We claim that $X^s(\eta,s_0) \subseteq X^s(\eta,s_-)$. In fact, if $x\in X^s(\eta,s_0)$, then there exists $t\in \mathbb{R}$ such that $\eta|_{\phi_t^{J\eta}(x)}\neq 0$ and $u_{\eta}(\phi_t^{J\eta}(x))=s_0$. It follows that $u_{\eta}(\phi_{t'}^{J\eta}(x))<s_0$ for all $t'>t$, hence (because there are no zeros of $\eta$ in $u_{\eta}^{-1}([s_-,s_0))$) $x\in X^s(\eta,s_-)$. By the universal mapping property of $X^{ss}(\eta,s_0)//\mathbb{C}^{\ast}$, the composition $X^s(\eta,s_0)\hookrightarrow X^s(\eta,s_-)\to X^s(\eta,s_-)//\mathbb{C}^{\ast}$ induces a holomorphic map
\begin{equation*}
    \sigma:X^{s}(\eta,s_0)//\mathbb{C}^{\ast} \hookrightarrow X_{s_-},
\end{equation*}
which satisfies $\pi_-\circ \sigma = \operatorname{id}$, where we identify $X^s(\eta,s_0)//\mathbb{C}^{\ast}$ with a dense open subset of $X_{s_0}$. We may moreover identify its complement with the image of $\cup_{i=1}^N (Z_i^- \cup Z_i^+)$, which is an analytic subvariety of $X_{s_0}$. Moreover, the image of $Z_i^- \cup Z_i^+$ equals the image of the intersection $Z_i = (Z_i^- \cup Z_i^+)\cap u_{\eta}^{-1}(s_0)$, but because $Z_i$ is a $\mathbb{C}^{\ast}$-invariant analytic subset of fixed points, the quotient map restricts to a biholomorphism onto its image. 
\end{proof}

\begin{remark}
Note that the results in Proposition \ref{prop:birational} give another proof of the fact that connectedness of $u_\eta^{-1}(r)$ for regular values implies the connectedness of $u_\eta^{-1}(r)$ for singular values. Indeed because $\pi_-$ is a holomorphic map of normal compact complex spaces whose image is the complement of a proper complex analytic set, it is surjective. Because $X^s(\eta,s_-)//\mathbb{C}^{\ast}$ is connected, it follows that $X^{ss}(\eta,s_0)//\mathbb{C}^{\ast}$ is connected as well. Because the restriction of $X^{ss}(\eta,s_0)\to X^{ss}(\eta,s_0)//\mathbb{C}^{\ast}$ to $u_{\eta}^{-1}(s_0)$ is a topological quotient map whose image and fibers are connected, it follows that $u_{\eta}^{-1}(s_0)$ is also connected.
\end{remark}

\begin{prop} \label{prop:locallyalgebraic} For any $\eta \in \Lambda_{\mathbb{Q}}$ and $x_0\in X$ not a minimal point of $u_\eta$, there exists a $\mathbb C^*$-invariant Zariski open neighborhood $V\subseteq X$ of $x_0$ which is $\mathbb{C}^{\ast}$-equivariantly isomorphic to a quasiprojective algebraic variety with an algebraic $\mathbb{C}^{\ast}$-action.
\end{prop}
\begin{proof} By Theorem \ref{thm:regvalues}, we know that for $r\in \operatorname{Int}(u(X))$,  the analytic Hilbert quotient $X^{ss}(\eta,r) //\mathbb C^*$
exists and is a compact normal K\"ahler space. 
By Proposition \ref{prop:birational}, $X^{ss}(\eta,r)//\mathbb{C}^{\ast}$ is bimeromorphic to $X^s(\eta,r')//\mathbb{C}^{\ast}$, where we choose $r'\in \mathbb{Q}$ to be any regular value of $u_{\eta}$ for which $u_{\eta}^{-1}(r')\neq \emptyset$. By Theorem \ref{thm:regvalues}, $X^{ss}(\eta,r)//\mathbb{C}^{\ast}$ is therefore a Moishezon space. Because $X^{ss}(\eta,r)//\mathbb{C}^{\ast}$ is a Moishezon and K\"ahler space, and since  \cite{greb-rational} implies that $X^{ss}(\eta,r)//\mathbb{C}^{\ast}$ has 1-rational singularities, we can apply \cite[Theorem 1.6]{Namikawa} to conclude that $X^{ss}(\eta,r)//\mathbb{C}^{\ast}$ is projective. Thus $X^{ss}(\eta,r)//\mathbb{C}^{\ast}$ is projective for any $r\in \operatorname{Int}(u(X))$. 

The complement of $X^{ss}(\eta,u_{\eta}(x_0))$ is a locally finite union of subsets of the form 
\begin{equation*}
\{x\in X\: | \: \lim_{t \to -\infty} \phi_t^{J\eta}(x)\in Z\}
\end{equation*}
where $Z \subseteq X$ is an irreducible component of the zero set of $\eta$. Because $\{ X^{ss}(\eta,s')\}_{s'\in \mathbb{R}}$ is an open cover of $X$, in order to show that $X^{ss}(\eta,u_{\eta}(x_0))$ is Zariski open, we need only show that each $Z\cap X^{ss}(\eta,s')$ is a complex analytic subset of $X^{ss}(\eta,s')$. However, this follows immediately from the proof of Proposition \ref{prop:birational}. Thus \cite[Theorem 3]{greb-projec} gives the claim for $V=X^{ss}(\eta,u_{\eta}(x_0))$. 
\end{proof}
\begin{corollary}\label{cor-birational}
    In the setting of Proposition \ref{prop:birational}, $X_{s_-}$ is birational to $X_{s_+}.$
\end{corollary}
\begin{proof}
By Theorem \ref{thm:regvalues} and the proof of Proposition \ref{prop:locallyalgebraic}, we know that $X_{s_-}$, $X_{s_0}$, and $X_{s_+}$ are normal projective varieties. By the GAGA principle, the holomorphic proper modifications $\pi_{\pm}$ in Proposition \ref{prop:locallyalgebraic} are algebraic morphisms. In particular, they induce a birational map between $X_{s_-}$ and $X_{s_+}$.
\end{proof}
\begin{theorem} \label{thm:simplyconnected} $X$ is simply connected.
\end{theorem}
\begin{proof}
Because $X$ is normal, \cite[Proposition 2.10]{kollar2014} and \cite{FL} give that the map
\begin{equation*}
    \pi_1(X\setminus Z)\to \pi_1(X)
\end{equation*}
induced by inclusion is surjective for any closed analytic subset $Z$. Let 
\begin{equation}
    z_\eta:=\min_{X}u_{\eta}\in \mathbb R.
\end{equation}Note that $u_\eta^{-1}(z_\eta)$ is an analytic subset of $X$ and set
    $Z = X^{\mathrm{sing}} \cup u_\eta^{-1}(z_\eta)$.
Then it suffices to show that any loop 
\[
\gamma : \mathbb{S}^1 \to X^{\mathrm{reg}} \setminus u_\eta^{-1}(z_\eta)
\]
is contractible in $X$.

Letting $(\phi_t^{J\eta})_{t\in \mathbb{R}}$ be the flow generated by $J\eta$, the set
\begin{equation*}
    Y:=\{x \in X^{\operatorname{reg}} \mid \lim_{t \to \infty} u_{\eta}(\phi_t^{J\eta}(x))>z_\eta\}
\end{equation*}
is a locally finite union of locally closed complex subvarieties of $X$ of codimension at least 1. By the Thom transversality Theorem, we may therefore assume $\gamma$ is a smooth curve satisfying $(Y\cup Z)\cap \gamma(\mathbb{S}^1)=\emptyset$. Noting that $X^{\operatorname{reg}}\setminus (Y\cup Z)$ is preserved by $\phi_t^{J\eta}$, we may therefore flow by $\phi_t^{J\eta}$ in order to assume that $\gamma$ lies entirely in $$u_{\eta}^{-1}((z_\eta, z_\eta+2\delta)),$$ where $\delta >0$ is small such that there is no critical value of $u_{\eta}$ in $(z_\eta, z_\eta+2\delta]$. Therefore $u_{\eta}^{-1}((z_\eta, z_\eta+2\delta))$ admits a deformation retraction to $u_{\eta}^{-1}( z_\eta+\delta)$.

Set $U:= X^{ss}(\eta,z_\eta)$, let $\iota:X^s(\eta,z_\eta+\delta)\to U$ be the inclusion, and let $\rho_{\delta}:X^s(\eta,z_\eta+\delta) \to X_{z_\eta+\delta}$, $\rho_0 : U \to X_{z_\eta}$ be the analytic Hilbert quotient maps. Because $$\rho_0 \circ \iota: u_{\eta}^{-1}(z_\eta+\delta) \to X_{z_\eta}$$ is a holomorphic map constant on $\mathbb{C}^{\ast}$-fibers, the mapping property of $X_{z_\eta+\delta}$ \cite[Section 1.3]{greb-projec} induces a unique holomorphic map $\pi$ making the following diagram commute:
\[
\begin{tikzcd}
X^s(\eta, z_\eta+\delta) \arrow[r, hook, "\iota"] \arrow[d, "\rho_{\delta}"] & U \arrow[d, "\rho_0"] \\
X_{z_\eta+\delta} \arrow[r, "\pi"'] & u_{\eta}^{-1}(z_{\eta})
\end{tikzcd}
\]
Because $u_{\eta}^{-1}(z_\eta)$ is contained in the fixed point set of the $\mathbb{C}^{\ast}$-action, it follows that $X_{z_\eta}$ is homeomorphic to $u_{\eta}^{-1}(z_\eta)$, and 
\begin{equation*}
    \pi(x) = \lim_{t\to \infty} \phi_t^{J\eta}(x)
\end{equation*}
for all $x\in U$. To see that this map is continuous, note that under a local equivariant embedding as in Proposition \ref{prop:birational}, $\pi$ is the restriction of a linear projection map. Similarly, we obtain that $\pi_t(x):= \phi_{\frac{t}{1-t}}^{J\eta}(x)$ for $(x,t)\in U\times [0,1)$, $\pi_1 :=\pi$ defines a strong deformation retraction of $U$ onto $u_{\eta}^{-1}(z_{\eta})$. In particular, $\pi$ induces an isomorphism $\pi_1(U)\to \pi_1(u_{\eta}^{-1}(z_{\eta}))$.

Consider the quotient at a regular value $r$ sufficiently close to $0$. By Proposition \ref{prop-fano type}, we know that $X_r$ is of Fano type. In particular, $X_r$ is simply connected \cite{takayama2003,QZhang,HM07}. 
By Corollary~\ref{cor-birational} together with \cite[Corollary 1.1]{takayama2003}, it follows that $X_{z_\eta+\delta}$ is also simply connected. By the commutative diagram, the image  
\[
\mathrm{Im}(\iota_*:\pi_1\big(X^s(\eta,z_\eta+\delta)\big)\rightarrow \pi_1(U))
\] is trivial. Consequently, the curve $\gamma$ is contractible in $U \subset X$.
\end{proof}

\begin{remark}
When X is smooth, the above argument provides a new proof that a smooth complete K\"ahler-–Ricci shrinker is simply connected, compared to \cite{esparza2,SunZhang}. In particular, it does not rely on \cite{wylie2008}.
\end{remark}

\begin{proof}[Proof of Theorem \ref{thm:additionalstructure}] Assertions \ref{thm:additional2} and \ref{thm:additional1}  follow from Theorem \ref{thm:convexityofmoment}, Theorem \ref{thm:simplyconnected}. 
\end{proof}

\section{Polarized Fano fibrations}\label{sec:fano-fibration}
In this section, we prove Theorem \ref{thm:pffibration} using the results established in the previous sections.

\subsection{Setup and analytic inputs}\label{sec:setup}

Let $(X,d)$ be a singular K\"ahler--Ricci shrinker arising as a limit under either assumption \ref{assume:KRF} or \ref{assume:KRS} of Section \ref{sec;convergence setting}.  By
Theorem~\ref{thm:complexstructure}, $X$ is a normal complex analytic variety with log terminal
singularities.
For a normal analytic or algebraic space $V$ and $m\in\mathbb Z$, we write
$\cO_V(mK_V)$ for the $m$-th reflexive canonical power, using the reflexive
dual when $m<0$.  When a complex Lie group acts holomorphically on $V$, this sheaf carries the canonical
pullback linearization induced on $V^{\rm reg}$ and extended reflexively.

\begin{definition}[\cite{SunZhang}]
\label{def:polarized-fano-fibration}
A polarized Fano fibration is a triple
$(\pi\colon X\to Y,\xi)$ with the following properties.
The spaces $X$ and $Y$ are normal quasi-projective varieties, $\pi$ is
surjective and projective, $\pi_*\cO_X=\cO_Y$, the variety $X$ is klt,
and $-K_X$ is a relatively ample $\mathbb{Q}$-Cartier divisor. There is a compact torus
$\mathbb T$ acting on $X$ and $Y$ by biholomorphisms such that $\pi$ is $\mathbb T$-equivariant,
$Y=\Spec R$ is affine with a unique $\mathbb T$-fixed
point, and $\xi\in\lie(\mathbb T)$ lies in the Reeb cone: if
\begin{equation*}
        R=\bigoplus_{\alpha\in \lie(\mathbb T)^*} R_\alpha
\end{equation*}
is the $\mathbb T$-weight decomposition, then
$\langle\alpha,\xi\rangle>0$ for every nonzero weight with
$R_\alpha\ne0$.
\end{definition}

To prove Theorem \ref{thm:pffibration}, we may assume that $X$ is noncompact and $n\ge2$.  Indeed, suppose that $X$ is compact.  By
Proposition~\ref{prop:bddlocal}, $\omega$ has continuous
strictly plurisubharmonic local potentials and is bounded below, in any local holomorphic embedding, by the ambient Euclidean K\"ahler form. 
Choose $m>0$ such that $-mK_X$ is Cartier. The current $m\omega$
defines a positive continuous Hermitian metric on $-mK_X$; Richberg regularization
and Grauert's extension of the Kodaira embedding theorem
\cite[\S3]{Gra62} show that $-mK_X$ is ample.
Thus $X$ is projective Fano and $(X\to\{\mathrm{pt}\},\xi)$ is a polarized
Fano fibration.  

Throughout this section, let $\mathbb T$ denote the compact torus of holomorphic isometries obtained as the closure of the one-parameter group generated by the soliton vector field $\xi=J\nabla f$. By Theorem~\ref{thm:convexityofmoment}, the $\mathbb T$-action has a continuous
proper moment map
\(
        \mu\colon X\longrightarrow\lie(\mathbb T)^*
\)
with closed convex image. By definition, $\langle \mu, \eta \rangle = u_{\eta}$, where $u_{\eta}|_{X^{\operatorname{reg}}}$ is defined by \eqref{eq--def of moment map}. 
Then we have
\begin{equation}\label{eq:drift-Hamiltonian}
        \Delta_fu_\eta+u_\eta=0
\end{equation}
on $X^{\operatorname{reg}}$. Because $\mu(X)$ is closed and convex, the argument of \cite[Lemma~3.1]{SunZhang} gives a rational
element $\eta_0\in\lie(\mathbb T)$, arbitrarily close to $\xi$, such that
there is a constant $C=C(\eta_0)$ such that
\begin{equation}\label{lem:canonical-growth}
      C^{-1}f-C   \leq u_{\eta_0}\le C(1+f-\min_Xf).
\end{equation}

\begin{lemma}\label{lem:weighted-centering} Let $d\nu=e^{-f}\omega^n$.
The Hamiltonian potentials satisfy
\begin{equation*}
 u_\eta,\ \nabla u_\eta\in L^2(X,\nu),
 \qquad
 \int_Xu_\eta d\nu=0.
\end{equation*}
\end{lemma}
\begin{proof}
The potential estimate and volume-ratio upper bound for a singular shrinker
\cite[Theorem~1.1 and Corollaries~1.2--1.3]{CMZ2024} imply
\[
        \int_X(1+f-\min_Xf)^k e^{-f}\omega^n<\infty
        \quad\text{for every }k.
\]
Together with \eqref{lem:canonical-growth}, this implies $u_\eta\in L^2(X,\nu)$.

We will show $\nabla u_{\eta}$ using \eqref{eq--def of moment map} and  integration by parts. Let
$\chi_\epsilon$ be
the cutoff functions of Lemma~\ref{lem:chicutoff}, and let $\rho_R$ be
as in Lemma~\ref{lem--rho-cutoff}, so that
$\rho_R |_{\{f\le R^2\}}\equiv1$, $\operatorname{supp}(\rho_R)\subseteq \{f\le2R^2\}$, and
\begin{equation*}
        |\nabla\rho_R|\le C R^{-1},\qquad
        |\Delta\rho_R|\le C R^{-2}.
\end{equation*}
Set $\psi_{\epsilon,R}:=\chi_\epsilon\rho_R$.  Multiplying
\eqref{eq:drift-Hamiltonian} by
$\psi_{\epsilon,R}^2u_\eta$ and integrating by parts gives the
estimate
\[
 \int_{X^{\rm reg}}\psi_{\epsilon,R}^2
             |\nabla u_\eta|^2\,d\nu
 \le
 C\int_{X^{\rm reg}}
       \bigl(\psi_{\epsilon,R}^2+|\nabla\psi_{\epsilon,R}|^2\bigr)
       u_\eta^2\,d\nu .
\]
For fixed $R$, choose $\sigma\in(0,2)$.  Letting $\epsilon\rightarrow 0$, the  gradient tends to zero in
$L^{4-\sigma}$ on compact sets and hence in $L^2$ there, while
$u_\eta$ is locally bounded.  After letting $\epsilon\downarrow0$, the
$\rho_R$ error tends
to zero as $R\to\infty$ by its derivative bound, the growth estimate for
$u_\eta$, and the finite weighted moments of $f$.  Therefore we obtain $\nabla u_\eta\in L^2(d\nu)$.

We now integrate \eqref{eq:drift-Hamiltonian} against
$\psi_{\epsilon,R}$.  Weighted integration by parts yields
\[
 \int_{X^{\rm reg}}\psi_{\epsilon,R}u_\eta\,d\nu
 =
 \frac12\int_{X^{\rm reg}}
 \langle\nabla\psi_{\epsilon,R},\nabla u_\eta\rangle\,d\nu .
\]
 The
right-hand side tends to zero, first as
$\epsilon\downarrow0$ and then as $R\to\infty$, by the cutoff estimates
and the $L^2$ bound just proved.  The left-hand side converges to
$\int_Xu_\eta\,d\nu$, proving the claim.
\end{proof}

After a positive rescaling, we
may assume that $\eta:= a\eta_0$ is the primitive integral generator of the corresponding effective circle action, so that the weights of the induced action on all sections on $X$ are integers. We write $u:=u_\eta$. Since $X$ is
connected and noncompact, the properness of $u$ gives $u(X)=[\min_Xu,\infty)$,
so every critical value above the minimum has nonempty level sets on both sides. Let $\rho_\lambda$ be the holomorphic action of $\C^*$ on $X$ such that $(\rho_{e^{\sqrt{-1}\theta}})_{\theta \in \mathbb{R}}$ is generated by $\eta$.  Holomorphicity
of the action and $\eta=J\nabla u$ give
\begin{equation}\label{eq:real-subgroup-direction}
 \left.\frac{d}{ds}\right|_{s=0}\rho_{e^s}
      =-J\eta=\nabla u,
 \qquad
 \frac{d}{ds}u(\rho_{e^s}x)=|\nabla u|^2(\rho_{e^s}x).
\end{equation}

Fix a regular value $r$ of $u$.  By Theorem~\ref{thm:regvalues}, every regular
reduction is a normal projective variety of klt type.  We write
\[
        X_r:=X^s(\eta,r)/\C^*, \quad U_r:=X^s(\eta,r) .
\]
 Its quotient boundary is
\[
        D_r:=\sum_B\left(1-\frac1{e_B}\right)B,
\]
where $e_B$ is the stabilizer order of the general orbit above $B$. For different $r$, $X_r$ are birational to each other and  we identify their function fields.

 We use the geometric pullback convention: a function
or form has weight
$k$ if
\[
        \rho_t^*f=f\circ\rho_t=t^kf\qquad(t\in\C^*).
\]
Let $\widetilde L_r:=\cO_{U_r}e$ be the trivial line bundle with chosen
nowhere-vanishing global frame $e$, and equip it with the inverse-standard
$\C^*$-linearization.  Explicitly, for each $t\in\C^*$, let
\[
        \lambda_t\colon \rho_t^*\widetilde L_r
              \xrightarrow{\ \sim\ }\widetilde L_r
\]
be the bundle isomorphism over $U_r$ determined by
\[
        \lambda_t\bigl(\rho_t^*e\bigr)=t^{-1}e.
\]
Thus $\rho_t$ is the point action on $U_r$, whereas $\lambda_t$ is its
linearization on the trivial line bundle, with fiber character $t^{-1}$.
Consequently, if $f$ has weight $k$, then
$f e^{\otimes k}$ is invariant. 

Choose $m$ divisible by all stabilizer orders.  Such an $m$ exists because
the  quotient is compact.
The descent theorem for equivariant line bundles
\cite[\S4, Proposition~4.2, pp.~83--84]{KKV89} then gives a line bundle
$L_r^{(m)}$ on $X_r$ whose pullback is the linearized line bundle
$\widetilde L_r^{\otimes m}$.  Set
\[
        L_r:=\frac1m L_r^{(m)}\in\operatorname{Pic}(X_r)_{\Q}.
\]
By the uniqueness of descent, it is easy to show that $L_r$ is independent of $m$.
The following result follows from Theorem~\ref{thm:regvalues}, its proof,
and Proposition~\ref{prop-fano type}.

\begin{prop}
\label{prop:regular-reduction}
For every regular value $r$, $(X_r,D_r)$ is klt, and there exists a K\"ahler current $\omega_r$ with local continuous potentials in the K\"ahler class
\begin{equation}\label{eq:ample-class}
       A_r:= -(K_{X_r}+D_r)+rL_r.
\end{equation}

\end{prop}

\subsection{Wall crossing on a singular total space}\label{sec:wall-crossing}

Let $c>0$ be a critical value of $u$ 
and choose $\varepsilon>0$ so that $c\pm\varepsilon$ are regular values with
nonempty levels and $c$ is the unique critical value in
$(c-\varepsilon,c+\varepsilon)$.  Write
\[
        U_-:=X^s(\eta,c-\varepsilon),\quad
        U_+:=X^s(\eta,c+\varepsilon),\quad U_c:=X^{ss}(\eta,c),\quad
 \varpi_c\colon U_c\longrightarrow X_c,
\]
\[
        X_-:=X_{c-\varepsilon},\qquad
        X_+:=X_{c+\varepsilon}.
\]
By Proposition~\ref{prop:birational}, there are proper modifications
$\pi_\pm\colon X_\pm\to X_c$ to a compact quotient $X_c$.  Let $W$ be the
normalization of the graph of the wall-crossing map, with projections
$p_\pm\colon W\to X_\pm$.  Thus
\begin{equation}\label{diag:wall-crossing}
\begin{tikzcd}[column sep=large,row sep=large]
 & W \arrow[dl,"p_-"'] \arrow[d,"g"] \arrow[dr,"p_+"] & \\
 X_- \arrow[r,"\pi_-"] & X_c & X_+ \arrow[l,"\pi_+"']
\end{tikzcd}
\end{equation}
and
\[
g=\pi_-\circ p_-=\pi_+\circ p_+.
\]
There is one change of convention.  The radial parameter used in the proof
of that proposition has generator $J\eta$, whereas our parameter $\lambda$
has generator $-J\eta$ by \eqref{eq:real-subgroup-direction}. 
Thus $Z_i^-$ in Proposition~\ref{prop:birational} is the
$\lambda\to\infty$ stratum on $X_-$, while $Z_i^+$ is the
$\lambda\to0$ stratum on $X_+$.

The wall-crossing argument has two ingredients, both of which are proved without assuming the smoothness of the critical locus.  The first is the descent of certain line bundles to the critical quotient.   Using the canonical-weight formula proved below, we
construct a line bundle $A_{c,N}$ on $X_c$ such that
\[
 \cO_{X_\pm}\!\left(N(K_{X_\pm}+D_\pm-cL_\pm)\right)
 \simeq \pi_\pm^*A_{c,N}.
\]
 The second ingredient
is the effectivity of the \textit{character jump}
\[
             E_c:=p_-^*L_--p_+^*L_+,
\]
proved by combining $g_*E_c\ge0$ from the graded wall algebra, the
$g$-nefness of $-E_c$ from relative ampleness, and the negativity lemma.

\subsubsection{Canonical weights at fixed components}

We first record the relation between the  character of the canonical bundle at a fixed
component and the value of the  Hamiltonian potential. 
Let $Z$ be an irreducible component of the fixed locus of the
$\C^*$-action generated by $\eta$.  By
Lemma~\ref{lem:constancy}, the Hamiltonian is constant on $Z$; denote its value
by $c_Z$.  Fix a point $z\in Z$ and an integer $m>0$ such that
$mK_X$ is Cartier near $z$.  The canonical pullback linearization makes
the one-dimensional fiber $\bigl(\cO_X(mK_X)\bigr)_z$ a
$\C^*$-representation.  Denote its weight by $a_{Z,m}$, so for every
$0\ne v\in\bigl(\cO_X(mK_X)\bigr)_z$,
\[
 \lambda_t^{K,m}\bigl(\rho_t^*v\bigr)=t^{a_{Z,m}}v.
\]

\begin{lemma}
\label{lem:canonical-weight}
One has $c_Z\in\Q$ and
\[
        a_{Z,m}=-mc_Z .
\]
\end{lemma}

\begin{proof}
The local eigensection construction in the proof of
Lemma~\ref{lem:constancy}, together with \eqref{eq:slopeandmomentmap},
applies at an arbitrary fixed point, including a singular one.  It gives
$\ell>0$ such that
$-\ell K_X$ is Cartier near $z$, together with a nonvanishing local
eigengenerator $s$ of $\cO_X(-\ell K_X)$ with character $\beta_Z$, such that
\[
        c_Z=u(z)=\frac{\beta_Z}{\ell}.
\]
Thus $c_Z\in\Q$.  Fix
$0\ne v\in\bigl(\cO_X(mK_X)\bigr)_z$, and let $M$ be a
common multiple of $m$ and $\ell$.  The nonzero vectors
\(
        s(z)^{M/\ell}, (v^\vee)^{M/m}
\)
lie in the same one-dimensional fiber
$\bigl(\cO_X(-MK_X)\bigr)_z$, so their
characters agree.  Hence
\[
        \frac{M}{\ell}\beta_Z=-\frac{M}{m}a_{Z,m},
\]
and therefore $a_{Z,m}=-mc_Z$.
\end{proof}

\subsubsection{Critical descent and the graph identity}

We now apply the weight formula to the critical value $c$ fixed at the
beginning of the section. Lemma~\ref{lem:canonical-weight} gives $c\in\Q$.
Let
\[
 U_c:=X^{ss}(\eta,c),\qquad
 \varpi_c\colon U_c\longrightarrow X_c
\]
be the critical semistable locus and its analytic Hilbert quotient, and let
$\widetilde L_c:=\cO_{U_c}e$ carry the inverse-standard linearization fixed
above.  For any positive integer $N$ with $Nc\in\mathbb Z$ and
$NK_{U_c}$ Cartier, put
\[
 \mathcal N_{c,N}
 :=\cO_{U_c}(NK_{U_c})
   \otimes\widetilde L_c^{\otimes(-Nc)}.
\]

\begin{lemma}\label{lem:critical-descent}
The integer $N$ can be chosen so that there exists a unique
line bundle $A_{c,N}$ on $X_c$ and an equivariant isomorphism
\[
        \mathcal N_{c,N}\simeq\varpi_c^*A_{c,N}.
\]
For this choice one has
\begin{equation}\label{eq:critical-descent}
 \cO_{X_\pm}\!\left(N(K_{X_\pm}+D_\pm-cL_\pm)\right)
 \simeq \pi_\pm^*A_{c,N}.
\end{equation}
\end{lemma}

\begin{proof}
  Theorem~\ref{thm:regvalues}, Proposition~\ref{prop:birational}, and the
proofs of Proposition~\ref{prop:locallyalgebraic} and
Corollary~\ref{cor-birational}, together with
\cite[Theorem~3]{greb-projec}, allow us to
work algebraically: $U_c$ is quasi-projective, $\varpi_c$ is a good quotient
onto the projective variety $X_c$, and the algebraic maps $\pi_\pm$ satisfy
$\varpi_c|_{U_\pm}=\pi_\pm\circ\varpi_\pm$.  Since $U_c$ is klt and
quasi-projective, a finite cover by Cartier-index neighborhoods gives $q>0$
such that $qK_{U_c}$ is Cartier.

Put
\[
 V:=U_c\setminus U_c^{\C^*},\qquad
 I_V:=\{(t,x)\in\C^*\times V:t\cdot x=x\}.
\]
For every \(x\in V\), the fiber of the projection
\(
 I_V\longrightarrow V
\)
is precisely the stabilizer \(\operatorname{Stab}_{\C^*}(x)\), which
is finite.  Since \(V\) is of finite
type and is quasi-compact, the orders of its geometric fibers are
uniformly bounded by some integer \(B\). Therefore, by choosing $N$ sufficiently divisible, we obtain that every closed-orbit stabilizer acts trivially on
$\mathcal N_{c,N}$, and the descent theorem
\cite[\S4, Proposition~4.2, pp.~83--84]{KKV89} gives a unique line bundle
$A_{c,N}$ on $X_c$ with
\[
        \mathcal N_{c,N}\simeq\varpi_c^*A_{c,N}.
\]

It remains to identify the descent \(\mathcal N_{c,N}|_{U_\pm}\) to \(X_\pm\).  At a closed orbit in
$U_\pm$ with finite stabilizer $H$, Luna's  slice theorem
\cite{Lun73} gives a normal slice $S$ and strongly \'etale maps
\[
        \C^*\times_H S\longrightarrow U_\pm,
        \qquad S/H\longrightarrow X_\pm.
\]
Since the $\C^*$-action on $X$ is effective, the induced $H$-action on the
Luna slice $S$ is faithful (an element of its kernel would fix the nonempty
open image of $\C^*\times_H S$).  The invariant form $dt/t$ removes
the canonical factor in the orbit direction, so it remains to consider the
finite quotient $$\tau\colon S\to S/H.$$

For a prime divisor $E\subset S$ above $B\subset S/H$, strong \'etaleness
identifies the ramification index of $\tau$ along $E$ with the generic
stabilizer order $e_B$ defining the quotient boundary.  Thus, after
localizing at the generic points of $E$ and $B$, finite Riemann--Hurwitz gives
\[
 \tau^*B=e_BE,\qquad R_\tau=(e_B-1)E,
 \qquad
 K_S=\tau^*\bigl(K_{S/H}+(1-e_B^{-1})B\bigr).
\]
Summing over the branch divisors yields
$$K_S=\tau^*(K_{S/H}+D_{S/H}),\qquad D_{S/H}:=\sum_B(1-e_B^{-1})B$$ where
$D_{S/H}$ is the pullback of $D_\pm$.
Since $N$ is divisible by every closed-orbit stabilizer order on $U_\pm$,
 $NK_{U_\pm}$ descends to a line bundle on the quotient; the formula above and reflexive extension
identify its descent with
\[
        \cO_{X_\pm}\!\left(N(K_{X_\pm}+D_\pm)\right).
\]
Since $Nc$ has the same divisibility, the definition of $L_\pm$ identifies
the other descent as $\cO_{X_\pm}(-NcL_\pm)$.  Hence the descent of
$\mathcal N_{c,N}|_{U_\pm}$ is
\(N(K_{X_\pm}+D_\pm-cL_\pm).
\)
On the other hand, restricting the descent from $X_c$ gives
$\mathcal N_{c,N}|_{U_\pm}\simeq
\varpi_\pm^*\pi_\pm^*A_{c,N}$.  Uniqueness of descent proves
\eqref{eq:critical-descent}.
\end{proof}

Let $W$ and $p_\pm$ be as in diagram~\eqref{diag:wall-crossing}. As a consequence of \eqref{eq:critical-descent}, we obtain the following.

\begin{prop}[Graph wall identity]\label{prop:graph-wall}
We have the following equality as $\Q$-Cartier divisors
\begin{equation}\label{eq:graph-wall}
        p_-^*(K_{X_-}+D_- - cL_-)
        =
        p_+^*(K_{X_+}+D_+ - cL_+).
\end{equation}

\end{prop}

\subsubsection{The graded character jump and effectivity}
Let $\Omega$ be the common dense quotient open in $X_-$ and $X_+$, and
identify the restrictions of $L_-$ and $L_+$ there through the
common equivariant inverse-character line on the stable locus. Choose an integer
$m>0$ for which $mL_-$ and $mL_+$ are Cartier, and fix one nonzero rational
frame $\ell^{(m)}$ of their common restriction to $\Omega$.  Denote its
rational extensions to $X_-$ and $X_+$ by $\ell_-^{(m)}$ and
$\ell_+^{(m)}$, respectively.  For every prime divisor $F\subset W$, define
\begin{equation}\label{eq:oriented-character-valuation}
 \operatorname{coeff}_F(E_c)
 :=\frac1m\left(
   \operatorname{ord}_F(p_-^*\ell_-^{(m)})
  -\operatorname{ord}_F(p_+^*\ell_+^{(m)})\right).
\end{equation}
Equivalently, $mE_c$ is the divisor of the rational section
\[
 p_-^*\ell_-^{(m)}\otimes
 (p_+^*\ell_+^{(m)})^{-1}.
\]
Replacing $\ell^{(m)}$ by $h\ell^{(m)}$, for
$h\in K(X_c)^\times=K(\Omega)^\times$, adds
$\operatorname{ord}_F(g^*h)$ to both terms in
\eqref{eq:oriented-character-valuation}; passing to a tensor power has the
same cancellation.  Thus the divisor is independent of these choices and gives 
\begin{equation}\label{eq:oriented-character-divisor}
             E_c\sim_{\mathbb Q}p_-^*L_--p_+^*L_+.
\end{equation}

\begin{lemma}[Affine description of the two adjacent quotients]
\label{lem:affine-wall-git}
Let $V=\operatorname{Spec}R\subset X_c$ be affine and let
\[
        \widetilde V:=\varpi_c^{-1}(V)=\operatorname{Spec}A,
        \qquad A=\bigoplus_{j\in\mathbb Z}A_j,\qquad A_0=R,
\]
where the grading is the geometric-pullback grading
$\rho_\lambda^*f=\lambda^jf$ on $A_j$.  For every sufficiently divisible
$N>0$, if
\[
        S_\pm:=\bigoplus_{k\ge0}A_{\pm kN},
\]
then the two base-changed chamber quotients are
\begin{equation}\label{eq:affine-wall-proj}
\begin{aligned}
 X_-\times_{X_c}V&=\operatorname{Proj}_R S_-,\\
 X_+\times_{X_c}V&=\operatorname{Proj}_R S_+.
\end{aligned}
\end{equation}
After enlarging $N$ to clear the finite stabilizer characters, the
tautological bundles satisfy
\[
        \cO_{\operatorname{Proj}S_-}(1)=-NL_-,
        \qquad
        \cO_{\operatorname{Proj}S_+}(1)=NL_+.
\]
\end{lemma}

\begin{proof}
 By \eqref{eq:real-subgroup-direction} and the chamber description in
Proposition~\ref{prop:birational}, the affine limit
criterion \cite[\S1, Lemma~1.2 and Proposition~1.3]{Tha96} gives
\[
 U_-\cap\widetilde V
 =\bigcup_{k>0}\ \bigcup_{f\in A_{-k}}D(f),
 \qquad
 U_+\cap\widetilde V
 =\bigcup_{k>0}\ \bigcup_{f\in A_{k}}D(f).
\]
The corresponding affine GIT quotients over $V$ are
$\operatorname{Proj}_R S_-$ and $\operatorname{Proj}_R S_+$
\cite[\S1, p.~694]{Tha96}.
The final paragraph of the proof of Theorem~\ref{thm:regvalues}, applying
\cite[Proposition~10.1]{greb-projec}, makes
$\varpi_\pm\colon U_\pm\to X_\pm$ algebraic good quotients, while
the proof of Corollary~\ref{cor-birational} shows that $\pi_\pm$ are
algebraic.  Since
$\varpi_c|_{U_\pm}=\pi_\pm\circ\varpi_\pm$,
\[
 U_\pm\cap\widetilde V
 =\varpi_\pm^{-1}\bigl(\pi_\pm^{-1}(V)\bigr).
\]
Restriction therefore gives good quotients onto
$\pi_\pm^{-1}(V)=X_\pm\times_{X_c}V$.  By uniqueness of good quotients,
these coincide with the two affine GIT quotients above, proving
\eqref{eq:affine-wall-proj}.

The two half-algebras $\bigoplus_{k\ge0}A_{-k}$ and
$\bigoplus_{k\ge0}A_k$ are finitely generated over $R$.
Thus, for every sufficiently divisible $N$, the rings $S_-$ and $S_+$ are
generated in degree one, and passing to a further common Veronese does not
change their relative Proj.  Degree one then comes
from the linearizations
$\widetilde L_c^{\otimes(-N)}$ and $\widetilde L_c^{\otimes N}$,
respectively.  Their descents are $-NL_-$ and $NL_+$, so they are precisely
the two tautological bundles asserted in the statement.
\end{proof}

\begin{lemma}
\label{lem:divisorial-character-jump}
Let $E_c$ be the divisor defined
in \eqref{eq:oriented-character-valuation}.  If $Q\subset X_c$ is a prime
divisor whose
generic point is contained in the image of the fixed locus,
then
\[
        \operatorname{coeff}_Q(g_*E_c)>0.
\]
\end{lemma}

\begin{proof}
Let $\eta_Q$ be the generic point of $Q$, and choose an irreducible component
$Z$ of the fixed locus with $\eta_Q\in\varpi_c(Z)$.  By
Proposition~\ref{prop:birational}(ii)--(iii), its
adjacent attracting and repelling quotient strata dominate $\varpi_c(Z)$.
Since the corresponding source strata have dimension less than $n$ and their
$\C^*$-orbits are one-dimensional, $\dim\varpi_c(Z)\le n-2=\dim Q$; hence
$\varpi_c(Z)=Q$.  Thus the localized positive- and negative-weight
semi-invariant algebras are both nonzero.

As in
Lemma~\ref{lem:critical-descent}, we can work in the algebraic setting. 
Choose an affine open neighborhood $V=\operatorname{Spec}R\subset X_c$ of
$\eta_Q$.  Let
$\mathfrak p_Q\subset R$ be the height-one prime defining $Q\cap V$, and
replace $R$ by $R_{\mathfrak p_Q}$, retaining the notation $R$.  Then
$R=\mathcal O_{X_c,\eta_Q}$ is a DVR.
The base change of the critical quotient to $\operatorname{Spec}R$ has
affine total space, say $\operatorname{Spec}A$, and the good-quotient
identity gives $R=A_0$.  Thus the local wall is described by the normal graded domain
\[
        A=\bigoplus_{j\in\mathbb Z}A_j,\qquad R=A_0,
\]
where the grading is the geometric-pullback grading fixed above,
\[
 A_j:=\{a\in A\mid \rho_\lambda^*a
                   =a\circ\rho_\lambda=\lambda^j a\}.
\]
Lemma~\ref{lem:affine-wall-git}, followed by this localization, gives
\[
 X_-\times_{X_c}\operatorname{Spec}R
   =\operatorname{Proj}_R\!\bigoplus_{k\ge0}A_{-kN},
 \qquad
 X_+\times_{X_c}\operatorname{Spec}R
   =\operatorname{Proj}_R\!\bigoplus_{k\ge0}A_{kN}.
\]
Since the two half-algebras are finitely generated, we can choose $N>0$
sufficiently divisible so that
\[
        S_\pm=\bigoplus_{k\ge0}A_{\pm kN}
\]
are generated in degree one.  

Choose $0\ne h_-\in A_{-N}$ and set $I_+=h_-A_N\subset A_0=R$.
Because $S_+$ is generated in degree one,
$h_-^kA_{kN}=(h_-A_N)^k=I_+^k$.  Hence multiplication by $h_-^k$
identifies the $k$-th graded piece of $S_+$ with the $k$-th piece of the
Rees algebra, and
\[
        \operatorname{Proj}_R S_+
        \simeq\operatorname{Bl}_{I_+}(\operatorname{Spec}R).
\]
Since
$I_+$ is a nonzero ideal of the DVR $R$, it is principal.  Write $I_+=(s_+)$.  Since
$s_+=h_-f_+$ for some $f_+\in A_N$, cancellation in the domain $A$
gives $A_N=f_+R$.  Repeating the argument with the signs reversed gives
$A_{-N}=f_-R$.  Let $\mathfrak m_R$ be the maximal ideal of the DVR $R$; it is the
localized height-one prime corresponding to $Q$.  Since $Z$ is fixed, every
semi-invariant of nonzero weight restricts to zero on $Z$: for $z\in Z$, the
identity $f_\pm(z)=\lambda^{\pm N}f_\pm(z)$ for every $\lambda\in\C^*$ forces
$f_\pm(z)=0$.  Hence the invariant function $\gamma_Q=f_+f_-\in A_0=R$
vanishes on $Z$.  Since $\varpi_c(Z)=Q$, it vanishes along $Q$, and
therefore $\gamma_Q\in\mathfrak m_R$.  Finally $A$ is a domain, so
$\gamma_Q\ne0$; as $R$ is a DVR, this gives
$\operatorname{ord}_Q(\gamma_Q)>0$.

A pullback-weight-$N$ function is a section of
$NL$, so $f_+$ is a frame of $NL_+$ on the positive chamber.  The function
$f_-$ has pullback weight $-N$, so $f_-^{-1}$ is a frame of $NL_-$ on the
negative chamber.  Choose the common rational frame $\ell^{(N)}$ in
\eqref{eq:oriented-character-valuation} to be $f_+$.  Its upper extension
is the frame $f_+$, whereas its lower extension is
\[
              f_+=\gamma_Qf_-^{-1}.
\]
The upper order is zero and the lower order is
$\operatorname{ord}_Q(\gamma_Q)$.  Thus
\[
 N\operatorname{coeff}_Q(g_*E_c)
     =\operatorname{ord}_Q(\gamma_Q)>0,
\]
with the character convention fixed above.
\end{proof}

\begin{lemma}\label{lem:wall-effectivity}
The $\Q$-Cartier divisor $E_c = p_-^{\ast}L_- - p_+^{\ast}L_+$ defined
in \eqref{eq:oriented-character-valuation}
is effective.
\end{lemma}

\begin{proof}
With the compatible representatives and the divisor $G_c$ fixed in the
proof of Proposition~\ref{prop:graph-wall}, we have identities of
$\mathbb Q$-divisors
\[
 K_{X_\pm}+D_\pm-cL_\pm=\pi_\pm^*G_c.
\]
By Proposition~\ref{prop:regular-reduction}, we have ample $\mathbb Q$-divisors on $X_{\pm}$ respectively:
\[
\begin{aligned}
 -(K_{X_-}+D_-)+(c-\varepsilon)L_-
      =-\pi_-^*G_c-\varepsilon L_-,\\
 -(K_{X_+}+D_+)+(c+\varepsilon)L_+
      =-\pi_+^*G_c+\varepsilon L_+.
\end{aligned}
\]
   Therefore $-L_-$ is
$\pi_-$-ample and $L_+$ is $\pi_+$-ample.  Consequently
\[
        -E_c=-p_-^*L_-+p_+^*L_+
\]
is $g$-nef.
We next check the non-exceptional coefficients.  Let $Q$ be a prime
divisor on $X_c$.  Away from the image of the critical locus, both
$\pi_\pm$ are isomorphisms over the generic point of $Q$, and the coefficient of $g_*E_c$ along $Q$
is zero. Lemma~\ref{lem:divisorial-character-jump} show that the coefficient of $g_*E_c$ along $Q$ is positive $Q$ is the image of a component of the critical locus and therefore $g_*E_c$ is effective.  We have shown that $-E_c$ is $g$-nef and
$g_*E_c\ge0$.  Since $W$ and $X_c$ are normal and $g$ is proper birational,
the negativity lemma \cite[Lemma~3.39]{KM98} gives $E_c\ge0$. 
\end{proof}

\subsection{Uniform bounds and compactness of critical locus}\label{sec:uniform-bounds}
The identity~\eqref{eq:graph-wall} and $E_c\ge0$ propagate fixed-index complements across singular positive walls. Moreover we obtain that the coefficients of the boundary divisors
$B_r=D_r+\Delta_r$ are in one fixed finite rational set and this condition allows us to use the boundedness of Hacon--Xu \cite{HX15}.   Then we prove a uniform lower bound on the normalized volumes, which implies
bounded stabilizers and Cartier indices. Finally we use the relative length theorem \cite{BCHM10} to obtain the uniform upper bound for critical values,
and a compact critical locus.

\subsubsection{The first complement}

\begin{lemma}
\label{lem:zero-point}
There is a point
\[
        x\in X^{\rm reg}\setminus\operatorname{Fix}(S^1)
        \quad\text{such that}\quad
        u(x)=0,\qquad du(x)\ne0.
\]
Its $S^1$-isotropy group is finite.
\end{lemma}

\begin{proof}
Lemma~\ref{lem:weighted-centering} then implies that $u$ takes both positive
and negative values.  Since
$X^{\rm reg}\setminus\operatorname{Fix}(S^1)$ is dense, both sign sets meet
this locus.  It is also path-connected: it is the set $X_{\max}(\eta)$
appearing in the proof of Proposition~\ref{prop:connected-levels}.  A path joining the
two sign points therefore meets $\{u=0\}$ at a point $x$ with
$\eta(x)\ne0$. 
The stabilizer of $x$ is a proper closed subgroup of $S^1$, and is therefore
finite.
\end{proof}

\begin{lemma}\label{lem:zero-patch}
There are $\varepsilon>0$,   and a
number $v>0$ such that for regular $|r|<\varepsilon$,
\[
        A_r^{\,n-1}\ge v.
\]
\end{lemma}

\begin{proof}
Let $x$ be a point given by 
Lemma~\ref{lem:zero-point}.  After shrinking an invariant neighborhood of
the compact orbit $S^1x$, every nearby level is locally free, and its
orbifold quotient is identified with a fixed transverse slice.  Choose a
relatively compact domain with smooth boundary in this slice and let
$K_r$ be its closed image in
$ X_r$.  The reduced forms $\omega_r$ depend smoothly
on $r$ on this family of compact orbifold domains.  Since $K_0$ has
nonempty interior,
\[
        \int_{K_0}\omega_0^{\,n-1}>0.
\]
Continuity gives the asserted lower bound for $|r|<\epsilon$.
The  current $\omega_r$ represents $2\pi c_1(A_r)$ and has
continuous local potentials.  By  pulling back to a resolution, we know that its Bedford--Taylor product is defined
and gives the top intersection of $A_r$.
It follows that
\[
 A_r^{\,n-1}
 \ge(2\pi)^{-(n-1)}\int_{K_r}\omega_r^{\,n-1}.
\]
Hence the desired result
follows.
\end{proof}

The critical values are discrete by Lemma~\ref{lem:constancy}.  They are also
closed, since $u$ is proper on the closed fixed locus.  After decreasing
$\varepsilon$, there is no critical value in $(0,\varepsilon)$.  The
reductions $X_r$ for $r\in (0,\epsilon)$ are naturally identified.

\begin{prop}
\label{prop:initial-complement}
On the first positive reduction there are an effective big
$\Q$-divisor $\Delta_r$ and a fixed number $\delta>0$ such that
\[
 K_{X_r}+D_r+\Delta_r\sim_{\Q}0,
 \qquad
 (X_r,D_r+\Delta_r)\ \text{is }\delta\text{-lc}.
\]
\end{prop}

\begin{proof}We write
\[
        (X_r,D_r,L_r)=(X_+,D,L),\qquad
        H:=-(K_{X_+}+D).
\]
On the fixed projective variety $X_+$, Proposition \ref{prop:regular-reduction} implies
\[
        A_r=-(K_{X_r}+D_r)+rL_r=H+rL
\]
is ample.  Letting $r\downarrow0$ proves that $H$ is nef.  By
Lemma~\ref{lem:zero-patch} and continuity of top intersection,
\[
        H^{n-1}
        =\lim_{r\downarrow0}A_r^{n-1}\ge v>0.
\]
Therefore the nef divisor $H$ is big.

The pair $(X_+,D)$ is klt.    By the klt base-point-free theorem,
after replacing $m$ by a sufficiently divisible multiple,
$|mH|$ is base-point-free.  Let
$G\in|mH|$ be general and set
\[
        \Delta:=\frac1mG.
\]
Then $\Delta$ is effective and big, and
\[
        K_{X_+}+D+\Delta\sim_{\Q}0.
\]
For a general $G$, Bertini theorem implies that
$(X_+,D+\Delta)$ is klt.  Choose $N>0$ such that
$N(K_{X_+}+D+\Delta)$ is Cartier.  Its log discrepancies belong to
$N^{-1}\mathbb Z_{>0}$, so the pair is $N^{-1}$-lc.    
\end{proof}

\subsubsection{Fixed-index boundaries across positive walls}\label{sec:complements}

\begin{lemma}[Fixed-index boundaries in every positive chamber]
\label{lem:complements}
There exist an integer $m>0$ and a number $\delta>0$, independent of the
positive regular value $r$, and an effective big $\Q$-divisor
$\Delta_r$ on $X_r$ such that
\begin{equation*}
\begin{gathered}
        (X_r,D_r+\Delta_r)\ \text{is }\delta\text{-lc}, \qquad 
        m(K_{X_r}+D_r+\Delta_r)\sim0 .
\end{gathered}
\end{equation*}
Moreover, the coefficients of $D_r+\Delta_r$ belong to one finite set
independent of $r$.
\end{lemma}

\begin{proof}
Fix a nonzero rational top form
$\kappa$ on the common function field  and, on every model, use the canonical divisor
\[
        K_{X_r}:=\operatorname{div}_{X_r}(\kappa).
\]
Start with the complement on the first positive reduction from
Proposition~\ref{prop:initial-complement}.  Choose $m>0$ sufficiently
divisible and $\varphi\in K^\times$ so that
\[
        m(K_{X_-}+D_-+\Delta_-)
        =\operatorname{div}_{X_-}(\varphi).
\]

Let $c>0$ be such a wall and suppose that the complement has already
been constructed on the chamber below it.  On the chamber above it,
define
\[
        \Delta_+:=\frac1m\operatorname{div}_{X_+}(\varphi)
                   -(K_{X_+}+D_+).
\]
This divisor is $\Q$-Cartier because $K_{X_+}+D_+$ is $\Q$-Cartier by
Proposition~\ref{prop:regular-reduction}.
On the normalized graph, put
$E_c=p_-^*L_--p_+^*L_+$.  The wall identity
\eqref{eq:graph-wall} gives
\[
 p_+^*(K_{X_+}+D_+)
 =
 p_-^*(K_{X_-}+D_-)-cE_c.
\]
The pullbacks of the principal divisor of $\varphi$ agree.  Therefore
\begin{equation}\label{eq:complement-transport}
        p_+^*\Delta_+
        =
        p_-^*\Delta_-+cE_c.
\end{equation}
Pushing \eqref{eq:complement-transport} forward by
$p_+$ proves $\Delta_+\ge0$.  Moreover $p_-^*\Delta_-$ is big and adding
$cE_c\ge0$ preserves bigness, so $p_+^*\Delta_+$ is big.  Bigness is
invariant under birational pullback, and therefore $\Delta_+$ is big.

Since the same rational function $\varphi$ is used on both models,
\begin{equation}\label{eq:complement-crepant}
 p_+^*(K_{X_+}+D_++\Delta_+)
 =\frac1m\operatorname{div}_W(\varphi)
 =p_-^*(K_{X_-}+D_-+\Delta_-).
\end{equation}
We obtain all their
log discrepancies agree by \cite[Chapter 2, Lemma 2.30]{KM98}
and hence the initial value of $\delta$ works in every
positive chamber.  By construction,
\[
        m(K_{X_+}+D_++\Delta_+)
        =\operatorname{div}_{X_+}(\varphi),
\]
so the complement index remains equal to $m$.

Finally set $B_r=D_r+\Delta_r$.  The fixed canonical divisors
$K_{X_r}=\operatorname{div}_{X_r}(\kappa)$ have integral coefficients.  The
equality $m(K_{X_r}+B_r)=\operatorname{div}_{X_r}(\varphi)$ then shows that
every coefficient of $B_r$ lies in $m^{-1}\mathbb Z$.  Effectivity and
$\delta$-log canonicity give
\[
        0\le\operatorname{coeff}(B_r)\le1-\delta.
\]
Thus all coefficients lie in the finite set
\[
        m^{-1}\mathbb Z\cap[0,1-\delta],
\]
independently of $r$.
\end{proof}

\subsubsection{Bounded reductions and a uniform volume gap}

Sun--Zhang apply Birkar's boundedness
\cite[Corollary~1.4]{Birkar2021} to bound the underlying varieties $X_r$ obtained from K\"ahler reduction.  Here the
fixed finite rational coefficient set allows Hacon--Xu's Theorem~1.3 to bound
the pairs themselves.  

\begin{lemma}[Uniform reduction-volume gap]
\label{lem:bounded-family-volume-gap}
There is a number $\nu>0$ such that
\[
        \widehat{\Vol}(x,X_r,D_r)\ge\nu
\]
for every positive regular value $r$ and every closed point $x\in X_r$.
\end{lemma}

\begin{proof}
Put $B_r:=D_r+\Delta_r$.  By Lemma~\ref{lem:complements}, the
$(X_r,B_r)$ are projective klt pairs of fixed dimension, $B_r$ is big,
$K_{X_r}+B_r\sim_\Q0$, and the coefficients of $B_r$ belong to one fixed
finite set.  Hence \cite[Theorem~1.3]{HX15} shows that these pairs form a
bounded family.

Taking $\epsilon=\delta/2$, generic flatness,
\cite[Proposition~2.4]{HX15}, and Noetherian induction place them in
finitely many $\Q$-Gorenstein flat families of klt pairs.  Base-changing
each family by its total space and using the diagonal section gives pointed
families containing every $x\in(X_r,B_r)$.  By
\cite[Proposition~8 and Theorem~28]{BL21}, normalized volume is Zariski
lower semicontinuous on each parameter space which gives a uniform constant $\nu>0$ such that
\[
        \widehat{\Vol}(x,X_r,B_r)\ge\nu .
\]
Finally, since $B_r-D_r=\Delta_r\geq 0$,
\(
 A_{X_r,D_r}(v)
 =A_{X_r,B_r}(v)+v(\Delta_r)\ge A_{X_r,B_r}(v).
\)
Therefore
\[
 \widehat{\Vol}(x,X_r,D_r)
 \ge\widehat{\Vol}(x,X_r,B_r)\ge\nu .
\]
\end{proof}

\subsubsection{Uniform on stabilizers and Cartier index}\label{sec:normalized-volume}

Let $\nu>0$ be the constant in
Lemma~\ref{lem:bounded-family-volume-gap}.  We bound the stabilizer order
$|H|$ by comparing this uniform lower bound for
$\widehat{\Vol}(x,X_r,D_r)$ with the universal upper bound on a transverse
slice via the finite degree formula.

Let $p\in X^s(\eta,r)$ be a stable point, let $x\in X_r$ be its image, and
let
\[
        H:=\operatorname{Stab}_{\C^*}(p).
\]
After restricting over an affine neighborhood of $x$, Luna's slice theorem \cite{Lun73} gives a transverse slice $(S,s)$ and a
strongly \'etale chart
\[
        \C^*\times_H S\longrightarrow X^s(\eta,r).
\]
Since $X^s(\eta,r)$ is an
open subspace of the klt space $X$, preservation of klt singularities under
\'etale pullback along the two \'etale maps
$\C^*\times S\to\C^*\times_HS\to X^s(\eta,r)$ shows that
$\C^*\times S$ is klt; smooth-factor invariance then shows that $S$ is klt.
On the quotient, we have
\[
 (s\in(S,0))\xrightarrow{\ \pi\ }
 (\bar s\in(S/H,D_H))
 \xrightarrow[\text{\'etale}]{\ \tau\ }
 (x\in(X_r,D_r)).
\]
Let $\pi\colon S\to S/H$ be the quotient map, and characterize $D_H$ by
\[
K_S=\pi^*(K_{S/H}+D_H).
\]
The strongly \'etale Cartesian square and the ramification formula give
\[
D_H=\tau^*D_r.
\]

\begin{lemma}
\label{lem:slice-quotient-volume}
One has
\begin{equation}\label{eq:finite-quotient-exact}
 |H|\,\widehat{\Vol}(x,X_r,D_r)
 =\widehat{\Vol}(s,S)
 \le (n-1)^{n-1}.
\end{equation}
\end{lemma}

\begin{proof}
Since $\pi$ is finite and $\tau$ is \'etale, both $S$ and $X_r$ have
dimension $n-1$. Since $H$ acts faithfully on the slice, a general point of $S$ has trivial stabilizer. Therefore the quotient
$\pi\colon S\to S/H$ has Galois degree $|H|$.   The source is klt by the
slice argument above, while the target is klt because $\tau$ is \'etale and
$(X_r,D_r)$ is klt.  The identity
$K_S=\pi^*(K_{S/H}+D_H)$ makes this morphism log crepant; hence
\cite[Theorem~1.3]{XZ21} gives
\[
 \widehat{\Vol}(s,S)
 =|H|\,\widehat{\Vol}(\bar s,S/H,D_H).
\]

Finally, $\tau$ and $D_H=\tau^*D_r$ identify the completed pointed log
germs at $\bar s$ and $x$, so
\[
 \widehat{\Vol}(\bar s,S/H,D_H)
 =\widehat{\Vol}(x,X_r,D_r)
\]
by \cite[Theorem~5]{BL21} and \cite[Proposition~2.11]{dFEM11}.  The
universal bound \cite[Theorem~A.4]{LX19} gives
$\widehat{\Vol}(s,S)\le(n-1)^{n-1}$, and the result follows.
\end{proof}

In the smooth setting of Sun--Zhang \cite[Section~3]{SunZhang}, the normalized volume of the underlying coarse variety is
considered without a boundary divisor, and discrepancy bounds are then used to
control the contribution from pseudoreflections. Here, instead, we retain the
boundary divisor \(D_r\) and apply the finite-degree formula directly, thereby
controlling the full stabilizer \(H\) in a single step.
\begin{lemma}
\label{lem:uniform-character-index}
There is an integer $M>0$ such that $ML_r$ is a Cartier divisor for
every positive regular value $r$.
\end{lemma}

\begin{proof}
Let $\nu>0$ be the constant in
Lemma~\ref{lem:bounded-family-volume-gap}, let $p$ be any stable point
over a positive regular reduction, let $x\in X_r$ be its image, and put
$H=\operatorname{Stab}_{\C^*}(p)$.  By
\eqref{eq:finite-quotient-exact},
\[
        |H|\,\nu
        \le |H|\,\widehat{\Vol}(x,X_r,D_r)
        \le (n-1)^{n-1}.
\]
Thus $|H|\le(n-1)^{n-1}/\nu$, and every stabilizer order is at most
\[
 B:=\max\left\{1,
      \left\lfloor\frac{(n-1)^{n-1}}{\nu}\right\rfloor\right\}.
\]
Set $M=\operatorname{lcm}\{1,\ldots,B\}$.

Every stabilizer acts trivially on the fiber of
$\widetilde L_r^{\otimes M}$.  By the descent theorem
\cite[\S4, Proposition~4.2, pp.~83--84]{KKV89}, this power descends to a line
bundle on $X_r$.  By the definition and uniqueness of $L_r$, its descent is
$ML_r$.
\end{proof}

\subsubsection{Bounded critical values and compactness}\label{sec:fixed-locus}

\begin{lemma}
\label{lem:isomorphic-chambers-type-I}
Let $c$ be an isolated critical value with both adjacent levels
nonempty.  If both adjacent modifications are isomorphisms, then $E_c$
has a nonzero positive coefficient.
\end{lemma}

\begin{proof}
This follows from Lemmas~\ref{lem:affine-wall-git}
and~\ref{lem:divisorial-character-jump}; we use the same notation as
there.  Let \(q\) be a closed point in the image of the fixed locus \(Z\) under the quotient map.
Restricting to an affine neighborhood of \(q\), we have a normal
graded domain
\[
A=\bigoplus_{j\in\mathbb Z}A_j,\qquad
R=A_0,\qquad
S_\pm=\bigoplus_{k\geq 0}A_{\pm kN},
\]
where, after passing to a common Veronese, both \(S_+\) and \(S_-\) are
generated in degree one.  By hypothesis, the two relative Proj morphisms
are isomorphisms.

Choose \(0\neq h_-\in A_{-N}\).  The Rees-algebra calculation in the proof
of Lemma~\ref{lem:divisorial-character-jump} identifies
\(\operatorname{Proj}_R S_+\) with the blowup of the ideal \(h_-A_N\).
Since by the assumption, the corresponding relative Proj morphism is an isomorphism, this
ideal is invertible.  Similarly, \(h_+A_{-N}\) is invertible for any
\(0\neq h_+\in A_N\).  After shrinking the affine neighborhood of \(q\),
we may therefore assume that both ideals are principal.  It follows that
\[
A_N=f_+R,\qquad A_{-N}=f_-R
\]
for some \(f_+\in A_N\) and \(f_-\in A_{-N}\).

As in the proof of Lemma~\ref{lem:divisorial-character-jump}, the product
\[
    \gamma_Z:=f_+f_-
\]
is a nonzero non-unit element of \(R\) and hence gives a nontrivial positive
contribution to \(E_c\). Consequently, \(E_c\neq 0\), which proves the claim.
\end{proof}

Let $M$ be the integer in Lemma~\ref{lem:uniform-character-index}.

\begin{prop}
\label{prop:bounded-critical-values}
Every positive critical value satisfies
\[
        c\le 2(n-1)M.
\]
Therefore, the zero set of $\eta$ is compact.
\end{prop}

\begin{proof}
Fix a positive critical value $c$ and the notation
\[
 \pi_\pm\colon X_\pm\longrightarrow X_c,\qquad
 E_c=p_-^*L_--p_+^*L_+ .
\]
Since $ML_\pm$ are Cartier,
$ME_c$ is an integral Cartier divisor.  Lemma \ref{lem:wall-effectivity} makes it effective,
so every nonzero coefficient $a$ of $E_c$ satisfies $a\ge1/M$. 

First suppose that $\pi_-$ is not an isomorphism.  Over $X_c$ one has
\[
        K_{X_-}+D_-\equiv_{\pi_-}cL_- ,
\]
and $-L_-$ is $\pi_-$-ample.  Hence
$-(K_{X_-}+D_-)$ is $\pi_-$-ample.  The relative length theorem
\cite[Theorem~3.8.1]{BCHM10}, applied to the klt pair $(X_-,D_-)$ over
$X_c$, gives a contracted rational curve $C$ with
\[
        0<-(K_{X_-}+D_-)\cdot C\le2(n-1).
\]
 Since $ML_-$ is
Cartier and $-L_-$ is relatively ample,
\[
                   -ML_-\cdot C\in\mathbb Z_{>0},
 \qquad            -L_-\cdot C\ge\frac1M.
\]
Intersecting the critical descent identity with $C$ gives
\[
        c=\frac{-(K_{X_-}+D_-)\cdot C}{-L_-\cdot C}
        \le2(n-1)M .
\]

It remains to assume that $\pi_-$ is an isomorphism.  Identify
$X_-=X_c$ via $\pi_-$. Then we can take
normalized graph $W$ to be $X_+$
and therefore
\[
        E_c=\pi_+^*L_--L_+ .
\]
The proof of Lemma~\ref{lem:wall-effectivity} shows that $L_+$ is
$\pi_+$-ample, so $-E_c=L_+-\pi_+^*L_-$ is also $\pi_+$-ample.   If $\pi_+$ is not an isomorphism, then  $E_c\ne0$.  Choosing a prime component $F$ of $E_c$, we have
\[
        \operatorname{coeff}_F(E_c)\geq \frac{1}{M}.
\]
The wall identity \eqref{eq:graph-wall} becomes
\[
 K_{X_+}+D_+
 =\pi_+^*(K_{X_-}+D_-)-cE_c .
\]
Since $D_+$ is a boundary,
$\operatorname{coeff}_F(D_+)\ge0$.  Since $(X_-,D_-)$ is klt, the
log discrepancy $A_F:=A(F;X_-,D_-)$ is positive.  Using compatible
canonical divisors, the definition of log discrepancy and the wall identity
give
\[
\begin{aligned}
 A_F
 &=1+\operatorname{coeff}_F\!\left(
        K_{X_+}-\pi_+^*(K_{X_-}+D_-)\right)\\
 &=1-\operatorname{coeff}_F(D_+)-c\cdot\operatorname{coeff}_F(E_c).
\end{aligned}
\]
Since $\operatorname{coeff}_F(E_c)\geq \frac{1}{M}$, $A_F>0$, and $\operatorname{coeff}_F(D_+)\ge0$, it follows that
$c<M$.

Finally suppose that both modifications are isomorphisms.
Lemma~\ref{lem:isomorphic-chambers-type-I} gives $E_c\ne0$.  On
the common normal variety choose the same canonical divisor on both sides and then the wall identity becomes
\[
        D_--D_+=cE_c.
\]
While both boundary coefficients belong to
$[0,1)$.  Hence
\[
     c\cdot \frac{1}{M}   \leq c\cdot\operatorname{coeff}_F(E_c)
        =\operatorname{coeff}_F(D_-)
         -\operatorname{coeff}_F(D_+)<1,
\]
and $c<M$. 

Because $\eta=J\nabla u$, the set $Z(\eta)$ is the critical locus of $u$.
 Since $u$ is proper, the upper bound on the critical values implies $Z(\eta)$ is compact.
\end{proof}

\subsection{Construction of the Fano fibration}\label{sec:fano-construction}

Following the
strategy of \cite{SunZhang}, we first construct the morphism outside a compact set, extend it to a global proper map,
 globalize a fixed anticanonical embedding, and verify its algebraic structure using the $\mathbb C^*$-action.

\begin{lemma}[Stabilized quotient and its punctured cone]
\label{lem:stabilized-cone}
There exist a regular value $r_*$, a $\C^*$-invariant Zariski-open set
$U\subset X$, a normal projective klt pair $(M,D)$, a semi-ample
$\Q$-line bundle $L_0$ on $M$, and a proper surjective algebraic
equivariant map
\[
        \pi_U\colon U\longrightarrow Y'\setminus\{o\},
\]
where $Y'$ is a normal affine cone with vertex $o$.  The set $U$ and the
quotient $(M,D,L_0)$ are the stable set and polarized quotient at every
regular value $r\ge r_*$.
\end{lemma}

\begin{proof}
By Proposition~\ref{prop:bounded-critical-values}, the critical values are
bounded above; choose a regular $r_*$ larger than all of them.  Then the stable set and its quotient are independent of
$r\ge r_*$.  We denote them by
\[
        U=X^s(\eta,r),\qquad (M,D,L_0).
\]
By Theorem~\ref{thm:regvalues} and the final paragraph of its proof, $U$ is a
$\C^*$-invariant Zariski-open quasi-projective variety, $U\to M$ is an
algebraic good quotient, and $M$ is normal and projective.  The pair $(M,D)$
is klt and $-(K_M+D)+rL_0$ is ample by Proposition~\ref{prop:regular-reduction}.

Dividing by $r$ and letting $r\to\infty$ shows that $L_0$ is nef, while
$aL_0-(K_M+D)$ is ample for $a\gg0$.  Applied to a Cartier multiple of
$L_0$, the base-point-free theorem
\cite[Theorem~3.3]{KM98} makes $L_0$ semi-ample.
Choose $k>0$, divisible by the uniform index of
Lemma~\ref{lem:uniform-character-index}, such that $L:=kL_0$ is a
base-point-free Cartier divisor.  Let
\[
        p\colon M\longrightarrow M',\qquad p^*L'=L,
\]
be its semi-ample morphism with connected fibers.  After replacing $k$ by a
multiple, we may assume that $L'$ is very ample and projectively normal.  Set
\[
        Y':=\Spec\bigoplus_{j\ge0}H^0(M',jL')
\]
and denote its vertex by $o$.

We can view $U\to M$ as the punctured total space of the dual
of the character line bundle.  Near an orbit with finite stabilizer
$H\subset\C^*$, the quotient map is strongly \'etale-locally modeled by
\[
   \C^*\times_HS\longrightarrow S/H,
   \qquad (t,s)\sim(th^{-1},h\cdot s).
\]
By the choice of $k$, the order of $H$ divides $k$, and hence $h^k=1$ for
every $h\in H$.  Thus the formula
\[
   \C^*\times_HS\longrightarrow (S/H)\times\C^*,
   \qquad [t,s]\longmapsto(\bar s,t^k),
\]
is well defined.  These intrinsic fiber-power maps glue to a finite
surjective morphism $U\to(L^\vee)^\times$. 
There is also a natural morphism from $(L^\vee)^\times $ to $\bigl((L')^\vee\bigr)^\times$ induced by $p$. Finally, the
very ampleness and projective normality of $L'$ identify
$\bigl((L')^\vee\bigr)^\times$ with $Y'\setminus\{o\}$ by contraction of
the zero section.  The composite is the required proper surjective map
$\pi_U$.
\end{proof}

Put $A:=X\setminus U$. The affine property of $Y'$ and the equivariant property of $\pi_U$ and the properness of the function $u$ implies that $\pi_U$ extends uniquely to a proper surjective
$\mathbb T$-equivariant holomorphic map
\[
        \pi'\colon X\longrightarrow Y'.
\]
Moreover, $A$ is compact analytic and
$\pi'^{-1}(o)=A$.

\begin{lemma}[Equivariant anticanonical algebraization]
\label{lem:end-algebraicity}
There are $m>0$ and $N>0$ such that $-mK_X$ is Cartier, and finitely many
$\mathbb T$-eigensections of $\cO_X(-mK_X)$ define a
$\mathbb T$-equivariant closed analytic immersion
\[
        (\pi',F)\colon X\hookrightarrow Y'\times\mathbb P^N.
\]
Moreover,
\[
        \cO_X(-mK_X)\simeq F^*\cO_{\mathbb P^N}(1),
\]
$X$ is the analytification of a normal quasi-projective variety,
$\pi'$ is projective algebraic, and $-K_X$ is $\pi'$-ample.
\end{lemma}

\begin{proof}
Choose relatively compact $\mathbb T$-invariant Stein neighborhoods
$B\Subset B'$ of $o$, stable under radial contraction, and put
$K':=\pi'^{-1}(\overline{B'})$.  Compactness of $K'$ and local
$\mathbb Q$-Gorensteinness give $m_0>0$ such that
$\mathcal L_0:=\cO_X(-m_0K_X)$ is invertible near $K'$.  Every orbit of
the radial $\C^*$-subgroup meets this neighborhood, so equivariance of the
reflexive canonical sheaf propagates invertibility and shows that
$\mathcal L_0$ is invertible on $X$.  By the soliton equation and
Theorem~\ref{thm:complexstructure}\ref{thm:complexstructure2},
$\mathcal L_0$ carries a continuous Hermitian metric whose curvature
is strictly positive.  Richberg regularization gives a smooth strictly
positive metric, so the dual line bundle $\mathcal L_0^\vee$ is relatively
negative in the sense of Grauert \cite[\S3, Satz~2]{Gra62}.  Equivalently,
$\mathcal L_0$ is relatively ample for the proper holomorphic map $\pi'$
\cite[Definition~5, pp.~859--860]{CDJ13}.

For a sufficiently divisible $r$, set
$m=rm_0$, $\mathcal L=\mathcal L_0^{\otimes r}$, and
$\mathcal E=\pi'_*\mathcal L$.  After shrinking $B'$, the relative
embedding theorem gives a closed immersion
\[
 X_{B'}\longrightarrow\mathbb P_{B'}(\mathcal E),
 \qquad X_{B'}:=\pi'^{-1}(B'),
\]
whose tautological bundle pulls back to $\mathcal L$
\cite[p.~860, discussion preceding Theorem~6]{CDJ13}.  Choosing a
$\mathbb T$-weight decomposition of the sections and shrinking to $B$ gives
$\mathbb P_B(\mathcal E)\hookrightarrow B\times\mathbb P^N$, and the
resulting eigensections $s_0,\ldots,s_N$ give a relative closed immersion on
$X_B$.

The $s_i$ extend uniquely to $X$ by their $\C^*$-characters, because every
radial orbit meets $X_B$.  Equivariance propagates base-point-freeness and
the local-immersion property.  If $\pi'(x)=\pi'(y)$, scale the common image
into $B$ and use separation on $X_B$ to obtain $x=y$.  Thus
$F=[s_0:\cdots:s_N]$, and
$(\pi',F)\colon X\to Y'\times\mathbb P^N$ is an injective local closed
immersion.  It is proper because $\pi'$ is
proper, hence is a closed analytic immersion, and evaluation gives
\begin{equation}\label{eq:end-anticanonical-evaluation}
        \mathcal L\simeq F^*\cO_{\mathbb P^N}(1).
\end{equation}

Let $Z$ be its analytic image, let
$q\colon Y'\times\mathbb P^N\to Y'$ be the projection, and let
$\mathcal I_Z$ be its ideal.  After shrinking a radial-invariant Stein
neighborhood $V$ of $o$, relative Serre generation and Grauert coherence
give $d\gg0$ such that
$q_*\mathcal I_Z(d)$ is coherent and
$q^*q_*\mathcal I_Z(d)\to\mathcal I_Z(d)$ is surjective over $q^{-1}(V)$
\cite[Chapter~III, Section~4.1 and Chapter~V, Section~4.8]{GPR94}.
In the local ring of the
affine cone, whose Reeb weights are positive, an invariant coherent submodule is generated by
homogeneous elements: decompose local sections into their weighted Taylor series, take the
homogeneous pieces by Fourier projection, and then use Noetherianity. Thus \[
        G_\nu(y,w)=\sum_{|\alpha|=d}f_{\nu,\alpha}(y)w^\alpha
\]
generating $\mathcal I_Z(d)$ over $q^{-1}(V)$.
Each $f_{\nu,\alpha}$ is a homogeneous holomorphic germ at $o$ and hence regular functions.
Thus the $G_\nu$ extend to algebraic eigen-equations on
$Y'\times\mathbb P^N$.  Let $\mathcal J$ be the algebraic ideal they generate
and put $W:=V(\mathcal J)$.  On $q^{-1}(V)$ one has
$\mathcal J^{\rm an}=\mathcal I_Z$; both ideals are radial-invariant, and
every radial orbit meets $q^{-1}(V)$, so equality holds globally.  Hence
$W^{\rm an}=Z$ globally.  Analytification detects normality
\cite[Expos\'e~XII, Proposition~2.1(vi)]{SGA1}, so $W$ is normal.  Since
the $G_\nu$ are eigensections for the diagonal algebraic torus action,
the ideal $\mathcal J$ they generate is invariant; hence
$W=V(\mathcal J)$ is invariant, and
\[
        W\subset Y'\times\mathbb P^N
\]
makes $W$ quasi-projective and $W\to Y'$ projective.
Identifying $X$ with $W^{\rm an}$, equation~\eqref{eq:end-anticanonical-evaluation}
shows that $-K_X$ is $\pi'$-ample.
\end{proof}

\begin{proof}[Proof of Theorem~\ref{thm:pffibration}]
Lemmas~\ref{lem:stabilized-cone}--\ref{lem:end-algebraicity} give a
normal quasi-projective algebraic structure on $X$, a projective morphism
$\pi'\colon X\to Y'$, and a $\pi'$-ample anticanonical class $-K_X$. Let
\[
        X\xrightarrow{\pi}Y\xrightarrow{\nu}Y'
\]
be the $\mathbb T$-equivariant algebraic Stein factorization of $\pi'$. Then $\pi$ is projective with connected fibers \cite[Corollary III.11.5]{Har77}.
 Moreover since $Y'$ is affine, the
finiteness of $\nu$ implies that $Y$ is affine
\cite[Chapter~II, Exercise~3.4]{Har77}.  A
Cartier multiple of $-K_X$ is $\pi'$-ample and therefore $\pi$-ample by
\cite[Proposition~4.6.13(v), pp.~91--92]{EGAII}.  Analytic and
algebraic klt singularities agree over $\C$
\cite[Proposition~1.10.3(3)]{Kaw24}.  Hence
$\pi\colon X\to Y$ is a Fano fibration.  The $\mathbb T$-action descends
through the canonical Stein factorization, so all maps are equivariant.

Put $R=\cO(Y)$.  The maximum principle for holomorphic functions gives $R_0=\C$ and
$\langle\alpha,\xi\rangle>0$ for every nonzero weight occurring in $R$.  Therefore
$(\pi\colon X\to Y,\xi)$ is a polarized Fano fibration.
\end{proof}

\bibliographystyle{amsalpha}
\bibliography{references}

\end{document}